\documentclass[a4,11pt,reqno]{amsart}%
\usepackage{amsfonts,textcomp,amssymb,amsmath,amsthm}
\usepackage{hyperref}
\usepackage{color,ulem}
\usepackage{fullpage}
\usepackage{graphicx}%
\usepackage[active]{srcltx}
\usepackage{mathabx}
\usepackage{dsfont,mathrsfs,stmaryrd,wasysym,amsbsy}

\def \bw{w_{0}}
\def \bmu{\boldsymbol \mu}
\def \balpha{\boldsymbol \alpha}

\def \R{\mathbb R}
\def \E{\mathbb E}

\def \P{\mathbb P}

\providecommand{\U}[1]{\protect\rule{.1in}{.1in}}
\newcommand{\bmf}[1]{\boldsymbol{#1}}
\newtheorem{theorem}{Theorem}

\newtheorem{corollary}[theorem]{Corollary}

\newtheorem{definition}[theorem]{Definition}

\newtheorem{lemma}[theorem]{Lemma}

\newtheorem{proposition}[theorem]{Proposition}
\newtheorem{remark}[theorem]{Remark}

\def\EE{\mathbb E}
\def\FF{\mathbb F}

\def\NN{\mathbb N}
\def\PP{\mathbb P}

\def\RR{\mathbb R}
\def\SS{\mathbb S}

\def\ZZ{\mathbb Z}

\def\balpha{\boldsymbol{\alpha}}
\def\bbeta{\boldsymbol{\beta}}

\def\bmu{\boldsymbol{\mu}}

\def\t{\tilde}

\def\bw{\omega_{0}}

\def\bU{{\boldsymbol U}}

\def\bW{{\boldsymbol W}}
\def\bX{{\boldsymbol X}}
\def\bY{{\boldsymbol Y}}
\def\bZ{{\boldsymbol Z}}

\def\cA{{\mathcal A}}

\def\cF{{\mathcal F}}

\def\cK{{\mathcal K}}

\def\cL{{\mathcal L}}

\def\cP{{\mathcal P}}

\def\cU{{\mathcal U}}
\def\cV{{\mathcal V}}

\begin{document}
\title{Restoring Uniqueness to Mean-Field Games by Randomizing the Equilibria}
\author{Fran\c{c}ois Delarue$^1$}
\date{}
\maketitle

\begin{center}
{\footnotesize Laboratoire J.-A. Dieudonn\'e,  

Universit\'e de Nice Sophia-Antipolis and UMR CNRS 7351, 
\vspace{-1.2pt}

Parc Valrose, 06108 Nice Cedex 02, France.}
\end{center}
\footnotetext[1]{\texttt{delarue@unice.fr}}

\begin{abstract}
We here address the question of restoration of 
uniqueness in mean-field games deriving from 
deterministic differential games with a large number 
of players. The general strategy for restoring uniqueness 
is inspired from earlier similar results on ordinary and stochastic differential equations. It consists in randomizing the equilibria
through an external noise. 

As a main feature, we choose the external noise as an infinite dimensional Ornstein-Uhlenbeck process. We first investigate existence and uniqueness of a solution to the noisy system made of the mean-field game forced by the
Ornstein-Uhlenbeck process. We also show how such a noisy system can be interpreted as the limit version of a stochastic differential game with a large number of players.
\end{abstract}

\section{Introduction}
The theory of mean-field games has encountered 
a tremendous success since it was introduced 
in 2006 by two independent groups, 
Lasry and Lions \cite{MFG1,MFG2,MFG3} on the one hand and 
Huang, Caines and Malham\'e
\cite{HuangCainesMalhame2,HuangCainesMalhame4}
 on the other hand. 

The purpose of mean-field games is to provide an asymptotic formulation for differential games involving a large number of players interacting with one another in a mean-field way.
The standard writing of mean-field games
consists in a forward-backward system involving 
a forward Fokker-Planck equation describing the state of the 
population in equilibrium and a backward Hamilton-Jacobi-Bellman describing the optimal cost to a typical player when the population is in equilibrium. This goes back to the earlier works 
of Lasry and Lions, see 
\cite{MFG1,MFG2,MFG3}, and to the subsequent series of lectures
 by Lions
 at the Coll\`ege de France, see 
\cite{Lions,Lionsvideo} together with the lecture notes \cite{Cardaliaguet} of Cardaliaguet.
This approach, referred to as ``the PDE approach'', 
 fits both the cases when the underlying differential games are deterministic or stochastic; in the deterministic case, the PDEs involved in the representation are 
first-order PDEs, whilst they are second-order PDEs in the stochastic framework. As pointed out in several works
 by Carmona and Delarue, 
 see
\cite{CarmonaDelarue_sicon,CarmonaDelarue_master,CarmonaDelarue_book_I,CarmonaDelarue_book_II,CarmonaDelarueLachapelle},
the problem may be reformulated in a purely Lagrangian form, using, instead of a forward-backward system of two PDEs, 
a forward-backward system of two 
ordinary or 
stochastic differential equations
of the McKean-Vlasov type, the name ``McKean-Vlasov'' emphasizing the fact that the coefficients of the equations depend upon the statistical distribution of the solution. In that case, 
the differential equations appearing in the representation 
are ordinary or stochastic 
according to the deterministic or stochastic nature of the differential game;
when the equations are ordinary, randomness manifests in the dynamics through the initial condition only.

Quite remarkably, the forward-backward structure is 
common 
to both formulations, the PDE one, 
in which equations are deterministic but set in infinite dimension,
and the Lagrangian one, in which equations are finite dimensional 
but of the McKean-Vlasov type. The forward-backward nature of the problem is a crucial feature
in the analysis of mean-field games  since 
forward-backward systems are known to be hard to solve:
Roughly speaking, Cauchy-Lipschitz theory 
for forward-backward systems of differential equations
holds in small time only, even when the differential equations are finite dimensional.
In arbitrary time, existence or uniqueness of solutions may fail, in which case the whole system is said to develop singularities in finite time. The typical example for such a phenomenon is provided by the inviscid one-dimensional backward Burgers equation: Solutions may be represented through characteristics that describe the motion of a representative particle. These characteristics solve the forward equation of the forward-backward system representing the Burgers equation; meanwhile, the backward equation describes the dynamics of the velocity of the particle, which remains constant along the motion of the particle. It is well known that, for some choices of the terminal condition, the forward paths may split, 
such a splitting phenomenon being usually referred to as a ``shock''. In this regard, one interesting question
is to decide of the right continuation 
of the forward paths once singularities have emerged and uniqueness has been lost. Anyhow, and quite remarkably, 
the existence of shocks is deeply connected with the form of the 
terminal condition and, under an appropriate monotonicity
assumption on the terminal condition, singularities cannot show up and existence and uniqueness hold true in arbitrary time.

The picture for solving mean-field games is quite similar. 
Sufficient conditions are known under which a solution 
(say for instance a solution to one of the two formulations)
does exist in arbitrary time, but, except in small time, uniqueness may not be guaranteed in most 
of the cases. We refer to the original papers \cite{MFG1,MFG2,MFG3}, to the video lectures \cite{Lions}, 
to the lecture notes \cite{Cardaliaguet}
and 
to the two-volume book \cite{CarmonaDelarue_book_I,CarmonaDelarue_book_II} 
for a review on the general strategy used to solve a mean-field game. 
We also refer to the subsequent papers 
\cite{Cardaliaguet_local_coupling,CardaliaguetGraber,CardaliaguetGraberPorrettaTonon,CardaliaguetMeszarosSantambrogio}
for 
other strategies, in connection with the theory of mean-field control problem,
and to
\cite{GomesPimentellog,GomesPimentelSanchezSub,GomesPimentelSanchezSuper}
 for the analysis of more intricated cases. 
For the small time analysis, we also refer to 
\cite{HuangCainesMalhame2} and to \cite[Chapter 4]{CarmonaDelarue_book_I} and \cite[Chapter 5]{CarmonaDelarue_book_II}. Existence of a solution to the Lagrangian formulation may be found in 
\cite{CarmonaDelarue_ecp,CarmonaDelarue_sicon,CarmonaDelarueLachapelle}, see also \cite[Chapter 4]{CarmonaDelarue_book_I}. 
Regarding uniqueness in arbitrary time, 
things are as follows. Similar to the analysis of the Burgers equation, 
uniqueness is know to hold when the coefficients satisfy a suitable monotonicity condition with respect to 
the distribution of the population. The most popular monotonicity property
used in this direction is due to Lasry and Lions, see once again \cite{MFG1,MFG2,MFG3}, and is usually referred to as the Lasry-Lions monotonicity condition.
However, as emphasized in \cite{Ahuja} and in \cite[Chapters 4 and 5]{CarmonaDelarue_book_I}, other forms of monotonicity may
be used.  

In analogy with our short description of the forward-backward system associated with Burgers' equation, the forward-backward system 
used for representing a mean-field game
(whatever the formulation) reads as the system of characteristics 
of some partial differential equation. In the framework  
of mean-field games, this partial differential equation is called the ``master equation'' of the game, the word ``master'' emphasizing the fact that 
the equation encapsulates all the 
information that is necessary to describe the equilibria of the game. 
This equation was investigated first by Lions in his lectures at the Coll\`ege de France and then by Gangbo and Swiech 
\cite{GangboSwiech} in small time, 
and by 
Chassagneux, Crisan and Delarue
\cite{ChassagneuxCrisanDelarue}
and by Cardaliaguet, Delarue, Lasry and Lions
\cite{CardaliaguetDelarueLasryLions}
 in arbitrary time. In the latter reference, it is shown to play a crucial role in the justification of the passage to the limit, 
from games with finitely many players to mean-field games. 
In arbitrary time, analysis of the equation is performed under 
the additional assumption that coefficients satisfy the Lasry-Lions monotonicity condition. We refer to   
\cite[Chapters 5 and 6]{CarmonaDelarue_book_II} 
for another point of view on 
the results contained in
\cite{ChassagneuxCrisanDelarue,CardaliaguetDelarueLasryLions}
and to
\cite{BensoussanFrehseYam,BensoussanFrehseYam2,
BensoussanFrehseYam3,CarmonaDelarue_master,GomesSaude,GomesVoskanyan,KolokoltsovTroeva} for other and more heuristic approaches.

In the current paper, we consider the case when the Lasry-Lions monotonicity condition may fail, the question being to find a strategy to restore uniqueness. Pursuing the same parallel as before, we observe that, somehow, a similar program has been investigated for the Burgers equation: Adding a Laplace operator in front of the Burgers equation permits to restore the existence and uniqueness of a classical solution in arbitrary time (as opposed to the inviscid case, for which 
the existence of a classical solution may fail). From the Lagrangian point of view, the additional Laplace operator reads as a Brownian motion that forces the motion of the underlying particle. Similar to the viscous version of the Burgers equation, the stochastically forced forward-backward system 
describing the ``random characteristics'' of the viscous Burgers equation is know to be uniquely solvable, see Delarue \cite{Delarue02}. In a way, ``noise restores uniqueness in the Lagrangian formulation''. Our goal here is
to adapt this strategy to mean-field games. 

The idea of restoring uniqueness by means of a random forcing has been extensively studied in probability theory. It goes back to the earlier work of Zvonkin \cite{Zvonkin} on the solvability of one-dimensional stochastic differential equations driven by non-Lipschitz continuous drifts. Several people also contributed to the subject and investigated the higher dimensional framework, among which Veretennikov \cite{veretennikov}, Flandoli, Russo and Wolf
\cite{fla:rus:wol:03,fla:rus:wol:04}, Krylov and R\"ockner 
\cite{kry:roc:05}, Davie \cite{dav:07}... 
Similar questions have been also addressed in the framework of infinite dimensional stochastic differential equations, see for instance Flandoli, Gubinelli and Priola \cite{FlandoliGubinelliPriola} and the monograph by Flandoli \cite{Flandoli}. In any case, the idea is to force in a convenient way the Lagrangian dynamics in order to restore uniqueness of solutions.  
Transposed to mean-field games theory, the question is here 
to find a suitable randomly forced version of the original mean-field games in order to guarantee uniqueness of the equilibria. 

Here is our main result: For a certain class of coefficients, we 
manage to restore uniqueness to mean-field games --deriving from a deterministic differential game--
 by means of a stochastic forcing. The stochastic forcing mostly consists in an infinite dimensional Ornstein-Uhlenbeck process. The reason why it is chosen of infinite dimension is well-understood. Roughly speaking, the stochastic forcing is indeed intended to 
act on the elements of the ``infinite dimensional manifold'' formed by 
 the $d$-dimensional probability measures 
with a finite second-order moment, which is usually called ``the Wasserstein space''
($d$ is the state dimension of a typical player). Here,
 probability measures  
 are
 used to describe the state of the population, whilst the limitation to probability measures with a finite second-order moment is a convenient assumption which permits to benefit from the Hilbertian  
structure of any $L^2$ space constructed above the Wasserstein space. 
Returning to the description of the forcing applied to the mean-field system, it is then well-understood that, in order to capture all the ``possible
tangent directions'' to the manifold at any point of it, it is necessary to use a noise of infinite dimension. In order to bypass any description of the differential geometry on the space of probability measures, we use the approach introduced by Lions in his lectures: We lift equilibria from 
the space of probability measures to a well-chosen space of square-integrable random variables and then use, as we just alluded to, 
the Hilbertian structure of this $L^2$ space. Fortunately, the Lagrangian description of mean-field games gives a canonical way to realize such a lift. Our strategy 
then consists in forcing the dynamics of the random variables 
representing the equilibria. In other words, our goal is to force a differential equation defined on an $L^2$ space. A convenient way to do so is to force the modes of the solution along an orthonormal basis of $L^2$. For instance, when $L^2$ is chosen as the space $L^2({\mathbb S}^1;\RR^d)$ of square-integrable Borel mappings from ${\mathbb S}^1$ to $\RR^d$, where ${\mathbb S}^1$ denotes the one-dimensional torus, it suffices to force the Fourier modes of 
square-integrable $\RR^d$-valued functions defined on ${\mathbb S}^1$. 
It is then a standard fact from the theory of stochastic partial differential equations that the Ornstein-Ulhenbeck 
process
has nice smoothing properties on $L^2({\mathbb S}^1;\RR^d)$, which is the key feature for restoring uniqueness.  

In addition to proving existence and uniqueness of a solution to 
the noisy version of the original mean-field game, we
 also show that the randomly forced version may be interpreted as 
 the limit of a game with a large number of agents. As a main feature, the finite game
 not only exhibits mean field interactions, which is well expected, but also local interactions to nearest neighbours, which 
 is certainly a new point in the literature on mean-field games; from a mathematical point of view, 
 local interactions arise from the discretization of the operator driving the 
 additional infinite dimensional 
 Ornstein-Ulhenbeck process. The route we take to connect the finite and the infinite regimes 
 is to prove that, from any equilibrium to the limiting problem, we can construct an approximate Nash 
 equilibrium to the finite system. Although this way of doing is pretty standard in the theory of 
 mean-field games, it turns out to be more challenging in our setting because of 
 the additional local interactions. Of course, another route would consist 
 in proving that any (say closed loop) equilibria to the finite player system do converge to the limiting equilibrium as the number 
 of players grows up. It turns out to be a pretty difficult question in the framework of mean-field games; 
 in this framework, 
 the only generic approach 
that has been known to handle the convergence of closed loop equilibria is due to  
\cite{CardaliaguetDelarueLasryLions}
and is based on the aforementioned master equation. 
We guess that a similar approach could be implemented here and we hope to address it in a future work. In fact, 
a form of master equation is already addressed in the paper: We prove that the equilibrium strategy (in the limiting regime)
can be put in a feedback form and we show that the feedback function, which may be regarded as a function from 
$L^2({\mathbb S}^1;\RR^d)$ into itself, is a mild solution to a system of nonlinear equations on 
$L^2({\mathbb S}^1;\RR^d)$, driven by the second-order operator generated by the Ornstein-Ulhenbeck process inserted in the dynamics; the latter system reads as 
a kind of master equation for our problem. We just say a ``kind of'' because the usual master equation for mean field games 
is the equation satisfied by the value function and not by the feedback function. In the standard mean field game regime, 
both are explicitly connected since the feedback function is the derivative of the value function with respect to the so-called 
``private state variable''. Things are slightly different in our setting and we prefer to work, in the noisy regime, with the feedback function directly. At the end of the day, our guess is that, to plug our own version of the master equation into the machinery 
developed by 
\cite{CardaliaguetDelarueLasryLions}, we would need the feedback 
function to be more regular than what we show below. Once again, this question is deferred to another work. 

Another interesting prospect that we would like to investigate is the zero noise limit: We guess that any limit of the solutions (to the noisy system), as the intensity of the forcing decreases to $0$, 
should generate a   
randomized equilibrium to the original mean-field game. 
We are not aware of similar results in the theory of mean-field games, except maybe 
in the case investigated by 
 Foguen \cite{Foguen}. There, restoration of uniqueness is investigated 
for linear-quadratic mean-field games. In comparison with the general case we handle here, 
linear-quadratic mean-field games present the main advantage to have parametrized solutions: Equilibria are Gaussian and are thus parametrized by their mean and variance and thus live in 
a finite-dimensional subspace of the space of probability measures. 
In this case, it suffices to use a finite dimensional noise to restore uniqueness, 
which is precisely what is done in 
\cite{Foguen}; then, it seems that, for some linear-quadratic mean-field games, zero noise limits could be addressed by using arguments 
similar 
to \cite{BaficoBaldi}. Once again, we hope to make this point clear in a future work in collaboration with Foguen. 

Lastly, we emphasize the fact that all these questions should be revisited 
for mean-field games deriving from stochastic differential games with idiosyncratic noises. We believe that part of the technology developed in the paper could be 
recycled in this framework, except for the fact, due to the simultaneous presence of two sources of noise --the idiosyncratic one and 
the external one used to restore uniqueness--, the formulation of the randomized version of the game should require a modicum of care. We make this fact clear in the text.

The paper is organized as follows. We present 
in Section
\ref{se:mollified:MFG}
the randomized version of the game. Main results are exposed in 
Section \ref{se:main}. The proof of existence and uniqueness of a solution to the randomized game
is given in Section \ref{se:proof}. Connection with finite games is 
addressed in Section 
\ref{se:construction:approximate:proof}.

\section{Mollified/Randomized MFG}
\label{se:mollified:MFG}
We first present the original Mean-Field Game (MFG for short) 
and then describe the ``mollified'' or ``randomized'' version 
that is expected to be uniquely solvable. 

Throughout the article, $d$ is an integer greater than 1 and ${\mathcal P}_{2}(\R^d)$ denotes the space of probability measures over $\R^d$. It is equipped with the $2$-Wasserstein distance (see for instance \cite{Villani_AMS,Villani,CarmonaDelarue_book_I}):
\begin{equation*}
\begin{split}
\forall \mu,\nu \in {\mathcal P}_{2}(\R), 
\quad
W_{2}(\mu,\nu)
&= 
\inf_{\pi}
\biggl( \int_{\R^d \times \RR^d} \vert x-y \vert^2 d \pi(x,y) 
\biggr)^{1/2},
\end{split}
\end{equation*}
where the infimum in the last line is taken over all the probability measures $\pi \in {\mathcal P}_{2}(\R^d \times \R^d)$ that 
have $\mu$ and $\nu$ as respective marginals.

\subsection{Original problem}
\label{subse:original:pb}
We start with a simple MFG consisting of the following matching 
problem:
\begin{enumerate}
\item Given a 
probability space $(\Omega,{\mathcal A},\P)$
and a 
flow of
probability measures $\bmu=(\mu_{t})_{t \in [0,T]}$
on $\R^d$, consider the optimization problem
\begin{equation*}
J^{\bmu}({\boldsymbol \alpha}) =  \E \Bigl[ 
g\bigl(X_{T}^{\balpha},\mu_{T}
\bigr)
+
\int_{0}^T 
\bigl(
f\bigl(X_{t}^{\balpha},\mu_{t}\bigr) + \frac12 \vert \alpha_{t} \vert^2
\bigr) dt 
\Bigr],
\end{equation*}
over controlled dynamics of the form
\begin{equation}
\label{eq:ode:x}
dX_{t}^{\balpha} = b(X_{t}^{\balpha},\mu_{t}) dt + \alpha_{t} dt, 
\end{equation}
with
the initial condition $X_{0}^{\alpha}=X_{0}$, $X_{0}$  
being a random variable from $\Omega$ to $\R^d$
with $\mu_{0}$ as distribution.
\item Find $(\mu_{t})_{t \in [0,T]}$ in such a way that the 
flow of marginal measures of the optimal path
$(X_{t}^\star)_{t \in [0,T]}$
 in the above optimization 
problem satisfies 
\begin{equation}
\label{eq:match:mfg}
\mu_{t} = {\mathcal L}\bigl(X_{t}^\star\bigr), \quad t \in [0,T].
\end{equation}
\end{enumerate}
Here, $\balpha$ is called the control and is a jointly-measurable mapping 
\begin{equation*}
\balpha : [0,T] \times \Omega \ni (t,\omega) 
\mapsto \alpha_{t}(\omega) \in \R^d,
\end{equation*}
satisfying
\begin{equation*}
\E \int_{0}^T \vert \alpha_{t} \vert^2 dt < \infty.
\end{equation*}
The coefficient $b : \R^d \times {\mathcal P}_{2}(\R^d) \rightarrow \R^d$ is called the drift.  It is
assumed to be jointly Lipschitz continuous, so that
\eqref{eq:ode:x} is uniquely solvable for any realization $\omega \in \Omega$ and the solution $\bX : [0,T] \times \Omega \ni (t,\omega) 
\mapsto X_{t}(\omega) \in \R^d$ is also jointly-measurable. The coefficients $g : 
\R^d \rightarrow \R$ and $f : \RR^d \times \cP_{2}(\RR^d) \rightarrow \RR$ are called cost functionals. They are assumed be jointly continuous on $\RR^d \times \cP_{2}(\RR^d)$. Throughout the paper, we assume them to 
be at most of quadratic growth in the sense that, 
for some constant $C \geq 0$,
\begin{equation*}
\vert f(x,\mu) \vert + \vert g(x,\mu) \vert \leq C \bigl( 1 + \vert x \vert^2 
+ M_{2}(\mu)^2 \bigr), \quad x \in \R^d, \quad \mu \in \cP_{2}(\RR^d),
\end{equation*}
where $M_{2}(\mu)^2 = \int_{\R^d} \vert x \vert^2 d\mu(x)$. 
In particular, it is well checked that the expectation in the definition of the cost $J$ makes sense.  

\begin{remark}
All the coefficients are here assumed to be time homogeneous. This is for simplicity only and the results given below can be extended quite easily to the time-inhomogeneous framework. 
Similarly, the fact that $f$ is a quadratic function of $\alpha$ is for convenience only; we could handle more general 
running costs of the form $f(x,\mu,\alpha)$ that are uniformly convex in $\alpha$, see for instance 
\cite[Chapters 3 and 4]{CarmonaDelarue_book_I}. However, the fact that $b$ is linear in $\alpha$ is really crucial for our purpose, at least if we want to make use, as we do below, of the sufficient version of the Pontryagin principle.

Another possible generalization 
would be to insert a Brownian motion in the dynamics \eqref{eq:ode:x}, in which case the mean-field game would be called ``stochastic'' or 
``second-order''. However, the approach developed 
below
for restoring uniqueness of solutions 
does not apply to that case, see Remark 
\ref{rem:why:not:sto}
below. We hope to address this question in a future work. 
\end{remark}

Usually,
solutions to the matching problem 
\eqref{eq:match:mfg}
may be characterized in two ways. The original one 
is to characterize the optimization problem in the first item above through a first order Hamilton-Jacobi-Bellman equation (HJB for short):
\begin{equation}
\label{eq:hjb}
\begin{split}
\partial_{t} u(t,x) + b(x,\mu_{t}) \cdot \partial_{x} u(t,x) + 
f(x,\mu_{t}) - \frac12 \vert \partial_{x} u(t,x) \vert^2 = 0,
\end{split}
\end{equation}
for $(t,x) \in [0,T] \times \R^d$, with $u(T,x) = g(x,\mu_{T})$
as boundary condition. Here, the function 
$u : [0,T] \times \R^d \rightarrow \R$ is understood as the value function of the optimization problem (in the environment 
$\bmu=(\mu_{t})_{0 \leq t \leq T}$). Given the value function, 
it is known that the optimal control process in the optimization problem reads (at least formally since the gradient below may not exist or may only exist as a multi-valued mapping):
\begin{equation*}
\balpha^\star = \bigl( \alpha_{t}^\star = - \partial_{x} u(t,X_{t}^{\star}) \bigr)_{0 \leq t\leq T},
\end{equation*}
where $(X_{t}^{\star})_{0 \leq t \leq T}$ now denotes the 
solution of the ordinary differential equation:
\begin{equation*}
dX_{t}^{\star} = \Bigl( b\bigl(X_{t}^{\star},\mu_t\bigr) -
\partial_{x} u \bigl(t,X_{t}^{\star}
\bigr)\Bigr) dt, \quad t \in [0,T].
\end{equation*}
It is now easy to implement analytically the fixed point condition in the second item above. 
Under the identification 
$(\mu_{t} = {\mathcal L}(X_{t}^{\star}))_{0 \leq t \leq T}$,
the flow $\bmu = (\mu_{t})_{0 \leq t \leq T}$
must solve the nonlinear Fokker-Planck equation:
\begin{equation}
\label{eq:fp}
\partial_{t} \mu_{t} + \partial_{x}
\Bigl(
\bigl( 
b(x,\mu_{t})
- \partial_{x} u(t,x)
\bigr)
\mu_{t}
\Bigr)
=  0, \quad (t,x) \in [0,T] \times \R^d.
\end{equation}
with the initial condition $\mu_{0}$ for the population. 
The forward-backward system made of \eqref{eq:hjb} and 
\eqref{eq:fp} is usually called the MFG system of PDEs. 
We refer to aforementioned references \cite{Cardaliaguet,Lions}
for further details.
\vspace{5pt}

Another strategy for characterizing the equilibria is to use the Pontryagin principle. Under appropriate conditions, we know that the optimal paths 
 of the control problem $\inf_{\balpha} J^{\bmu}(\balpha)$ 
 in the first item of the above definition of an MFG equilibrium solve
 the forward-backward system of two ODEs:
\begin{equation}
\label{eq:original:smp}
\begin{split}
&dX_{t}^{\star} = \Bigl( b\bigl(X_{t}^{\star},\mu_{t} \bigr) - Y_{t}^{\star} \Bigr) dt,
\\
&dY_{t}^{\star} = \Bigl( -\partial_{x} b \bigl( X_{t}^{\star},\mu_{t}
\bigr) Y_{t}^{\star}
- \partial_{x} f \bigl( X_{t}^{\star},\mu_{t}
\bigr) 
\Bigr) dt, 
\end{split}
\end{equation}
with the 
initial condition $X_{0}^\star = X_{0}$
and the 
terminal condition 
$Y_{T} = \partial_{x} g(X_{T}^\star,\mu_{T})$. 
Implementing the matching condition 
\eqref{eq:match:mfg}
in the second item of the definition of an MFG equilibrium, 
we deduce that equilibria of the MFG must solve the 
forward-backward system of the McKean-Vlasov type:
\begin{equation}
\label{eq:original:mkv}
\begin{split}
&dX_{t}^{\star} = \Bigl( b\bigl(X_{t}^{\star},
{\mathcal L}(X_{t}^{\star})
 \bigr) - Y_{t}^{\star} \Bigr) dt,
\\
&dY_{t}^{\star} = \Bigl( -\partial_{x} b \bigl( X_{t}^{\star},{\mathcal L}(X_{t}^{\star})
\bigr) Y_{t}^{\star}
- \partial_{x} f \bigl( X_{t}^{\star},{\mathcal L}(X_{t}^{\star})
\bigr) 
\Bigr) dt, 
\end{split}
\end{equation}
with the terminal condition $Y_{T} = \partial_{x} g(X_{T}^\star,\cL(X_{T}^\star))$. Under suitable convexity properties of the coefficients in the variable $x$, which we spell out in Subsection 
\ref{subse:standing:assumptions}
below, the system 
\eqref{eq:original:mkv} is not only a necessary condition 
satisfied by any equilibria of the mean-field game
but is also a sufficient condition. In this framework,  
\eqref{eq:original:mkv} characterizes the equilibria of the game. 
This is precisely this system that we force stochastically below. 
\vspace{5pt}

Throughout the article, we focus on this specific convex regime when the 
Pontryagin principle is both a necessary and a sufficient condition of optimality. Although it demands strong
assumptions on the structure of the coefficients in the spatial variable 
$x$, this so-called ``convex regime'' turns out to be especially useful for our purposes: 
It provides a sharp framework under which, for a given input $\bmu=(\mu_{t})_{0 \leq t \leq T}$, 
the system
\eqref{eq:original:smp}
is uniquely solvable for any initial condition and its solution 
is stable under perturbation of the initial condition and perturbation of the 
input. It is worth mentioning that, even in this strong setting, it still makes sense to address the restoration of uniqueness for the mean-field game, since  
the McKean-Vlasov 
forward-backward system \eqref{eq:original:mkv} may not be uniquely solvable.
Clearly, 
we shall appreciate having a sharp
 framework for solving the control problem $\inf_{\balpha} J^{\bmu}(\balpha)$ 
 as it will 
permit to 
focus on the 
difficulties that are exclusively related with the non-uniqueness of the MFG equilibria.

\subsection{Reformulation}
\label{subse:reformulation}
In order to proceed, we first notice that
$(\Omega,{\mathcal A},\P)$ may chosen as the probability space
$({\mathbb S}^1,{\mathcal B}({\mathbb S}^1),\textrm{\rm Leb}_{1}
)$, where 
$\textrm{\rm Leb}_{1}$ is the Lebesgue measure. 
In this regard, we recall 
from 
\cite{BlackwellDubins} that there exists a measurable function 
$\Psi : {\mathbb S}^1 \times {\mathcal P}(\R^d) \rightarrow 
\R^d$ 
such that, for every probability $\mu$ on $\R^d$, 
$[0,1] \ni u \mapsto \Psi(u,\mu)$ is a random variable with $\mu$ as distribution.

With such a convention, the control ${\boldsymbol \alpha}$ is understood 
as a jointly-measurable mapping
\begin{equation*}
{\boldsymbol \alpha} : [0,T] \times {\mathbb S}^1 
\ni (t,x) \mapsto \alpha_{t}(x) \in \R^d,
\end{equation*}
and the cost functional may be rewritten as 
\begin{equation}
\label{eq:Jmu:reformulation:1}
\begin{split}
J^{\bmu}({\boldsymbol \alpha})
 &= 
 \int_{{\mathbb S}^1}
 g
 \bigl(x,
{\mathcal L}(X_{T}^{\boldsymbol \alpha})
 \bigr) d{\mathcal L}(X_{T}^{\boldsymbol \alpha})(x)
 \\
&\hspace{15pt} 
+
 \int_{0}^T
\biggl[  
\int_{{\mathbb S}^1}
 f
 \bigl(x,
{\mathcal L}(X_{t}^{\boldsymbol \alpha})
 \bigr) d{\mathcal L}(X_{t}^{\boldsymbol \alpha})(x)
 + 
 \frac12 \int_{{\mathbb S}^1}
\vert \alpha_{t}(x)\vert^2 dx 
\biggr] dt. 
\end{split}
\end{equation}
With this reformulation, we introduce   
the $L^2$ spaces $L^2({\mathbb S}^1) = L^2({\mathbb S}^1,
{\mathcal B}({\mathbb S}^1),
\textrm{\rm Leb}_{1})$ and $L^2({\mathbb S}^1;\RR^d) \cong [L^2({\mathbb S}^1)]^d$. 
A key fact is that the functions 
\begin{equation*}
e^{0} : {\mathbb S}^1 \ni x \mapsto 1,
\quad
e^{n,+} : {\mathbb S}^1 \ni x \mapsto 
\sqrt{2} \cos\bigl( 2 \pi n x), \quad 
e^{n,-} : 
{\mathbb S}^1 \ni x \mapsto 
\sqrt{2} \sin\bigl( 2 \pi n x), \quad n \in \NN^*, 
\end{equation*}
form an orthonormal basis of $L^2({\mathbb S}^1)$. 
In particular, for any element $\ell \in L^2({\mathbb S}^1)$, we call 
$\ell^{0}$, $\ell^{n,+}$ and ${\ell}^{n,-}$, $n \in \NN^*$, the different weights 
of $\ell$;
we use the same notation when $\ell \in L^2(\SS^1;\RR^d)$, in which case
$\ell^{0}$, $\ell^{n,+}$ and ${\ell}^{n,-}$ are vectors of size $d$.
 Then, we may write
\begin{equation*}
\int_{{\mathbb S}^1}
\vert \alpha_{t}(x) \vert^2
dx = \vert \alpha_{t}^0 \vert^2 + \sum_{n \in \NN^*}
\bigl( \vert \alpha_{t}^{n,+} \vert^2 
+ 
\vert \alpha_{t}^{n,-} \vert^2
\bigr),
\end{equation*}
which we shall often summarize into 
\begin{equation*}
\int_{{\mathbb S}^1}
\vert \alpha_{t}(x) \vert^2
dx = \sum_{n \in \NN}
\vert \alpha_{t}^{n,\pm} \vert^2,
\end{equation*}
with the convention that $\alpha^{0,+}=\alpha^0$ and 
$\alpha^{0,-}=0$.

Moreover, given a mapping $h : \R^d \rightarrow \R$, at most of linear growth,
we may consider the mapping
\begin{equation*}
{\mathfrak h}_{0} : L^2\bigl({\mathbb S}^1;\RR^d) \ni \ell \mapsto {\mathfrak h}_{0}(\ell) = 
\int_{{\mathbb S}^1} h \bigl( \ell(x) \bigr) dx.
\end{equation*}
Then, we observe that the cost functional $J^{\bmu}$ may be rewritten:
\begin{equation}
\label{eq:Jmu:reformulation:2}
J^{\bmu}(\balpha) = 
{\mathfrak g}_{0}\bigl(X_{T}^{\balpha}(\cdot),\mu_{T}\bigr)
+
\int_{0}^T \Bigl\{ 
{\mathfrak f}_{0}\bigl(X_{t}^{\balpha}(\cdot),\mu_{t}\bigr) + \frac{1}{2} 
\Bigl( \vert \alpha^0_{t} \vert^2 +
\sum_{n \in \NN}
\bigl(
 \vert \alpha_{t}^{n,+}\vert^2 
+ 
 \vert \alpha_{t}^{n,-}\vert^2 
 \bigr) \Bigr) \Bigr\} dt,
\end{equation}
where ${\boldsymbol X}^{\balpha}(\cdot)=(X_{t}^{\balpha}(\cdot))_{0 \leq t \leq T}$
(pay attention to the dot we put in the notation to emphasize the fact that the path has functional values)
 is
a path with values in $L^2({\mathbb S}^1;\RR^d)$ 
in such a way that 
$\bX(\cdot)=\bX^{\balpha}(\cdot)$ (we get rid of the superscript $\balpha$ to simplify the notations) satisfies
\begin{equation}
\label{eq:modes:deterministic:X}
\dot{X}_{t}^{n,\pm} = {\mathfrak b}^{n,\pm}\bigl(X_{t}(\cdot),\mu_{t}\bigr)  + \alpha_{t}^{n,\pm}, \quad t \in [0,T], \quad n \in \NN,
\end{equation}
where, for $\ell \in L^2({\mathbb S}^1;\R^d)$ and $\mu \in {\mathcal P}_{2}(\R^d)$, 
\begin{equation*}
{\mathfrak b}^{n,\pm}(\ell,\mu)=
\int_{{\mathbb S}^1} b\bigl(\ell(x),\mu \bigr) e^{n,\pm}(x) dx, \quad 
n \in \NN,
\end{equation*}
denote the modes of 
${\mathfrak b}(\ell,\mu)$.
%

\subsection{Enlarged problem}
\label{se:1:subse:enlarged}
A strategy for restoring uniqueness to mean-field games now consists in 
forcing the modes 
$(\bX^{n,\pm})_{n \in \NN}$ introduced in the previous paragraph. 
To do so, we need to disentangle the two sources of noise that 
will manifest in the construction of the new mean field game: On the one hand, the initial condition is still defined as a square-integrable random variable on the torus $\SS^1$
(equipped with the collection ${\mathcal L}(\SS^1)$  of Lebesgue sets); on the other hand, 
we need another space for carrying the random forcing acting on the nodes
$(\bX^{n,\pm})_{n \in \NN}$. 

Having this picture of our general strategy in mind, we now enlarge the probability space and consider 
$\Omega= {\mathbb S}^1 \times \Omega_{0}$, 
where $(\Omega_{0},{\mathcal A}_{0},\FF_{0}=({\mathcal F}_{0,t})_{t \in [0,T]},\PP_{0})$ is a complete filtered probability 
space equipped with a collection $(\bW^0,(\bW^{n,+},\bW^{n,-})_{n \in \NN^*})$ of  
$\FF_{0}$-Brownian motions of dimension
$d$. The filtration $\FF_{0}$ satisfies the usual conditions. 

We then equip $\Omega$ with the completion ${\mathcal A}$ of ${\mathcal L}(\SS^1) \otimes {\mathcal A}_{0}$
and with the completion $\PP$ of $\textrm{\rm Leb}_{1} \otimes \PP_{0}$. We call
$\FF$ the completion of the filtration $({\mathcal L}(\SS^1) \otimes {\mathcal F}_{0,t})_{t \in [0,T]}$ and 
we denote by 
 $\xi$ the identity mapping on $\SS^1$
 (which is extended in a canonical way to $\Omega$). 
Despite the fact that $\Omega$ has been enlarged, we keep 
the same notations as above for 
${\mathfrak h}_{0}(\ell)$ and $\ell^{n,\pm}$
whenever $\ell$ is an element of $L^2({\mathbb S}^1;\RR^d)$. In particular,
whenever $X$ is a square-integrable random variable 
defined on $\Omega$, 
we may consider, for $\PP_{0}$-almost every $\omega_{0} \in \Omega_{0}$, the random variable $X(\cdot,\omega_{0})$ on 
${\mathbb S}^1$ and then ${\mathfrak h}_{0}(X(\cdot,\omega_{0}))$ and $X^{n,\pm}(\cdot,\omega_{0})$.
Recall indeed from the version of Fubini's theorem for completion 
of product spaces that, for $\PP_{0}$-almost every 
$\omega_{0}$, 
$X(\cdot,\omega_{0})$ is a square-integrable random variable 
on 
$({\mathbb S}^1,{\mathcal L}(\SS^1))$, see Lemma 
\ref{lem:measurability} for more details. 
\vspace{5pt}

The question now is to explain how to use 
the collection 
$(\bW^0,(\bW^{n,\pm})_{n \in \NN^*})$ in order to construct
a uniquely solvable
 randomized mean field game. A na\"ive way would consist in forcing each mode process
$\bX^{n,\pm} = (X_{t}^{n,\pm})_{0 \leq t \leq T}$ in 
\eqref{eq:modes:deterministic:X}, for $n \in \NN$, by the corresponding Wiener process $\bW^{n,\pm}$ (with the same convention as above that
$\bX^{0}$ and $\bW^{0}$ are understood 
as $\bX^{0,+}$ and $\bW^{0,+}$). 
However, it is a well-known fact that the solution 
\begin{equation*} 
X_{t}(\cdot,\omega_{0}) = \sum_{n \in \NN} X_{t}^{n,\pm}(\cdot,\omega_{0}) e^{n,\pm}(\cdot), \quad t \in [0,T],
\end{equation*}
would not belong to $L^2({\mathbb S}^1;\RR^d)$.

In order to render the modes 
$((X_{t}^{n,\pm})_{0 \leq t \leq T})_{n \in \NN}$ square summable, 
we may 
force \eqref{eq:modes:deterministic:X}
 by another $\FF_{0}$-semi-martingale process $\bU^{n,\pm}$ 
such that 
\begin{equation}
\label{eq:square-integrability:U}
\EE_{0} \Bigl[ \sup_{0 \leq t \leq T} \Bigl( \sum_{n \in \NN}
\vert U^{n,\pm}_{t} \vert^2 
\Bigr) \Bigr] < \infty, 
\end{equation}
namely
\begin{equation}
\label{eq:sde:U}
dX_{t}^{n,\pm} = \Bigl( 
{\mathfrak b}^{n,\pm} \bigl(X_{t}(\cdot), \mu_{t} \bigr)
+ \alpha_{t}^{n,\pm}
 \Bigr)dt 
  +  dU_{t}^{n,\pm}, 
 \quad t \in [0,T], \quad n \in \NN.
\end{equation}
Assume for instance that 
\begin{equation}
\label{eq:square-integrability:any X:0}
 \int_{0}^T
 \sum_{n \in \NN} 
 \Bigl( \sup_{x \in \SS^1}
 \vert {\mathfrak b}^{n,\pm}(x,\mu_{t}) \vert^2
\Bigr)
 dt < \infty. 
 \end{equation}
 Then,
\begin{equation}
\label{eq:square-integrability:any X}
\EE_{0} \Bigl[ \sup_{0 \leq t \leq T} \Bigl( \sum_{n \in \NN}
\vert X^{n,\pm}_{t} \vert^2 
\Bigr) \Bigr] < \infty, 
\end{equation} 
and we can regard
\begin{equation*}
X_{t}(\cdot,\omega_{0}) = \sum_{n \in \NN} X_{t}^{n,\pm}(\omega_{0}) e^{n,\pm}(\cdot), \quad t \in [0,T],
\end{equation*}
as a process with values in $L^2(\SS^1;\RR^d)$. 
\vskip 4pt

In this regard, the following lemma 
(see for instance 
\cite[Chapter 2]{CarmonaDelarue_book_II} for similar considerations)
makes clear the connection between random variables 
from $\Omega$ into $\RR$ and random variables 
from $\Omega_{0}$ into $L^2(\SS^1)$:
\begin{lemma}
\label{lem:measurability}
Assume that $X$ is a square-integrable $\RR^d$-valued random variable on $\Omega$. Then, for $\PP_{0}$ almost every 
$\omega_{0} \in \Omega_{0}$, $\SS^1 \ni x \mapsto X(x,\omega_{0}) \in L^2(\SS^1;\RR^d)$; moreover, we can construct a random
variable $X(\cdot)$ on $\Omega_{0}$ with values in $L^2(\SS^1;\RR^d)$, such that, for $\PP_{0}$-almost every $\omega_{0} \in \Omega_{0}$, 
$\SS^1 \ni x \mapsto X(x,\omega_{0})$ coincides in $L^2(\SS^1;\RR^d)$ with the realization of the variable 
$X(\cdot)$ at $\omega_{0}$. Conversely, 
given a random variable $X(\cdot)$ from $\Omega_{0}$ to $L^2(\SS^1;\RR^d)$, we can construct a random
variable $X$ on $\Omega$ such that, for $\PP_{0}$-almost every $\omega_{0} \in \Omega_{0}$, 
$\SS^1 \ni x \mapsto X(x,\omega_{0})$ coincides in $L^2(\SS^1;\RR^d)$ with the realization of the variable 
$X(\cdot)$ at $\omega_{0}$. 
%
\end{lemma}

\begin{proof}
The proof is pretty straightforward. 
Given a square-integrable $\RR^d$-valued random variable on $\Omega$, 
Fubini's theorem for completion 
of product spaces says that, for $\PP_{0}$-almost every 
$\omega_{0}$, 
$\SS^1 \ni x \mapsto 
X(x,\omega_{0})$ is a square-integrable random variable 
on 
$({\mathbb S}^1,{\mathcal L}(\SS^1))$. In particular, 
for $\PP_{0}$-almost every $\omega_{0}$, we can define 
$X^{n,\pm}(\omega_{0}) = \int_{\SS^1} X(x,\omega_{0}) e^{n,\pm}(x) dx$. Each 
$X^{n,\pm}$ is a random variable (on $\Omega_{0}$). We then let
\begin{equation*}
X(\cdot) = \sum_{n \in \NN} X^{n,\pm} e^{n,\pm}(\cdot).
\end{equation*}
Noticing that a mapping $\chi(\cdot)$ from $\Omega_{0}$ into $L^2(\SS^1;\RR^d)$ is measurable with respect to 
a $\sigma$-field ${\mathcal G}$ if and only if its modes $(\chi^{n,\pm})_{n \in \NN}$ are measurable 
with respect to ${\mathcal G}$, we deduce that $X(\cdot)$ is a random variable 
from $\Omega_{0}$ to $L^2(\SS^1;\RR^d)$. 

Conversely, if we are given a square integrable
random variable $X(\cdot)$ from $\Omega_{0}$ into $L^2(\SS^1;\RR^d)$, then we can define 
$(X^{n,\pm})_{n \in \NN}$ as random variables with values in $\RR^d$. We then let 
\begin{equation*}
X^n(x,\omega_{0}) = \sum_{k=0}^n X^{k,\pm}(\omega_{0}) e^{k,\pm}(x), \quad n \in \NN.
\end{equation*}
Obviously, we can identify $X^n$ (seen as a random variable on $\Omega$ with values in $\RR^d$) with 
$X^n(\cdot)$ (seen as a random variable on $\Omega_{0}$ with values in $L^2(\SS^1;\RR^d)$). 
It is clear that $X^n(\cdot)$ converges to $X(\cdot)$ in 
$L^2(\Omega_{0},{\mathcal A}_{0},\PP_{0};L^2(\SS^1;\RR^d))$
and $X^n$ has a limit $\tilde X$ in $L^2(\Omega,{\mathcal A},\PP;\RR^d)$. 
We then identify $\tilde X(\cdot)$ with $X(\cdot)$. 
\end{proof}

 Importantly, observe that we can proceed similarly with processes. 
For instance, we can associate, with any $\FF$-progressively-measurable process with values in $\RR^d$, an 
$\FF_{0}$-progressively-measurable process with values in $L^2(\SS^1;\RR^d)$, and conversely. Indeed, if
${\boldsymbol X}=(X_{t})_{0 \leq t \leq T}$ is an $\FF$-progressively-measurable $\RR^d$-valued process on $\Omega$ satisfying $\EE \int_{0}^T \vert X_{t} \vert^2 dt < \infty$, 
then it can be approximated in $L^2([0,T] \times \Omega)$ by simple processes of the form 
\begin{equation*}
\biggl( X^n_{t} = \sum_{i=0}^{n-1} X^{n,i} {\mathbf 1}_{(t_{i},t_{i+1}]}(t) \biggr)_{0 \le t \le T},
\quad n \in {\mathbb N},
\end{equation*}
where $0=t_{0}<\dots<t_{n}=T$ is a subdivision of $[0,T]$ and $X^{n,i}$, for each $i \in \{0,\cdots,n-1\}$,
is ${\mathcal F}_{t_{i}}$ measurable. Then, 
by Lemma 
\ref{lem:measurability}, we can associate with each $X^{n,i}$ an ${\mathcal F}_{0,t_{i}}$-measurable random variable 
$X^{n,i}(\cdot)$ from $\Omega_{0}$ into $\RR^d$. Letting
\begin{equation*}
\biggl( X^n_{t}(\cdot) = \sum_{i=0}^{n-1} X^{n,i}(\cdot) {\mathbf 1}_{(t_{i},t_{i+1}]}(t) \biggr)_{0 \le t \le T},
\end{equation*}
the sequence $({\boldsymbol X}^n(\cdot)=(X^n_{t}(\cdot))_{0 \le t \le T})_{n \in {\mathbb N}}$
is Cauchy in $L^2([0,T] \times \Omega_{0};L^2(\SS^1;\RR^d))$. 
The limit ${\boldsymbol X}(\cdot)=(X_{t}(\cdot))_{0 \le t \le T}$
is $\FF_{0}$-progressively-measurable and, for almost every $t \in [0,T]$,
for almost every $\omega_{0} \in \Omega_{0}$,
the realization of $X_{t}(\cdot)$ coincides with 
$\SS^1 \ni x \mapsto X_{t}(x,\omega_{0})$. 

Conversely, if we are given an $\FF_{0}$-progressively-measurable
${\boldsymbol X}(\cdot)=(X_{t}(\cdot))_{0 \le t \le T}$
 from $\Omega_{0}$ into $L^2(\SS^1;\RR^d)$
satisfying 
$\EE_{0} \int_{0}^T \|  X_{t}(\cdot) \|_{L^2(\SS^1;\RR^d)}^2 
dt < \infty$, then we can construct ${\boldsymbol X}=(X_{t})_{0 \le t \le T}$ as 
the limit in $L^2([0,T] \times \Omega;\RR^d)$ of the sequence of processes
 \begin{equation*}
\biggl( \biggl( 
(x,\omega_{0}) \mapsto
\sum_{k=1}^n X_{t}^{k,\pm}(\omega_{0}) e^{k,\pm}(x) \biggr)_{0 \leq t \leq T}
\biggr)_{n \in \NN}
 \end{equation*}
Clearly, 
 ${\boldsymbol X}=(X_{t})_{0 \le t \le T}$ is $\FF$-progressively-measurable and, 
for almost every $t \in [0,T]$,
for almost every $\omega_{0} \in \Omega_{0}$,
the realization of $X_{t}(\cdot)$ coincides with 
$\SS^1 \ni x \mapsto X_{t}(x,\omega_{0})$. 

Given processes ${\boldsymbol X}$ and ${\boldsymbol X}(\cdot)$ as we just considered, we can define 
\begin{equation*}
{\boldsymbol \chi} = \biggl( \chi_{t} = \int_{0}^t X_{s} ds \biggr)_{0 \le t \le T}, 
\quad 
\textrm{\rm and}
\quad
{\boldsymbol \chi}(\cdot) = \biggl( \chi_{t}(\cdot) = 
\sum_{n \in \NN}
\int_{0}^t X_{s}^{n,\pm} e^{n,\pm} ds
\biggr)_{0 \le t \le T}. 
\end{equation*}
Then, it is pretty easy to check that, for almost every $\omega_{0} \in \Omega_{0}$, 
for all $t \in [0,T]$, the function $\SS^1 \ni x \mapsto \chi_{t}(x,\omega_{0})$
coincides with the realization of $\chi_{t}(\cdot)$ at $\omega_{0}$.

\subsection{Randomized MFG}
\label{subsec:randomized}
%

With the same assumption as in 
\eqref{eq:square-integrability:U} for the collection of semi-martingales 
$(\bU^{n,\pm}=(U^{n,\pm}_{t})_{0 \leq t \leq T})_{n \in \NN}$,
we consider
the following (informally defined) \textit{randomized} MFG
in lieu of the original MFG presented in Subsection
\ref{subse:original:pb}:
\begin{enumerate}
\item 
Given an ${\mathcal F}_{0,0}$-measurable random variable ${\mathscr V}$ from $\Omega_{0}$ into $\cP_{2}(\R^d)$, 
with ${\mathbb E}_{0}[M_{2}({\mathscr V})^2] < \infty$, 
 and
an ${\mathbb F}_{0}$-adapted flow of random measures $\bmu=(\mu_{t})_{0 \le t \le T}$
on $\RR^d$ with continuous paths from $[0,T]$ into $\cP_{2}(\RR^d)$
such that $\PP_{0}(\mu_{0}={\mathscr V})=1$, consider the following cost functional
\begin{equation*}
\begin{split}
J^{\bmu}({\boldsymbol \alpha }) &= \int_{\Omega_{0}}
\biggl[ 
{\mathfrak g}_{0} \bigl( X_{T}(\cdot,\bw),\mu_{T}(\bw)
\bigr)
\\
&\hspace{15pt}+
\int_{0}^T \Bigl( {\mathfrak f}_{0}\bigl( X_{t}(\cdot,\bw) ,\mu_{t}(\bw)\bigr) 
+ \frac{1}{2}
\sum_{n \in \NN}
\vert \alpha^{n,\pm}_{t}(\bw) \vert^2
\Bigr) dt
\biggr] d\P_{0}(\bw),
\end{split}
\end{equation*}
over controlled dynamics of the form
\begin{equation}
\label{eq:sde}
dX_{t}^{n,\pm} = \Bigl( 
{\mathfrak b}^{n,\pm} \bigl(X_{t}(\cdot), \mu_{t} \bigr)
+ \alpha_{t}^{n,\pm}
 \Bigr)dt 
  +  dU_{t}^{n,\pm}, 
 \quad t \in [0,T], \quad n \in \NN,
\end{equation}
where $(X_{0}^{n,\pm})_{n \in \NN}$
denote the modes of 
a random variable $X_{0}(\cdot)$ with values in 
$L^2(\SS^1;\RR^d)$ such that, $\PP_{0}$-almost everywhere, 
$\textrm{\rm Leb}_{1} \circ X_{0}(\cdot)^{-1}
= {\mathscr V}$. Such a random variable exists: it suffices to 
take 
$X_{0}(\cdot) : \Omega_{0} \ni \omega_{0}
\mapsto \Psi(\xi,{\mathscr V}(\omega_{0})) \in L^2({\mathbb S}^1;\RR^d)$
(see the first lines of 
Subsection
\ref{subse:reformulation}
for the definition of $\Psi$)
and
with the same convention as above that 
$\bX^{0,-}$ is identically zero. Here 
the controls $((\alpha^{n,\pm}_{t})_{0 \leq t \leq T})_{n \in \NN}$ are required to be progressively-measurable
with respect to the filtration ${\mathbb F}_{0}$ and to satisfy:
\begin{equation}
\label{eq:bound:summable:alpha}
\sum_{n \in \NN} {\mathbb E}_{0} \int_{0}^T 
\vert \alpha_{t}^{n,\pm} \vert^2 dt < \infty.
\end{equation} 
\item Find $\bmu = (\mu_{t} : \Omega_{0} \ni \bw \mapsto \mu_{t}(\bw))_{t \in [0,T]}$ such that, with probability $1$ under $\P_{0}$, 
for all $t \in [0,T]$, 
\begin{equation}
\label{eq:matching:enlarged}
\mu_{t}(\bw) 
= \textrm{Leb}_{1} \circ X_{t}^\star(\cdot,\bw)^{-1}, 
\end{equation}
where $\bX^\star(\cdot)=(X_{t}^\star(\cdot))_{0 \leq t \leq T})$ is the optimal path 
in the optimization problem $\inf_{\balpha} J^{\bmu}(\balpha)$. 
\end{enumerate}
\vskip 4pt

\noindent Recalling 
\eqref{eq:square-integrability:any X:0}, observe that we can provide
a simple assumption on 
${\mathfrak b}$ such that, for $X$ as in 
\eqref{eq:sde}, 
\begin{equation*} 
X_{t}(\cdot) = \sum_{n \in \NN} X_{t}^{n,\pm} e^{n,\pm}(\cdot), \quad t \in [0,T],
\end{equation*}
makes sense as a process from $\Omega_{0}$ into $L^2({\mathbb S}^1;\RR^d)$. 
In this regard, \eqref{eq:sde}
just
says that each Fourier mode of the state variable $X_{t}(\cdot)$ in the space $L^2({\mathbb S}^1;\RR^d)$ is forced by the corresponding 
$(\bU^{n,\pm})_{n \in \NN}$. 
\vspace{5pt}

Of course, the choice of $(\bU^{n,\pm})_{n \in \NN}$ is the key point in our analysis. In full analogy with $\bmu$, we shall define it 
as the solution of a fixed point involving the 
optimal trajectory of the new optimization problem $\inf_{\balpha} J^{\mu}(\balpha)$ introduced right above, namely we choose 
each $\bU^{n,\pm}=(U^{n,\pm}_{t})_{0 \leq t \leq T}$ as  
\begin{equation}
\label{eq:U:as:fixed:point}
U^{n,\pm}_{t} = - (2\pi n)^2
\int_{0}^t X^{\star n,\pm}_{s} ds + W^{n,\pm}_{t}, \quad t \in [0,T], 
\quad n \in \NN.
\end{equation}
Under this choice, the optimal trajectory of 
the optimization problem $\inf_{\balpha} J^{\bmu}(\balpha)$
in environment $\bmu$ (as already explained, sufficient conditions will be given below so that an optimal path exists and is unique) takes the form:
\begin{equation}
\label{eq:sde:with:noise}
dX_{t}^{\star n,\pm} = \Bigl( {\mathfrak b}^{n,\pm}\bigl(X_{t}(\cdot),\mu_{t}\bigr) 
+ \alpha^{\star n,\pm}_{t} - (2 \pi n)^2 X_{t}^{\star n,\pm}
\Bigr) dt + dW^{n,\pm}_{t}, \quad t \in [0,T], 
\end{equation}
where $\balpha^{\star}$ is the optimal control. Here the rationale for choosing the dissipative 
factor $-(2 \pi n)^2$ in the dynamics 
is twofold. First, the fact that the series 
of the inverses of the factors, that is 
$\sum_{n \in \NN^*} (2 \pi n)^{-2}$, converge 
will permit us to prove, under suitable assumptions, 
that the modes of $\bX^\star$ are square-summable. 
Second, the
 factors 
$-(2 \pi n)^2$
appear in the formal computation:
\begin{equation*}
\partial^2_{x} X_{t}^\star(\cdot) = \sum_{n \in \NN} X_{t}^{\star n,\pm} \partial_{x}^2 e_{t}^{n,\pm}(\cdot)
= - \sum_{n \in \NN^*} (2 \pi n)^2 X_{t}^{\star n,\pm}  e_{t}^{n,\pm}(\cdot),
\end{equation*}
where $X_{t}^\star(\cdot) = \sum_{n \in \NN} X_{t}^{\star n,\pm} e^{n,\pm}(\cdot)$,
which prompts us
to reformulate  
\eqref{eq:sde:with:noise}
 as the controlled SPDE:
\begin{equation}
\label{eq:spde}
\partial_{t} X_{t}^\star(x) = b(X_{t}(x),\mu_{t})  + 
\alpha_{t}^\star(x) + \partial^2_{x} X_{t}^\star(x)  +
\dot{W}_{t}(x), \quad t \in [0,T], \quad x \in {\mathbb S}^1. 
\end{equation}
The notation $\dot{W}$ denotes a space-time white noise, namely
\begin{equation} 
\label{eq:white:noise}
W_{t}(\cdot) = \sum_{n \in \NN} W^{n,\pm}_{t} e^{n,\pm}(\cdot), 
\quad t \in [0,T], \ x \in {\mathbb S}^1,
\end{equation}
is a cylindrical Wiener process with values in $L^2({\mathbb S}^1;\RR^d)$, 
meaning that, for any $f \in L^2({\mathbb S}^1;\RR^d)$, the process 
\begin{equation*}
\biggl( \int_{{\mathbb S}^1} f(x) \cdot W_{t}(dx) 
= \sum_{n \in \NN}  f^{n,\pm}
\cdot
W^{n,\pm}_{t} 
\biggr)_{t \in [0,T]}
\end{equation*}
is a Brownian motion with $\int_{{\mathbb S}^1} \vert f(x) \vert^2 dx$
as variance.
\vspace{5pt}

So, choosing $\bU$ as in 
\eqref{eq:U:as:fixed:point}
is especially convenient for reformulating the dynamics of 
the equilibrium as the solution of an SPDE. In this regard, 
a crucial fact in the subsequent analysis will be played by the 
structure of the SPDE, which is close to that 
of an Ornstein-Ulhenbeck (OU) process with values in $L^2(\SS^1;\RR^d)$. 

If the modes of $\bX^\star(\cdot)$ satisfy
\begin{equation*}
\EE_{0} \Bigl[ \sup_{0 \leq t \leq T} \Bigl( \sum_{n \in \NN}
\vert X^{\star n,\pm}_{t} \vert^2 
\Bigr) \Bigr] < \infty, 
\end{equation*}
it is then obvious
from 
\eqref{eq:square-integrability:any X:0}, 
\eqref{eq:bound:summable:alpha},
\eqref{eq:U:as:fixed:point}
and
\eqref{eq:sde:with:noise} 
 that $\bU$ satisfy 
\eqref{eq:square-integrability:U}, which proves that 
\eqref{eq:square-integrability:any X} holds for 
any $\balpha$. 
%
%
%
%
\vspace{5pt}

In order to reconstruct the dynamics satisfied by $\bX$ for any controlled $\balpha$, we may focus on the difference $\bX - \bU$. Clearly, $\bX-\bU$
satisfies a controlled ODE with random coefficients:
\begin{equation*}
d \bigl( X_{t}^{n,\pm} - U_{t}^{n,\pm} \bigr) = 
\bigl[ 
{\mathfrak b}^{n,\pm}\bigl(X_{t}(\cdot),\mu_{t}\bigr)
+ \alpha_{t}^{n,\pm} \bigr] dt, 
\quad t \in [0,T], \quad n \in \NN,
\end{equation*}
so that
\begin{equation*}
d \bigl( X_{t} - U_{t} \bigr) = 
\bigl[ 
b(X_{t},\mu_{t})
+ \alpha_{t}  \bigr] dt, 
\quad t \in [0,T],
\end{equation*}
with $0$ as initial condition.

%
%
%
\vspace{5pt}

So, we end up with the following definition: 

\begin{definition}
\label{def:randomized:MFG}

Given a square integrable ${\mathcal F}_{0,0}$-measurable random 
variable 
$X_{0}(\cdot)$ from $\Omega_{0}$ into $L^2(\SS^1;\RR^d)$, 
we call a solution of the randomized MFG 
a pair of $\FF_{0}$-progressively measurable and 
$L^2(\SS^1;\RR^d)$-valued 
processes 
$\bX^\star (\cdot)=(X_{t}^\star(\cdot))_{0 \leq t \leq T}$, with 
$X_{0}^\star(\cdot)=X_{0}(\cdot)$ as initial condition, 
and $\balpha^\star(\cdot) = (\alpha_{t}^\star(\cdot))_{0 \leq t \leq T}$, 
satisfying the integrability conditions 
\begin{equation*}
\begin{split}
&\EE_{0} \Bigl[ \sup_{0 \leq t \leq T} \| X_{t}^\star(\cdot) \|^2
\Bigr] < \infty,
\\
&\EE_{0} \Bigl[ \int_{0}^T \| \alpha_{t}^\star(\cdot) \|^2 dt
\Bigr] < \infty,
\end{split}
\end{equation*}
and 
satisfying the system 
\eqref{eq:sde:with:noise}, 
such that, under the notations
\begin{equation*}
\begin{split}
&\mu_{t}(\bw) = \textrm{\rm Leb}_{1} \circ X_{t}^\star(\cdot,\bw)^{-1},
\quad \bw \in \Omega_{0}, 
\\
&U_{t}^{n,\pm} = - (2 \pi n)^2 
\int_{0}^t X_{s}^{\star n,\pm} ds + W_{t}^{n,\pm}, \quad t \in [0,T], \quad n \in \NN,
\end{split}
\end{equation*}
the process $\balpha^\star$ (regarded as an $\FF$-progressively measurable process with values from $\Omega$ into $\RR^d$)
is an optimal control of 
the optimal control problem with random coefficients consisting in minimizing
\begin{equation}
\label{eq:bar:sde:cost:functional}
\bar J^{\bmu}(\balpha) = 
\EE \biggl[ 
g\bigl(U_{T}+ \bar X_{T}^{\balpha},\mu_{T}\bigr) 
+ \int_{0}^T
\Bigl( f( U_{t} + \bar X_{t}^{\balpha},\mu_{t}) + \frac12 \vert \alpha_{t} \vert^2
\Bigr) dt
\biggr],
\end{equation}
over $\FF$-progressively measurable processes 
$\balpha$ satisfying 
\begin{equation*}
\EE \int_{0}^T \vert \alpha_{t} \vert^2 dt < \infty,
\end{equation*}
where 
$\bar{\bX}^{\balpha}$
solves 
\begin{equation}
\label{eq:bar:sde:control:problem}
d \bar{X}_{t}^{\balpha} = \Bigl( b\bigl( U_{t} + \bar X_{t}^{\balpha},\mu_{t}\bigr)  + \alpha_{t} \Bigr) dt, \quad t \in [0,T],
\end{equation}
with $\bar{X}_{0}^{\balpha} = X_{0}$ as initial condition ($X_{0}$ being regarded 
as an $\RR^d$-valued random variable
on $\Omega$). 
\end{definition}

\begin{remark}
Definition 
\ref{def:randomized:MFG}
provides another interpretation of the randomization of the equilibria. 
It says everything works as if we kept the same MFG as before, but 
with random coefficients obtained by 
an additive perturbation of the original ones. 
\end{remark}

\begin{remark}
\label{rem:why:not:sto}
The reader may now understand the reason why we have limited our result to the case of deterministic (instead of stochastic) 
differential equations. Our strategy is indeed clear: We enclose the private (or idiosyncratic) noise 
underpinning the initial condition of the 
representative player in the torus; the infinite dimensional noise ${\boldsymbol W}(\cdot)$ (which reads 
as a ``common noise'') then acts on the modes of the initial condition. If we had to do so
with a stochastic differential game, we should enclose the 
whole private random signal (e.g., a Brownian motion) in the torus, but, then, adaptability conditions would be a delicate issue
to handle. In fact, our guess is that, to respect the adaptability constraints, 
the forcing procedure has to be slightly different (and in fact less straightforward than it is here). 
 \end{remark}

\subsection{Infinite dimensional McKV forward-backward system}
We now observe that, for a given ${\mathbb F}_{0}$-progressively measurable random flow $\bmu=(\mu_{t})_{0 \leq t \leq T}$ as in the first item of the randomized MFG problem
defined in 
\eqref{eq:sde}--\eqref{eq:matching:enlarged}, the optimal paths (whenever they exist)  
should be given by the stochastic Pontryagin principle, 
see for instance 
\cite{Peng90,Pham_book,YongZhou}, see also 
\cite{CarmonaDelarue_book_I}. 
Here, the 
stochastic Pontryagin principle
takes the form of 
the following forward-backward system of SDEs:
\begin{equation}
\label{eq:infinite:sde}
\begin{split}
&dX_{t}^{\star,n,\pm} = \Bigl( {\mathfrak b}^{n,\pm}(
X_{t}^{\star}(\cdot),
\mu_{t}) - Y_{t}^{\star,n,\pm}
\Bigr) dt + dU_{t}^{n,\pm}, 
\\
&dY_{t}^{\star,n,\pm} = 
\Bigl( - 
\sum_{k \in \NN}
D_{n,\pm} 
{\mathfrak b}^{k,\pm}(
X_{t}^{\star}(\cdot),\mu_{t} ) Y_{t}^{\star,k,\pm}
- D_{n,\pm} {\mathfrak f}_{0}
(
X_{t}^{\star}(\cdot),\mu_{t})
\Bigr) dt  + \sum_{k \in \NN}
Z_{t}^{\star,n,k,\pm} dW_{t}^{k,\pm},
\end{split}
\end{equation}
for $t \in [0,T]$,
with the terminal condition 
$Y_{T}^{\star,n,\pm}
= D_{n,\pm} {\mathfrak g}_{0}
(X_{T}^{\star}(\cdot),\mu_{T})$, for 
all
$n \in \NN$. 
Above, 
$(X_{t}^{\star,n,\pm})_{0 \leq t \leq T}$
and
$(Y_{t}^{\star,n,\pm})_{0 \leq t \leq T}$
take values in $\R^d$ and 
$(Z_{t}^{\star,n,\pm})_{0 \leq t \leq T}$
takes values in 
$\R^{d \times d}$; also, 
we have denoted by $D$ 
the Fr\'echet derivative on $L^2({\mathbb S}^1;\RR^d)$
and by 
$D_{n,\pm} \, \bullet = \langle e^{n,\pm}(\cdot), D \, \bullet \rangle_{L^2(\SS^1;\RR^d)}$ the 
$d$-dimensional 
derivative in the direction $e^{n,\pm}$. 
Of course, in the notation $D_{n,\pm} {\mathfrak h}(\ell,\mu)$, with 
${\mathfrak h}$ matching ${\mathfrak b}^{k,\pm}$, 
${\mathfrak f}_{0}$ or 
${\mathfrak g}_{0}$, the operator $D$ acts on the first coordinate 
depending on $\ell \in L^2({\mathbb S}^1;\RR^d)$. In the notation 
$D_{n,\pm} 
{\mathfrak b}^{k,\pm}(
X_{t}^{\star}(\cdot),\mu_{t} ) Y_{t}^{\star,k,\pm}$, 
$D_{n,\pm} 
{\mathfrak b}^{k,\pm}(
X_{t}^{\star}(\cdot),\mu_{t} )$ is implicitly regarded as 
a square matrix with columns 
$(D_{n,\pm} 
{\mathfrak b}_{j}^{k,\pm}(
X_{t}^{\star}(\cdot),\mu_{t} ))_{1 \leq j \leq d}$, 
so that the whole reads 
as 
$\sum_{j=1}^d 
D_{n,\pm} 
{\mathfrak b}_{j}^{k,\pm}(
X_{t}^{\star}(\cdot),\mu_{t} ) (Y_{t}^{\star,k,\pm})_{j}$. 
We shall check properly that all the derivatives make sense in our framework. 
Lastly, 
in \eqref{eq:infinite:sde}, $X_{t}^{\star}(\cdot)$ is 
a shorten notation for the function in 
$L^2({\mathbb S}^1;\RR^d)$:
\begin{equation*}
X_{t}^{\star}(\cdot) = \sum_{n \in \NN} X_{t}^{\star,n,\pm} e^{n,\pm}(\cdot). 
\end{equation*}

For the time being, we do not establish rigorously the derivation of the stochastic Pontryagin principle. We shall address this question in Proposition 
\ref{prop:SMP}. Meanwhile, we observe that,
inserting the fixed point condition \eqref{eq:U:as:fixed:point}, 
\eqref{eq:infinite:sde}
 may be rewritten as 
\begin{equation}
\label{eq:mkv:sde:0}
\begin{split}
&dX_{t}^{\star,n,\pm} = \Bigl( {\mathfrak b}^{n,\pm}\bigl(
X_{t}^{\star}(\cdot),
\mu_{t}\bigr) - Y_{t}^{\star,n,\pm}
- (2 \pi n)^2 X_{t}^{\star,n,\pm}
\Bigr) dt + dW_{t}^{n,\pm}, 
\\
&dY_{t}^{\star,n,\pm} = 
\Bigl( - 
\sum_{k \in \NN}
D_{n,\pm} 
{\mathfrak b}^{k,\pm}\bigl(
X_{t}^{\star}(\cdot),\mu_{t} \bigr) Y_{t}^{\star,k,\pm}
- D_{n,\pm} {\mathfrak f}_{0}
\bigl(
X_{t}^{\star}(\cdot),\mu_{t}
\bigr)
\Bigr) dt  + \sum_{k \in \NN}
Z_{t}^{\star,n,k,\pm} dW_{t}^{k,\pm},
\end{split}
\end{equation}
for $t \in [0,T]$,
with the terminal condition 
$Y_{T}^{\star,n,\pm}
= D_{n,\pm} {\mathfrak g}_{0}
(X_{T}^{\star}(\cdot),\mu_{T})$, 
for all
$n \in \NN$.
\vskip 5pt

Of course, nothing guarantees \textit{a priori} that the modes 
in \eqref{eq:mkv:sde:0} are square summable. So, we impose, 
in the definition of a solution to \eqref{eq:mkv:sde:0}, that the modes are indeed square summable.

\begin{definition}
\label{def:infinite:sde}
Given a square integrable ${\mathcal F}_{0,0}$-measurable random 
variable 
$X_{0}(\cdot)$ from $\Omega_{0}$ into $L^2(\SS^1;\RR^d)$, we call a solution 
to \eqref{eq:mkv:sde:0} a countable 
collection of 
${\mathbb F}_{0}$-progressively measurable
processes 
$((X_{t}^{n,\pm})_{0\leq t \leq T})_{n \in \NN}$,
$((Y_{t}^{n,\pm})_{0 \leq t \leq T})_{n \in \NN}$,
$((Z_{t}^{n,k,\pm})_{0 \leq t \leq T})_{n, k \in \NN}$, 
such that 
\begin{equation*}
\begin{split}
&\sum_{n \in \NN}
{\mathbb E}
\Bigl[ 
\sup_{0 \leq t \leq T} 
\bigl( 
\vert X_{t}^{n,\pm} \vert^2
+
\vert Y_{t}^{n,\pm} \vert^2
\bigr)
\Bigr]
+ 
{\mathbb E}
\biggl[
\sum_{k,n \in \NN}
\int_{0}^T \vert Z_{t}^{n,k,\pm} \vert^2
dt
\biggr] < \infty,
\end{split}
\end{equation*}
satisfying, with probability 1, 
 \eqref{eq:mkv:sde:0} (and the associated terminal condition) with 
 the initial condition $X_{0}^{n,\pm}$
 for all $n \in \NN$, as given by the modes of $X_{0}(\cdot)$.   

Then, we  can define
${\mathbb F}_{0}$-adapted and continuous
 processes 
$(X_{t}(\cdot))_{0 \leq t \leq T}$ 
and $(Y_{t}(\cdot))_{0 \leq t \leq T}$
with values in $L^2({\mathbb S}^1;\RR^d)$ such that, with 
probability 1, for all $t \in [0,T]$,
\begin{equation*}
X_{t}(\cdot) = \sum_{n \in \NN}
X_{t}^{n,\pm} e^{n,\pm}(\cdot), 
\quad 
Y_{t}(\cdot) = \sum_{n \in \NN}
Y_{t}^{n,\pm} e^{n,\pm}(\cdot).
\end{equation*}
\end{definition}

Implementing the matching condition 
\eqref{eq:matching:enlarged} in the formulation of the enlarged problem, we understand that, whenever they exist, fixed points
should solve a McKean-Vlasov SDE of the conditional type.
Similar to \eqref{eq:mkv:sde:0}, this McKean-Vlasov SDE must be infinite dimensional. 
In analogy with 
\eqref{eq:original:mkv}
and with the same notation as in \eqref{eq:matching:enlarged}, it takes the form:
\begin{equation}
\label{eq:mkv:sde:1}
\begin{split}
&dX_{t}^{\star,n,\pm} = \Bigl( {\mathfrak b}^{n,\pm}\bigl(
X_{t}^{\star}(\cdot),\textrm{\rm Leb}_{1} \circ (X_{t}^{\star}(\cdot))^{-1}
\bigr) - Y_{t}^{\star,n,\pm}
- (2 \pi n)^2 X_{t}^{\star,n,\pm}
\Bigr) dt + dW_{t}^{n,\pm}, 
\\
&dY_{t}^{\star,n,\pm} = 
\Bigl( - 
\sum_{k \in \NN}
D_{n,\pm} 
{\mathfrak b}^{k,\pm}\bigl(
X_{t}^{\star}(\cdot),\textrm{\rm Leb}_{1} \circ (X_{t}^{\star}(\cdot))^{-1}\bigr) Y_{t}^{\star,k,\pm}
\\
&\hspace{50pt} - D_{n,\pm} {\mathfrak f}_{0}
\bigl(
X_{t}^{\star}(\cdot),\textrm{\rm Leb}_{1} \circ (X_{t}^{\star}(\cdot))^{-1}\bigr)
\Bigr) dt \phantom{\sum_{{}^{-}}}
 + \sum_{k \in \NN}
Z_{t}^{\star,n,k,\pm} dW_{t}^{k,\pm},
\end{split}
\end{equation}
for $t \in [0,T]$,
with the terminal condition 
$Y_{T}^{\star,n,\pm}
= D_{n,\pm} {\mathfrak g}_{0}
(X_{T}^{\star}(\cdot),\textrm{\rm Leb}_{1} \circ (X_{T}^{\star}(\cdot))^{-1})$, 
for all
$n \in \NN$.

Letting 
\begin{equation}
\label{eq:coeff:mkv:infinite}
\begin{split}
&{\mathfrak B}(\ell) = {\mathfrak b}\bigl(\ell,\textrm{Leb}_{1} \circ \ell^{-1}\bigr),
\\
&{\mathfrak H}(\ell,h) = 
\sum_{k \in \NN}
\sum_{j=1}^d
D {\mathfrak b}_{j}^{k,\pm}\bigl(\ell,\textrm{Leb}_{1} \circ \ell^{-1}\bigr) \bigl( 
h^{k,\pm} \bigr)^j
+ D {\mathfrak f}_{0}
\bigl(\ell,\textrm{Leb}_{1} \circ \ell^{-1}\bigr),
\\
&{\mathfrak G}(\ell) = D{\mathfrak g}_{0}\bigl(\ell,\textrm{Leb}_{1} \circ \ell^{-1}\bigr),\phantom{\sum_{k \in \NN}}
\end{split}
\end{equation}
for any two $\ell,h \in L^2({\mathbb S}^1;\RR^d)$, 
\eqref{eq:mkv:sde:1}
may be
 written as  
\begin{equation}
\label{eq:mkv:sde:complicated}
\begin{split}
&dX_{t}^{\star n,\pm} = \Bigl( {\mathfrak B}^{n,\pm}\bigl(
X_{t}^\star(\cdot)\bigr) - Y_{t}^{\star n,\pm}
- (2 \pi n)^2 X_{t}^{\star n,\pm}
\Bigr) dt + dW_{t}^{n,\pm}, \quad n \in \NN,
\\
&dY_{t}^{\star n,\pm} = 
- {\mathfrak H}^{n,\pm}\bigl(X_{t}^\star(\cdot),Y_{t}^\star(\cdot)\bigr)
dt + \sum_{k \in \NN}
Z_{t}^{\star n,k,\pm} dW_{t}^{k,\pm},
\end{split}
\end{equation}
for $t \in [0,T]$,
with the terminal condition 
$Y_{T}^{\star n,\pm}
= {\mathfrak G}^{n,\pm}(X_{T}^\star(\cdot))$, for all $n \in \NN$.
\vskip 5pt

This permits to give a similar definition 
to Definition \ref{def:infinite:sde}:
\begin{definition}
\label{def:mkv:sde}
Given a square integrable ${\mathcal F}_{0,0}$-measurable random 
variable 
$X_{0}(\cdot)$ from $\Omega_{0}$ into $L^2(\SS^1;\RR^d)$, we call a solution 
to \eqref{eq:mkv:sde:complicated}
 (or \eqref{eq:mkv:sde:1}), a countable 
collection of 
${\mathbb F}_{0}$-progressively measurable
processes 
$((X_{t}^{n,\pm})_{0\leq t \leq T})_{n \in \NN}$,
$((Y_{t}^{n,\pm})_{0 \leq t \leq T})_{n \in \NN}$,
$((Z_{t}^{n,k,\pm})_{0 \leq t \leq T})_{n \in \NN, k \in \NN}$, 
such that 
\begin{equation*}
\begin{split}
&\sum_{n \in \NN}
{\mathbb E}
\Bigl[ 
\sup_{0 \leq t \leq T} 
\bigl( 
\vert X_{t}^{n,\pm} \vert^2
+
\vert Y_{t}^{n,\pm} \vert^2
\bigr)
\Bigr]
+ 
{\mathbb E}
\biggl[
\sum_{k,n \in \NN}
\int_{0}^T \vert Z_{t}^{n,k,\pm} \vert^2
dt
\biggr] < \infty,
\end{split}
\end{equation*}
satisfying, with probability 1, 
 \eqref{eq:mkv:sde:complicated} (and the associated terminal condition) with 
 the initial condition $X^{\star n,\pm} = X_{0}^{n,\pm}$
 for all $n \in \NN$.  

Then, we  can define
${\mathbb F}_{0}$-adapted and continuous
 processes 
$(X_{t})_{0 \leq t \leq T}$ 
and $(Y_{t})_{0 \leq t \leq T}$
with values in $L^2({\mathbb S}^1;\RR^d)$ such that, with 
probability 1, for all $t \in [0,T]$,
\begin{equation*}
X_{t}(\cdot) = \sum_{n \in \NN}
X_{t}^{n,\pm} e^{n,\pm}(\cdot), 
\quad 
Y_{t}(\cdot) = \sum_{n \in \NN}
Y_{t}^{n,\pm} e^{n,\pm}(\cdot).
\end{equation*}
\end{definition}

\subsection{Standing assumptions}
\label{subse:standing:assumptions}

Throughout the paper, we assume that 
\vskip 4pt

\noindent {\bf Assumption (A).} The coefficient $b$ is assumed to be 
independent of $x$ and to be 
bounded and Lipschitz continuous 
on ${\mathcal P}_{2}(\R^d)$ --${\mathcal P}_{2}(\R^d)$
being equipped with the $2$-Wasserstein distance--. 
The coefficients $f$ and $g$ are differentiable in $x$, and $\partial_{x} f$ and $\partial_{x} g$ are  
 bounded and Lipschitz continuous on $\R^d \times {\mathcal P}_{2}(\R^d)$.  Moreover, 
 for any $\mu \in {\mathcal P}_{2}(\R^d)$,
 the functions $\R^d \ni x \mapsto 
f(x,\mu)$ and 
$\R^d \ni x \mapsto g(x,\mu)$ are convex.
\vskip 5pt

Importantly, notice that, under assumption \textbf{A}, the coefficients in
\eqref{eq:coeff:mkv:infinite} take the simplest form:
\begin{equation}
\label{eq:mkv:sde:coefficients}
\begin{split}
&{\mathfrak B}(\ell)  = {\mathfrak b}_{0}(\ell) e_{0}(\cdot), \quad \textrm{with}
\
{\mathfrak b}_{0}(\ell)
=
b \bigl( \textrm{Leb}_{1} \circ \ell^{-1}
\bigr),
\\
&{\mathfrak H}(\ell,h) 
= {\mathfrak F}(\ell), \quad 
\textrm{with}
\ 
{\mathfrak F}(\ell)
=
D {\mathfrak f}_{0}\bigl(\ell,\textrm{Leb}_{1} \circ \ell^{-1}
\bigr).
\end{split}
\end{equation}
In particular, the system 
\eqref{eq:mkv:sde:complicated}
becomes (removing the stars in the labels):
\begin{equation}
\label{eq:mkv:sde}
\begin{split}
&dX_{t}^{n,\pm} = \Bigl( {\mathbf 1}_{(n,\pm)=(0,+)}
{\mathfrak b}_{0}\bigl(
X_{t}(\cdot)\bigr) - Y_{t}^{n,\pm}
- (2 \pi n)^2 X_{t}^{n,\pm}
\Bigr) dt + dW_{t}^{n,\pm}, \quad n \in \NN,
\\
&dY_{t}^{n,\pm} = 
- {\mathfrak F}^{n,\pm}\bigl(X_{t}(\cdot)\bigr)
dt + \sum_{k \in \NN}
Z_{t}^{n,k,\pm} dW_{t}^{k,\pm},
\end{split}
\end{equation}
for $t \in [0,T]$,
with the terminal condition 
$Y_{T}^{n,\pm}
= {\mathfrak G}^{n,\pm}(X_{T}(\cdot))$, for all $n \in \NN$.
\vspace{5pt}

In order to fully legitimate 
the existence of the Fr\'echet derivatives 
of ${\mathfrak f}_{0}$ and ${\mathfrak g}_{0}$ 
in the direction $\ell$, we may invoke the following lemma, the proof of which is quite straightforward and is left to the reader: 
\begin{lemma}
\label{lem:diff}
For a continuously differentiable Lipschitz function 
$F : \R^d \rightarrow \R$ (so that $F$ is at most of linear growth),
define ${\mathfrak F}_{0} : 
L^2({\mathbb S}^1;\R^d) \ni \ell \mapsto \int_{{\mathbb S}^1}
F(\ell(x)) dx$. 
Then, ${\mathfrak F}_{0}$ is Fr\'echet differentiable and 
\begin{equation*}
D {\mathfrak F}_{0}(\ell)  = \nabla F \circ \ell.
\end{equation*}
\end{lemma}

In particular, we have the following expression for ${\mathfrak F}$ (and similarly for ${\mathfrak G}$):
\begin{equation*}
{\mathfrak F} : L^2(\SS^1;\R^d) \ni \ell \mapsto \bigl( 
\SS^1 \ni x \mapsto 
\partial_{x} f\bigl(\ell(x),\textrm{\rm Leb}_{1} \circ \ell^{-1} \bigr)
\bigr),
\end{equation*}
and then
\begin{equation*}
{\mathfrak F}^{n,\pm}(\ell) = \int_{\SS^1} 
\partial_{x} f\bigl(\ell(x),\textrm{\rm Leb}_{1} \circ \ell^{-1} \bigr)
\bigr) e^{n,\pm}(x) dx.
\end{equation*}

The introduction of Assumption {\bf (A)} --namely asking $b$ to be independent of $x$ and $f$
and $g$ to be convex in $x$--
is fully justified by our desire to use
the Pontryagin principle as 
a sufficient condition of optimality. Generally speaking, it requires the underlying Hamiltonian to be convex, 
which is indeed the case 
under Assumption {\bf (A)} even though it could be slightly relaxed: We could certainly allow $b$ to be linear in $x$; 
we could also think of 
allowing the derivatives of $f$ and $g$ to be at most of linear growth, but
this seems a more challenging question. 
So, 
under Assumption {\bf (A)}, 
the Pontryagin principle
is not only a necessary but also a sufficient condition for the original control problem 
 described in Subsection
\ref{subse:original:pb}; in 
particular, the 
McKean-Vlasov equation 
\eqref{eq:original:mkv}
characterizes equilibria of the original (non-randomized) mean-field game.
The following proposition is to check that this fact remains true in our randomized framework:

\begin{proposition}
\label{prop:SMP}
Given a square integrable ${\mathcal F}_{0,0}$-measurable random 
variable 
$X_{0}(\cdot)$ from $\Omega_{0}$ into $L^2(\SS^1;\RR^d)$, 
any solution 
to \eqref{eq:mkv:sde}
is a solution of the randomized matching problem
defined in Definition 
\ref{def:randomized:MFG}.
Conversely, any solution 
to the randomized matching problem
provides a solution to 
\eqref{eq:mkv:sde}.

In particular, the randomized matching problem is uniquely solvable if and only if 
the McKean-Vlasov equation
\eqref{eq:mkv:sde}
is uniquely solvable. 
\end{proposition}

\begin{proof}
\textit{First Step.}
Assume first that the McKean-Vlasov equation
\eqref{eq:mkv:sde} has a solution, which we denote by $((X_{t}^{n,\pm})_{n \in \NN},(Y_{t}^{n,\pm})_{n \in \NN},(Z_{t}^{n,k,\pm})_{n,k \in \NN})_{0 \leq t \leq T}$. Denote by $(X_{t}(\cdot))_{0 \leq t \leq T}$
and $(Y_{t}(\cdot))_{0 \leq t \leq T}$ the associated $L^2({\mathbb S}^1;\RR^d)$-valued processes as in Definition 
\ref{def:mkv:sde} and let
\begin{equation*}
\mu_{t} = \textrm{Leb}_{1} \circ X_{t}(\cdot)^{-1}, \quad t \in [0,T].  
\end{equation*}
Since the mapping $L^2({\mathbb S}^1;\RR^d) \ni \ell 
\mapsto \textrm{Leb}_{1} \circ \ell^{-1}
\in {\mathcal P}_{2}(\R^d)$ is continuous, 
each $\mu_{t}$ is a random variable with values in 
${\mathcal P}_{2}(\R^d)$ and the process 
$(\mu_{t})_{0 \leq t\leq T}$ is ${\mathbb F}_{0}$-adapted.
Following 
\eqref{eq:U:as:fixed:point},
we also let (pay attention that we dropped the symbol $\star$ in the notation for the solution of the McKean-Vlasov equation):
\begin{equation*}
U_{t}^{n,\pm} = - (2 \pi n)^2 
\int_{0}^t X_{s}^{n,\pm} ds + W_{t}^{n,\pm}, \quad t 
\in [0,T], \quad n\in \NN. 
\end{equation*}
Observe that $\bU^{n,\pm}$ is also given by
\begin{equation*}
U_{t}^{n,\pm} = X_{t}^{n,\pm} - X_{0}^{n,\pm} - \int_{0}^t 
\bigl[ {\mathbf 1}_{(n,\pm)=(0,+)}
b \bigl( \mu_{s} \bigr) ds
- Y_{s}^{n,\pm} \bigr] ds, \quad t 
\in [0,T], \quad n\in \NN,
\end{equation*}
from which we deduce that 
\begin{equation*}
\EE_{0} \Bigl[ \sup_{0 \leq t \leq T} \sum_{n \in \NN} \vert U_{t}^{n,\pm} \vert^2 \Bigr] < \infty. 
\end{equation*}
\vskip 4pt

Consider now an $\RR^d$-valued control $\balpha=(\alpha_{t})_{0 \leq t \leq T}$ as in 
\eqref{eq:bar:sde:control:problem}
and denote by $(\bar X_{t}^{\balpha})_{0 \leq t \leq T}$ the solution to \eqref{eq:bar:sde:control:problem}, namely
\begin{equation*}
d \bar X_{t}^{\balpha} = \bigl[ b(\mu_{t}) + \alpha_{t} \bigr] dt, 
\quad t \in [0,T]. 
\end{equation*}
Thanks to Lemma 
\ref{lem:measurability}, we can regard 
$\balpha$ and
$\bar \bX^{\balpha}$ as $\FF_{0}$-progressively measurable processes
$\balpha(\cdot)$ and  
$\bar \bX^{\balpha}(\cdot)$ from $\Omega_{0}$ to $L^2(\SS^1)$. 
Since $\balpha$ is fixed, we just note $\bar \bX$ for $\bar \bX^{\balpha}$. 
Then, 
the modes of $\bX(\cdot)$ satisfy:
\begin{equation*}
d \bar{X}_{t}^{n,\pm} = \bigl( 
{\mathbf 1}_{(n,\pm)=(0,+)}
b ( \mu_{t}) 
+ \alpha_{t}^{n,\pm}
 \bigr)dt,
 \quad t \in [0,T],
\end{equation*}
where $(\alpha^{n,\pm}_{t})_{0 \leq t \leq T}$ denotes the modes of 
$\balpha(\cdot)$. 
%
%
%
%
Letting $(\hat{X}_{t}^{n,\pm} = \bar X_{t}^{n,\pm} + U_{t}^{n,\pm})_{0 \leq t \leq T}$, we get 
\begin{equation*}
d
\bigl(
\hat X_{t}^{n,\pm} 
-
X_{t}^{n,\pm}
\bigr)
=  \bigl(  \alpha_{t}^{n,\pm}
+ Y_{t}^{n,\pm}
\bigr)  dt, 
 \quad t \in [0,T],
\end{equation*}
with $X_{0}^{\balpha,n,\pm} 
-
X_{0}^{n,\pm}=0$, for all $n \in \NN$.  

Now, using the notation ``$\cdot$'' for the inner product in $\R^d$,  
\begin{equation*} 
\begin{split}
d \Bigl[ Y_{t}^{n,\pm}
\cdot \bigl(
\hat X_{t}^{n,\pm} 
-
X_{t}^{n,\pm}
\bigr)
\Bigr]
&= \bigl(  \alpha_{t}^{n,\pm}
+ Y_{t}^{n,\pm}
\bigr) \cdot Y_{t}^{n,\pm} dt
 \\
 &\hspace{15pt}
-
D_{n,\pm} {\mathfrak f}_{0}\bigl(X_{t}(\cdot),\mu_{t}\bigr)
\cdot
\bigl(
\hat X_{t}^{n,\pm} 
-
X_{t}^{n,\pm}
\bigr)  dt 
+ dM_{t}^{n,\pm},
 \end{split}
\end{equation*}
where $(M_{t}^{n,\pm})_{0 \leq t \leq T}$ is a square-integrable ${\mathbb F}_{0}$-martingale. Taking  expectation, we deduce that 
\begin{equation*}
\begin{split}
&{\mathbb E}_{0}
\bigl[ D_{n,\pm} {\mathfrak g}_{0}\bigl(X_{T}(\cdot),\mu_{T}\bigr) 
\cdot \bigl(
\hat X_{T}^{n,\pm} 
-
X_{T}^{n,\pm}
\bigr)
\bigr]
\\
&\hspace{15pt}
= \E_{0} \int_{0}^T \Bigl[ \bigl(  \alpha_{t}^{n,\pm}
+ Y_{t}^{n,\pm}
\bigr) 
\cdot
Y_{t}^{n,\pm}
- D_{n,\pm} {\mathfrak f}_{0}\bigl(X_{t}(\cdot),\mu_{t}\bigr)
\cdot
\bigl(
\hat X_{t}^{n,\pm} 
-
X_{t}^{n,\pm}
\bigr) \Bigr]  dt.
\end{split}
\end{equation*}
Summing over $n \in \NN$ (which is licit in our framework), we deduce that 
\begin{equation*}
\begin{split}
&{\mathbb E}_{0}
\bigl[ \bigl\langle D {\mathfrak g}_{0}\bigl(X_{T}(\cdot),\mu_{T}\bigr), 
\bigl(
\hat X_{T}(\cdot)
-
X_{T}(\cdot)
\bigr)
\bigr\rangle_{L^2(\SS^1;\RR^d)}
\bigr]
\\
&= \E_{0} \int_{0}^T \Bigl[
\bigl\langle
 \bigl(  \alpha_{t}(\cdot)
+ Y_{t}(\cdot)
\bigr) ,
Y_{t}(\cdot)
\bigr\rangle_{L^2(\SS^1;\RR^d)}
-
\bigl\langle
 D {\mathfrak f}_{0}\bigl(X_{t}(\cdot),\mu_{t}\bigr),
\bigl(
X_{t}^{\balpha}(\cdot) 
-
X_{t}(\cdot)
\bigr)
\bigr\rangle_{L^2(\SS^1;\RR^d)}
 \Bigr]  dt,
\end{split}
\end{equation*}
where, as usual, we have let $\hat X_{t}(\cdot)
= \sum_{n \in \NN} \hat X_{t}^{n,\pm} e^{n,\pm}(\cdot)$.
Observing that, for two random variables 
$\chi(\cdot)$ and $\chi'(\cdot)$ with values in 
$L^2(\SS^1;\RR^d)$, ${\mathbb E}_{0} [\langle \chi(\cdot),\chi'(\cdot)
\rangle_{L^2(\SS^1;\RR^d}] = {\mathbb E} [ \chi \cdot \chi']$, where, 
in the last term, $\chi$
and $\chi'$ are regarded as $\RR^d$-valued random variables, 
we deduce from  Lemma 
\ref{lem:diff} that
\begin{equation*}
\begin{split}
&{\mathbb E}
\bigl[
 \partial_{x} g(X_{T},\mu_{T}) \cdot
\bigl(
\hat X_{T}
-
X_{T}
\bigr)
\bigr]
= \E \int_{0}^T \Bigl[ 
\bigl(  \alpha_{t}
+ Y_{t}
\bigr)
\cdot 
Y_{t}
-
 \partial_{x} f(X_{t},\mu_{t})
\cdot \bigl(
\hat{X}_{t} 
-
X_{t}
\bigr) \Bigr]  dt.
\end{split}
\end{equation*}
Therefore,
\begin{equation*}
\begin{split}
\bar J^{\bmu}(\balpha) 
- 
\bar  J^{\bmu}(-{\boldsymbol Y}) 
&= {\mathbb E}_{0}
\Bigl[ 
g(\hat X_{T},\mu_{T})
- 
g(X_{T},\mu_{T})
- \partial_{x} g 
(X_{T},\mu_{T})\cdot
\bigl( \hat X_{T} - X_{T}
\bigr)
\\
&\hspace{15pt}
+ \int_{0}^T 
\Bigl( \frac12 \bigl\vert \alpha_{t} + Y_{t} \bigr\vert^2
+  
f(\hat X_{t},\mu_{t})
- 
f(X_{t},\mu_{t})
- \partial_{x} f 
(X_{t},\mu_{t})
\cdot
\bigl( \hat X_{t} - X_{t}
\bigr)
\Bigr) dt 
\Bigr].
\end{split}
\end{equation*}
Since $g$ and $f$ are convex,
we deduce that the right-hand side above is non-negative, which shows that 
$-\bY$ is an optimal control for $\bar{J}^{\bmu}$, that is to 
say $\bX$ and $-\bY$ form a randomized equilibrium. 
\vskip 5pt

\textit{Second Step.}
We now turn to the converse. 
Assume that a pair $(\bX^\star(\cdot),\balpha^\star(\cdot))$ satisfies
Definition  
\ref{def:randomized:MFG}. 
Then, 
we regard
the optimization problem 
$\inf_{\balpha} \bar{J}^{\bmu}(\balpha)$
defined in  
\eqref{eq:bar:sde:cost:functional}--\eqref{eq:bar:sde:control:problem}
as a standard optimization problem in random environment. 
By the standard stochastic Pontryagin principle (up to a straightforward adaptation due to the fact that the noise is infinite dimensional),
we know that a necessary condition of optimality for 
some control process $\balpha$ --the corresponding path being denoted by 
$\bar{\bX}^{\balpha}$-- is that the solution of the adjoint backward equation 
\begin{equation}
\label{eq:sde:adjoint:smp:0} 
d \bar{Y}_{t} = - \partial_{x} f(U_{t}+\bar{X}_{t}^{\balpha},\mu_{t}) dt + \sum_{n \in \NN} Z_{t}^{n,\pm} dW_{t}^{n,\pm}, \quad t \in [0,T],
\end{equation}
with $\bar{Y}_{T} = \partial_{x} g(U_{T} + \bar{X}_{T}^{\balpha},\mu_{T})$ as terminal condition coincides with $-\balpha$, namely
\begin{equation}
\label{eq:sde:adjoint:smp} 
\bar Y_{t} = - \alpha_{t}, \quad t \in [0,T]. 
\end{equation}
Now, if, as required, we have a control process $\balpha^\star(\cdot)$ (with values in $L^2(\SS^1;\RR^d)$) with 
$\bX^\star(\cdot)$ as associated path (also with values in $L^2(\SS^1;\RR^d)$) such that 
$\balpha^\star$
(when regarded as a process with values in $\RR^d$, see Lemma
\ref{lem:measurability})
 minimizes $\bar{J}^{\bmu}$ 
 in 
 \eqref{eq:bar:sde:cost:functional}
 when $\bU(\cdot)$ is given by 
\eqref{eq:U:as:fixed:point}
and $\bmu$ by 
\eqref{eq:matching:enlarged}, then,
following the discussion right after Lemma 
\ref{lem:measurability}, 
we can  
identify
the path 
of
$\bX^\star - \bU$ (seen as an $\RR^d$-valued process on $\Omega$) 
with the path of 
$\bar \bX^{\balpha^\star}$. Also, we can define $\bY^\star$
(also seen as an $\RR^d$-valued process)
 through 
\eqref{eq:sde:adjoint:smp};
it solves an equation of the same type as 
\eqref{eq:sde:adjoint:smp:0}.
Computing the modes
of $\bX^\star(\cdot)$ and $\bY^\star(\cdot)$,  we get 
that 
$(\bX^\star(\cdot),\bY^\star(\cdot))$
is a solution of the McKean-Vlaosv equation 
\eqref{eq:mkv:sde}. 
If the latter one is at most uniquely solvable, this shows that there is 
at most one MFG equilibrium.
\end{proof}

\section{Main results}
\label{se:main}
We here expose the main results of the paper. 
Proofs will given next. 

\subsection{Existence and uniqueness}
The first main result of the paper
(whose proof is deferred to Section 
\ref{se:proof}) is 

\begin{theorem}
\label{thm:existence:uniqueness}
Under 
Assumption \textbf{\bf (A)},
\eqref{eq:mkv:sde}
is uniquely solvable for any initial
condition in the form of a  
square-integrable ${\mathcal F}_{0,0}$-measurable 
random variable  
$X_{0}(\cdot)$ from $\Omega_{0}$ to $L^2(\SS^1;\RR^d)$. 
\end{theorem}

\subsubsection*{Comparison with the case without noise}
It is worth comparing 
Theorem \ref{thm:existence:uniqueness}
with solvability results for the original mean-field game.
Existence of a solution under Assumption {\bf (A)} to 
\eqref{eq:original:mkv} was investigated by Carmona and Delarue
\cite{CarmonaDelarue_sicon}, see also 
\cite[Chapters 3 and 4]{CarmonaDelarue_book_I}, by adapting the analytical techniques developed by Lasry and Lions, see
\cite{MFG1,MFG2,MFG3,Cardaliaguet}. 
Uniqueness is known to hold under the so-called monotonicity condition due to Lasry and Lions:
\begin{enumerate}
\item $b$ is independent of the measure argument $\mu$; 
since $b$ is here assumed to be independent of $x$, it is thus constant;
\item for any two $\mu,\mu' \in {\mathcal P}_{2}(\R^d)$,
\begin{equation*}
\begin{split}
&\int_{\R^d} \bigl( 
f(x,\mu) - f(x,\mu')
\bigr) d \bigl( \mu - \mu' \bigr) (x) \geq 0,
\qquad \int_{\R^d} \bigl( 
g(x,\mu) - g(x,\mu')
\bigr) d \bigl( \mu - \mu' \bigr) (x) \geq 0.
\end{split}
\end{equation*}
\end{enumerate}
Conversely, we can provide explicit examples for which uniqueness fails under 
Assumption {\bf (A)}. 
Choose for instance $d=1$, $b \equiv 0$, $f \equiv 0$
and $g(x,\mu) = x g(\bar{\mu})$, where
$\bar{\mu}$ is understood as the mean of $\mu$ 
when $\mu \in {\mathcal P}_{2}(\R)$, 
with $g$ being non-increasing. 
Then, taking the mean in 
\eqref{eq:original:mkv}, we get
\begin{equation*}
\begin{split}
&d \E [X_{t}^{\star}] =  - \E [ Y_{t}^{\star}] dt,
\\
&d \E [ Y_{t}^{\star}] = 0,
\quad \E [ Y_{T}^\star ]
= g \bigl( 
\E [ X_{T}^\star ]
\bigr),
\end{split}
\end{equation*}
which coincides with the system of characteristics associated with 
the inviscid Burgers equation, which we alluded to in introduction:
\begin{equation*}
\partial_{x} u(t,x) - u(t,x) 
\partial_{x} u(t,x) = 0, \quad u(T,x) = g(x), \quad x \in \R. 
\end{equation*}
Choosing for instance $g(x) = - x$ for $\vert x \vert \leq 1$
and $g(x) = - \textrm{sign}(x)$ for $\vert x \vert \geq 1$, we know that uniqueness fails to the above forward-backward system when $T>1$ 
and $\E[X_{0}^\star]=0$
(it is easily checked that 
$((\E[X_{t}^\star],\E[Y_{t}^\star])=(0,0))_{0 \leq t \leq T}$, 
$((\E[X_{t}^\star],\E[Y_{t}^\star])=(t,-1))_{0 \leq t \leq T}$, 
$((\E[X_{t}^\star],\E[Y_{t}^\star])=(-t,1))_{0 \leq t \leq T}$
are solutions). This shows that 
noise in the mollified version 
\eqref{eq:mkv:sde} indeed restores uniqueness. 
\vspace{5pt}

\subsection{Master equation}
In our analysis, we shall use the fact 
that \eqref{eq:mkv:sde} is connected with some infinite dimensional PDE. 
Provided that existence and uniqueness hold true, the 
system \eqref{eq:mkv:sde} must admit a decoupling field
${\mathcal U} : [0,T] \times L^2({\mathbb S}^1;\RR^d) \rightarrow L^2({\mathbb S}^1;\R^d)$
such that, with probability 1, 
\begin{equation*}
Y_{t}(\cdot) = {\mathcal U}\bigl(t,X_{t}(\cdot)\bigr), \quad t \in [0,T],
\end{equation*}
or, equivalently,
\begin{equation*}
Y_{t}^{n,\pm} = {\mathcal U}^{n,\pm}\bigl(t,X_{t}(\cdot)\bigr), \quad t \in [0,T], \quad n\in \NN,
\end{equation*}
where $({\mathcal U}^{n,\pm})_{n \in \NN}$ denotes the Fourier modes of ${\mathcal U}$. 

Construction of the decoupling field is a standard procedure in the theory of forward-backward processes. 
We provide a short account here and we refer to \cite[Chapter 4]{CarmonaDelarue_book_I} for further details.
Given $t \in [0,T]$
and $\ell \in L^2({\mathbb S}^1;\RR^d)$, consider 
\eqref{eq:mkv:sde} but with $X_{t} = \ell$ as initial condition 
at time $t$ (or equivalently
$X_{t}^{n,\pm} = \ell^{n,\pm}$). Note the solution 
$((X_{s}^{n,\pm;t,\ell})_{n \in \NN},
(Y_{s}^{n,\pm;t,\ell})_{n \in \NN},(Z_{s}^{n,k,\pm;t,\ell})_{n,k \in \NN})_{0 \leq t \leq T}$
and define accordingly 
the processes
$(X_{s}^{t,\ell},Y_{s}^{t,\ell})_{t \leq s \leq T}$
from $\Omega_{0}$ into $L^2(\SS^1;\RR^d) \times L^2(\SS^1;\RR^d)$
as in the discussion right after Lemma \ref{lem:measurability}. 
By changing the filtration ${\mathbb F}_{0}$
into the augmented filtration generated by 
$(W^{n,\pm}_{s}-W^{n,\pm}_{t})_{n \in \NN, t \leq s \leq T
}$, we deduce that 
$Y_{t}^{t,\ell}$ is almost surely deterministic, which permits to let
\begin{equation}
\label{eq:decoupling}
{\mathcal U}(t,\ell) = Y_{t}^{t,\ell}. 
\end{equation}
Given this definition, we prove next that 
\begin{lemma}
\label{lem:decoupling:field}
For any initial condition $X_{0}(\cdot) \in L^2(\Omega_{0},{\mathcal F}_{0,0},\P_{0};L^2(\SS^1;\R^d))$, it holds, with probability 1 under $\P_{0}$,
\begin{equation}
\label{eq:repres:formula}
 Y_{t}(\cdot) ={\mathcal U}\bigl(t,X_{t}(\cdot)\bigr), \quad t \in [0,T]. 
\end{equation}
\end{lemma}
Provided that ${\mathcal U}$ is smooth enough, 
it must satisfy, by a formal application of It\^o's formula
\begin{equation*}
\begin{split}
dY_{t}^{n,\pm} &=
\biggl( \partial_{t} {\mathcal U}^{n,\pm}\bigl(t,X_{t}(\cdot)\bigr)
+ 
\Bigl\langle D {\mathcal U}^{n,\pm}\bigl(t,X_{t}(\cdot)\bigr),{\mathfrak B}\bigl(X_{t}(\cdot)\bigr) 
- Y_{t}(\cdot) + \partial^2_{x} X_{t}(\cdot)
\Bigr\rangle_{L^2({\mathbb S}^1;\RR^d)}
\\
&\hspace{45pt}
+ \frac12 \textrm{Trace}
\Bigl[ D^2 {\mathcal U}^{n,\pm}\bigl(t,X_{t}(\cdot)\bigr)
\bigr]
\biggr) dt
\\
&\hspace{15pt}
+ \Bigl\langle D {\mathcal U}^{n,\pm}\bigl(t,X_{t}(\cdot)\bigr),d W_{t}(\cdot)\Bigr\rangle_{L^2({\mathbb S}^1;\RR^d)},
\end{split}
\end{equation*}
where ${\boldsymbol W}(\cdot)$ denotes the white noise
defined in 
\eqref{eq:white:noise}.

Identifying with the backward equation 
in 
\eqref{eq:mkv:sde}, we deduce that 
${\mathcal U}$ should be a solution of the infinite dimensional system of infinite dimensional PDEs
(on $L^2({\mathbb S}^1;\RR^d)$):
\begin{equation}
\label{eq:se:2:PDE:infinite}
\begin{split}
& \partial_{t} {\mathcal U}^{n,\pm}(t,\ell)
+ 
\bigl\langle \partial^2_{x} D {\mathcal U}^{n,\pm}(t,\ell), 
  \ell
\bigr\rangle_{L^2({\mathbb S}^1;\RR^d)}
+ \frac12 \textrm{Trace}
\bigl[ D^2 {\mathcal U}^{n,\pm}(t,\ell)
\bigr]
\\
&\hspace{15pt}
+ 
\bigl\langle D {\mathcal U}^{n,\pm}(t,\ell),{\mathfrak B}(\ell)
\bigr\rangle_{L^2({\mathbb S}^1;\RR^d)}
- \bigl\langle {\mathcal U}(t,\ell), D{\mathcal U}^{n,\pm}(t,\ell) 
\bigr\rangle_{L^2({\mathbb S}^1;\RR^d)}
+ {\mathfrak F}^{n,\pm}\bigl(\ell,{\mathcal U}(t,\ell)\bigr) = 0,
\end{split}
\end{equation}
with ${\mathcal U}^{n,\pm}(T,\cdot) = {\mathfrak G}^{n,\pm}$.
The operator
\begin{equation*}
L {\mathfrak h}(\ell)
 = \bigl\langle \partial^2_{x} D {\mathfrak h}(\ell), 
  \ell
\bigr\rangle_{L^2({\mathbb S}^1;\RR^d)}
+ \frac12 \textrm{Trace}
\bigl[ D^2 {\mathfrak h}(\ell)
\bigr], \quad \ell \in L^2({\mathbb S}^1;\RR^d),
\end{equation*}
is called the Ornstein-Uhlenbeck operator
on $L^2({\mathbb S}^1;\RR^d)$ 
driven by the unbounded linear operator $\partial^2_{x}$ acting 
on $L^2({\mathbb S}^1;\RR^d)$. 
It is associated with 
the semi-group 
 $({\mathcal P}_{t})_{t \geq 0}$ 
generated by the Ornstein-Ulhenbeck process on 
$L^2({\mathbb S}^1;\RR^d)$, namely,  
for a bounded measurable function $\cV$ from $L^2({\mathbb S}^1;\R^d)$ into 
$\RR$, $\cP_{t} \cV$ maps $L^2(\SS^1;\R^d)$ into $\RR$:
\begin{equation}
\label{eq:SG}
{\mathcal P}_{t} {\mathcal V} : L^2({\mathbb S}^1) \ni 
\ell \mapsto {\mathbb E}_{0} \bigl[
{\mathcal V}(U_{t}^\ell)
\bigr], 
\end{equation}
where, for $\ell \in L^2(\SS^1;\R^d)$, 
$\bU^{\ell}(\cdot) = (U_{t}^{\ell}(\cdot))_{0 \leq t \leq T}$ is the solution of the
OU equation on $L^2(\SS^1;\RR^d)$ (constructed on $(\Omega_{0},\FF_{0},\PP_{0})$):
\begin{equation*}
dU_{t}^{\ell}(\cdot) = \partial^2_{x} U_{t}^{\ell}(\cdot) 
dt + d W_{t}(\cdot), \quad t \in [0,T] \ ; \quad U_{0}^{\ell} = \ell. 
\end{equation*}

Although there exist several results on infinite dimensional 
nonlinear PDEs (see for instance 
\cite{MR1840644,MR3674558,MR1731796}), it seems that systems of 
type \eqref{eq:se:2:PDE:infinite}
have not been considered so far. We thus prove in Section 
\ref{se:proof}
the following tailored-made solvability result:

\begin{theorem}
\label{thm:pde}
Under Assumption {\bf (A)}, the decoupling field ${\mathcal U}$
of \eqref{eq:mkv:sde} is a mild solution of 
the system of PDEs
\eqref{eq:se:2:PDE:infinite}, namely, for all $n \in \NN$:
\begin{equation*}
\label{eq:mild:pde:0}
\begin{split}
{\mathcal U}^{n,\pm}(t,\cdot)
&= {\mathcal P}_{T-t}
\Bigl(
 D_{n,\pm} {\mathfrak g}_{0}(\cdot,
\textrm{\rm Leb}_{1} \circ \cdot^{-1})
\Bigr)
\\
&\hspace{15pt}
+ \int_{t}^T 
{\mathcal P}_{s-t}
\Bigl[
 D_{n,\pm} {\mathfrak f}_{0}(\cdot,
\textrm{\rm Leb}_{1} \circ \cdot^{-1})
+
\bigl\langle
{\mathfrak B}(\cdot) - 
{\mathcal U}(s,\cdot), D {\mathcal U}^{n,\pm}(s,\cdot)
\bigr\rangle_{L^2({\mathbb S}^1;\RR^d)}
\Bigr] ds.
\end{split}
\end{equation*}
Moreover, the function ${\mathcal U}$ is Lipschitz continuous in the direction $\ell \in L^2(\SS^1;\RR^d)$, uniformly 
in time $t \in [0,T]$.
\end{theorem}

\subsubsection*{Comparison with the case without noise} Once again, it is worth comparing 
Theorem \ref{thm:pde}
with results obtained for the original mean-field game. 
Under the Lasry-Lions monotonicity condition (say with $b \equiv 0$) and appropriate
regularity assumptions on the coefficients,
it is proven in Chassagneux, Crisan and Delarue 
\cite{ChassagneuxCrisanDelarue}
(see also 
\cite{CardaliaguetDelarueLasryLions}
for the periodic case and 
\cite[Chapter 5]{CarmonaDelarue_book_II}
for another point of view on \cite{ChassagneuxCrisanDelarue})
that 
there exists a function
\begin{equation*}
V : [0,T] \times \R^d \times {\mathcal P}_{2}(\R^d) 
\rightarrow \R,
\end{equation*}
such that the function 
\begin{equation*}
[0,T] \times \R 
\times L^2({\mathbb S}^1;\R^d)
\ni (t,x,\ell) 
\mapsto V \bigl( t,x,\textrm{Leb}_{1} \circ \ell^{-1}
\bigr)
\end{equation*}
is differentiable 
 and satisfies the so-called master equation
\begin{equation}
\label{eq:master:equation:V}
\begin{split}
&\partial_{t} V(t,x,\mu)
- \frac12 \vert \partial_{x} V(t,x,\mu) \vert^2
- \int_{\R}
\partial_{\mu} V(t,x,\mu)(v) \partial_{x} V(t,v,\mu) d \mu(v) 
 + f(x,\mu) 
= 0,
\end{split}
\end{equation} 
for $(t,x,\mu) \in [0,T] \times \R^d \times \cP_{2}(\R^d)$,
with $V(T,x,\mu) = g(x,\mu)$, 
where 
$\partial_{\mu} V$ is understood as follows. 
The Fr\'echet derivative of $\ell \mapsto V(t,x,\textrm{\rm Leb}_{1} \circ \ell^{-1})$ in the direction $\ell$ takes the 
form
\begin{equation}
\label{eq:wasserstein}
D \bigl[ 
V \bigl( t,x,\textrm{Leb}_{1} \circ \cdot^{-1}
\bigr)
\bigr]_{\cdot = \ell}
= \partial_{\mu} V\bigl(t,x,\textrm{Leb}_{1} \circ \ell^{-1}\bigr)
(\ell(\cdot)),
\end{equation}
for some function $\partial_{\mu} V(t,x,\mu)(\cdot) \in L^2(\R^d,\mu;\R^d)$
with $\mu = \textrm{Leb}_{1} \circ \ell^{-1}$. It is also shown in \cite{ChassagneuxCrisanDelarue} that 
$\partial_{x} V$ and $\partial_{\mu} V$ 
are differentiable in $x$ (provided that $f$ and $g$ are sufficiently smooth). 
Therefore,
\begin{equation}
\label{eq:master:derivee}
\begin{split}
&\partial_{t} \bigl( \partial_{x}V(t,x,\mu)
\bigr)
- \partial_{x} \bigl( \partial_{x} V(t,x,\mu)
\bigr) \partial_{x} V(t,x,\mu)
\\
&\hspace{15pt}-
\int_{\R}
\partial_{x} \partial_{\mu} V(t,x,\mu)(v) \partial_{x} V(t,v,\mu)d \mu(v) 
+ \partial_{x} f(x,\mu) 
= 0,
\end{split}
\end{equation}
for $(t,x,\mu) \in [0,T] \times \R^d \times {\mathcal P}_{2}(\RR^d)$, 
with $\partial_{x} V(T,x,\mu) = \partial_{x} g(x,\mu)$. 

Define now
\begin{equation*}
{\mathcal V} : [0,T] \times L^2({\mathbb S}^1;\RR^d)
\ni (t,\ell) \mapsto 
\bigl( 
{\mathbb S}^1
\ni x \mapsto
\partial_{x} V\bigl(t,\ell(x),\textrm{Leb} \circ
\ell^{-1} \bigr)
\in \R^d \bigr) \in L^2({\mathbb S}^1;\R^d). 
\end{equation*}
Notice that the right-hand side indeed belongs to 
$L^2(\SS^1;\RR^d)$ if 
$\partial_{x} V$ is at most of linear growth in $x$, see 
the aforementioned references. 
On the model of 
\eqref{eq:se:2:PDE:infinite}, 
compute
\begin{equation*}
\begin{split}
D {\mathcal V}^{n,\pm}(t,\ell)
= D 
\Bigl( L^2({\mathbb S}^1;\RR^d) \ni h \mapsto 
 \int_{{\mathbb S}^1}
\partial_{x} V\bigl(t,h(x),\textrm{Leb} \circ
h^{-1} \bigr) e^{n,\pm}(x) dx \Bigr)_{\vert h = \ell}.
\end{split}
\end{equation*}
By \eqref{eq:wasserstein}
and following Lemma 
\ref{lem:diff} (provided again that we have enough regularity), we have
\begin{equation*}
\begin{split}
D {\mathcal V}^{n,\pm}(t,\ell)(x)
&= 
\partial_{x}^2 V\bigl(t,\ell(x),\textrm{Leb} \circ
\ell^{-1} \bigr) e^{n,\pm}(x) 
+
 \int_{{\mathbb S}^1}
\partial_{x} \partial_{\mu} V\bigl(t,\ell(v),\textrm{Leb} \circ
\ell^{-1} \bigr)\bigl( \ell(x) \bigr)  e^{n,\pm}(v) dv,
\end{split}
\end{equation*}
so that 
\begin{equation*}
\begin{split}
&\bigl\langle {\mathcal V}(t,\ell), D{\mathcal V}^{n,\pm}(t,\ell)
\bigr\rangle_{L^2({\mathbb S}^1;\RR^d)}
\\
&=
\int_{{\mathbb S}^1}
\partial_{x}^2 V\bigl(t,\ell(x),\textrm{Leb} \circ
\ell^{-1} \bigr)
\partial_{x} V\bigl(t,\ell(x),\textrm{Leb} \circ
\ell^{-1} \bigr)
 e^{n,\pm}(x) dx
\\
&\hspace{15pt} +
 \int_{{\mathbb S}^1}
 \int_{{\mathbb S}^1}
\partial_{x} \partial_{\mu} V\bigl(t,\ell(v),\textrm{Leb} \circ
\ell^{-1} \bigr)\bigl( \ell(x) \bigr)
\partial_{x} V\bigl(t,\ell(x),\textrm{Leb} \circ
\ell^{-1} \bigr)
  e^{n,\pm}(v) dv dx. 
\end{split}
\end{equation*}
Going back to 
\eqref{eq:master:derivee}, changing $x$ into $\ell(x)$
with $x \in {\mathbb S}^1$, choosing $\mu = \textrm{\rm Leb} \circ \ell^{-1}$,
multiplying by $e^{n,\pm}(x)$
and taking the integral over ${\mathbb S}^1$, we
can write
\begin{equation}
\label{eq:se:2:PDE:infinite:2}
\begin{split}
\partial_{t} {\mathcal V}^{n,\pm}(t,\ell)
- \bigl\langle {\mathcal V}(t,\ell), D{\mathcal V}^{n,\pm}(t,\ell) 
\bigr\rangle_{L^2({\mathbb S}^1;\R^d)} 
+ 
\int_{{\mathbb S}^1}
\partial_{x} f \bigl( \ell(x), \textrm{Leb}_{1} \circ \ell^{-1}
\bigr) e^{n,\pm}(x) dx 
 = 0,
\end{split}
\end{equation}
with ${\mathcal V}^{n,\pm}(T,\cdot) = {\mathfrak G}^{n,\pm}$, 
which is the \textrm{inviscid} analogue of 
\eqref{eq:se:2:PDE:infinite}.
Put it differently, 
\eqref{eq:se:2:PDE:infinite}
reads as a second-order version of 
\eqref{eq:se:2:PDE:infinite:2}; equivalently, Theorems 
\ref{thm:existence:uniqueness}
and 
\ref{thm:pde}
read as a regularization result for the master equation \textit{via} an infinite dimensional Ornstein-Ulhenbeck operator. 

\begin{remark}
\label{rem:master:equation}
The reader may wonder why, in the statement of Theorem \ref{thm:pde}, we focus on the equation satisfied by the feedback function and 
not on the equation satisfied by the value function. Indeed, it is worth noting that, in the standard theory of mean-field 
games, the so-called ``master equation'' is the equation for the value function, 
as exemplified in 
\eqref{eq:master:equation:V} (therein, $V$ identifies with the value of the mean-field game). 

In fact, the main reason is that it looks simpler. Indeed, our analysis is based upon the auxiliary
control problem 
\eqref{eq:bar:sde:cost:functional}--\eqref{eq:bar:sde:control:problem}, which is --and this is the key feature--
driven by random coefficients (not only the measure-valued process $\bmu$ is random but also 
the process ${\boldsymbol U}$ depends on $\omega_{0}$). In this framework, 
the Pontryagin principle provides a very robust approach: Except for 
the additional martingale term in the backward equation
\eqref{eq:sde:adjoint:smp:0} in the proof of Proposition 
\ref{prop:SMP}, it has a standard structure; and, in fact, the martingale structure plays almost no role 
in the overall discussion. This is the reason why we use this approach here; and, as a result, this explains why the master equation 
we get is an equation for the feedback function. 

Of course, once the feedback function is given, the value function is easily recovered. 
They are two strategies to do so. The first one is to regard the optimal cost 
$\bar J^{\bmu}(\balpha^\star)$ 
in 
\eqref{eq:bar:sde:cost:functional}
when the initial condition $(t,X_{t}^\star(\cdot))$ varies 
in $[0,T] \times L^2(\SS^1;\RR^d)$; equivalently, 
this amounts to consider 
$\int_{\RR^d} V(t,x,\mu)d\mu(x)$ in 
\eqref{eq:master:equation:V}.
Here the resulting function would satisfy 
a linear PDE on $[0,T] \times L^2(\SS^1;\RR^d)$, but the coefficients would depend on the feedback function. Pay attention that, as a mean-field game is not an optimization problem, this equation could not be regarded as an autonomous Hamilton-Jacobi-Bellman equation deriving from an optimal control problem in infinite dimension. 
Another strategy is to disentangle 
the initial state of 
$\bar {\boldsymbol X}^{\balpha}$ in \eqref{eq:bar:sde:control:problem}
from the initial condition
$X_{0}(\cdot) \in L^2(\SS^1;\RR^d)$ for ${\boldsymbol X}^\star(\cdot)$, which is exactly what is done for standard mean-field games. 
In fact, by doing so, we
first
compute,
with $\bar X_{0}^{\balpha}=x \in \RR^d$ as initial condition, 
 the optimal value  
of the optimal control problem \eqref{eq:bar:sde:cost:functional}--\eqref{eq:bar:sde:control:problem}
in the random environment formed by ${\boldsymbol X}^\star(\cdot)$; since the environment is uniquely defined in terms of 
 $X_{0}(\cdot)$ (this is Theorem 
 \ref{thm:existence:uniqueness}), the optimal value is a mere function of $x$ and $X_{0}(\cdot)$. 
 Using the same notation as in 
\eqref{eq:master:equation:V}, this should be 
``our'' $V(0,x,X_{0}(\cdot))$ (here $t=0$ because 
\eqref{eq:bar:sde:cost:functional}--\eqref{eq:bar:sde:control:problem}
is initialized at time $0$, but it is pretty easy to adapt the argument to any initial time 
$t$); then $\partial_{x} V(0,X_{0}(\cdot),X_{0}(\cdot))$ should coincide with ${\mathcal U}(0,X_{0}(\cdot))$. 
 
It is worth noting that, following the usual approach to mean-field games based on the MFG PDE system, we could directly 
address 
the optimal value  
of the optimal control problem \eqref{eq:bar:sde:cost:functional}--\eqref{eq:bar:sde:control:problem}
in an arbitrary environment
${\boldsymbol X}^\star(\cdot)$ (before we know that it is an equilibrium)
 and then look for an equilibrium by solving a fixed point obtained by plugging the 
 resulting optimal feedback in the dynamics of ${\boldsymbol X}^\star(\cdot)$. Basically, this would require to 
 write down the stochastic Hamilton-Jacobi-Bellman equation associated with 
\eqref{eq:bar:sde:cost:functional}--\eqref{eq:bar:sde:control:problem}
in the arbitrary environment ${\boldsymbol X}^\star(\cdot)$; this is the point where we feel that using the Pontryagin principle is simpler.

\end{remark}

\subsection{Interpretation as an asymptotic game}
\label{subse:connection:game:statement}

Classical
 MFGs arise as 
 asymptotic versions of games with 
 a large number of players.
 Similarly, 
 a natural question here is to
address the interpretation of the 
randomized MFG defined above 
as the limiting version of a large game (with finitely many players).  
Generally speaking, there are two ways to make the connection 
between mean-field games and finite games: 
The first one 
is to prove that equilibria of the finite games (if they do exist) 
converge to a solution of the limiting mean-field game, 
see for instance 
 \cite{CardaliaguetDelarueLasryLions}
 for the convergence of closed-loop equilibria
 and 
 \cite{Lacker_limits}
 for the convergence of open-loop equilibria; the second one is to prove that 
 any solution to the limiting game 
 induces a sequence of approximate Nash equilibria to the corresponding 
 finite games, 
 see for instance
\cite{Cardaliaguet,CarmonaDelarue_sicon,HuangCainesMalhame2}
for earlier references in that direction. It turns out that, for standard mean-field games, the second approach is (much) 
easier to implement than the first one; for that reason, this is that one that we try to adapt below, see however 
Remark 
\ref{rem:convergence:directe} about the possible implementation of the first approach. 

In comparison with the standard case, there are two main differences between our framework and the aforementioned references. The first one is that 
the limiting system is perturbed by an infinite dimensional noise, which 
should be called 
``an infinite dimensional common noise''. 
This terminology is frequently used in the theory of MFGs to emphasize the fact that 
the law of the population feels the realization of the noise,
as opposed to more standard cases where the law of the population 
is defined as the average over all the possible realizations of the noise, see for instance 
\cite{CardaliaguetDelarueLasryLions,CarmonaDelarueLacker} and the book \cite{CarmonaDelarue_book_II}. 
The second feature is the presence of local interactions due to the Laplacian in the dynamics
\eqref{eq:mkv:sde} (see also the SPDE 
\eqref{eq:spde}). 

In order to describe the corresponding finite games, we proceed as follows.
We consider $N A_{N}$ particles (with state in $\RR^d$) that are uniformly distributed 
all along the $N$ roots of unity of order $N$, 
with exactly $A_{N}$ particles per root, where 
$A_{N} \in \NN^*$.  
States of the $N A_{N}$ particles at time $t$ are denoted by 
$(X_{t}^{k,j})_{k=0,\cdots,N-1;j=1,\cdots,A_{N}}$. 
The index $k$ is understood as a label for the position (or the site) of the 
particle $(k,j)$
on the unit circle: 
it is located at point with angle $2\pi k /N$. 
In particular (and it is important for 
the sequel), the set of indices
for the location of the site
may be identified with ${\mathbb Z}/N {\mathbb Z}$; sometimes, we
thus use the notation $X_{t}^{k+\ell N,j}$
for $X_{t}^{k,j}$,
for $k \in \{0,\dots,N-1\}$ and $\ell \in {\mathbb Z}$. 
In the notation $X_{t}^{k,j}$, 
$j$ stands for the label of the particle at the site $k$, since that there are 
$A_{N}$ particles at the site $k$.

The dynamics of each particle is controlled, each particle $(k,j)$ having dynamics of the form
 \begin{equation*}
 d X_{t}^{k,j} =
 \Bigl( b \bigl( \bar{\mu}_{t}^N 
 \bigr) + \alpha_{t}^{k,j}
 + N^2
 \bigl( \bar X^{k+1}_{t} + \bar X^{k-1}_{t} - 2 \bar X^{k}_{t}
 \bigr) 
 \Bigr) dt + \sqrt N dB_{t}^{k},
 \end{equation*}
with $$\bar X^{k}_{t} = \frac1{A_{N}} \sum_{j=1}^{A_{N}} X^{k,j}_{t},$$
 and
$X^{k,j}_{0}=\bar{X}^k_{0}$ for all $j \in \{1,\cdots,A_{N}\}$, where 
 $(\bar X_{0}^{k})_{k=0,\cdots,N-1}$
 are given by the following finite volume approximation of $X_{0}(\cdot)$ (which is here assumed to be independent of 
 $\omega_{0}$):
 \begin{equation*}
 \bar X_{0}^k = N \int_{k/N}^{(k+1)/N}
 X_{0}(x) dx, \quad k =0,\cdots,N-1,
 \end{equation*}
whilst the noises 
 $(\bmf{B}^{k}= (B_{t}^{k})_{0 \leq t \leq T})_{k=0,\cdots,N-1}$ are independent $d$-dimensional Brownian motions on 
 the interval $[0,T]$ with the following definition:
 \begin{equation*}
 B_{t}^k = \sqrt{N} \int_{k/N}^{(k+1)/N}
 W_{t}(dx). 
 \end{equation*}
The random variables 
$(\bar X_{0}^{k})_{k=0,\cdots,N-1}$ are
thus constructed on 
the space $(\SS^1,{\mathcal L}(\SS^1),\textrm{\rm Leb}_{1})$
whilst the processes 
$(\bmf{B}^{k}= (B_{t}^{k})_{0 \leq t \leq T})_{k=0,\cdots,N-1}$
are constructed on the space $(\Omega_{0},\cA_{0},\PP_{0})$, as defined in 
Subsection \ref{se:1:subse:enlarged}. 

Above $\bar{\mu}^N_{t}$ is the empirical distribution 
 \begin{equation*}
 \bar{\mu}^N_{t} = 
 \frac{1}{N A_{N}}
 \sum_{k=0}^{N-1}
 \sum_{j=1}^{A_{N}}
 \delta_{X_{t}^{k,j}}. 
 \end{equation*}
Processes $(\balpha^{k,j}= (\alpha^{k,j}_{t})_{0 \leq t \leq T})_{k=0,\cdots,N-1;j=1,\cdots,A_{N}}$
 are controls with values in $\RR^d$; they are progressively-measurable 
 with respect to 
 the filtration generated by the cylindrical white noise $(W_{t}(\cdot))_{0 \le t \le T}$. 
 Controls are required to satisfy
 \begin{equation*} 
\EE \int_{0}^T  \vert \alpha_{t}^{k,j} \vert^2 dt < \infty.
\end{equation*}
\vskip 4pt

We assign to player
$(k,j)$ the following cost functional 
\begin{equation*}
J^{k,j}\bigl((\balpha^{k',j'})_{k'=0,\cdots,N-1;j'=1,\cdots,A_{N}}\bigr)
= \E \Bigl[ g\bigl(X_{T}^{k,j},\bar{\mu}^N_{T}\bigr)
+ \int_{0}^T 
\Bigl(
f\bigl(X_{t}^{k,j}, \bar{\mu}^N_{t}\bigr)
+ \frac12 \vert \alpha_{t}^{k,j} \vert^2
\Bigr) dt \Bigr].
\end{equation*}
\vspace{5pt}

Recall that we call an open-loop Nash equilibrium a tuple 
$(\balpha^{\star k,j}=(\alpha^{\star k,j}_{t})_{0 \leq t \leq T})_{k=0,\cdots,N-1;j=1,\cdots,A_{N}}$ such that, 
for any $(k_{0},j_{0}) \in \{0,\cdots,N-1\} \times \{1,\cdots,A_{N}\}$, 
for any control $\balpha^{k_{0},j_{0}} = (\alpha^{k_{0},j_{0}}_{t})_{0 \leq t \leq T}$,
$J^{k_{0},j_{0}}((\bbeta^{k,j})_{k=0,\cdots,N-1;j=1,\cdots,A_{N}}) \geq J^{k_{0},j_{0}} ((\balpha^{\star k,j})_{k=0,\cdots,N-1;j=1,\cdots,A_{N}})$, where 
$\bbeta^{k,j} = \balpha^{\star k,j}$ if $(k,j) \not = (k_{0},j_{0})$
and $\bbeta^{k_{0},j_{0}} = \balpha^{k_{0},j_{0}}$. 

The following statement shows that we can construct an approximated Nash equilibrium from the solution to problem 
\eqref{eq:mkv:sde:0} (compare for instance with 
\cite{Cardaliaguet,CarmonaDelarue_sicon,HuangCainesMalhame2}
and \cite[Chapter 6]{CarmonaDelarue_book_II}).

\begin{theorem}
\label{thm:approx:nash}
On top of 
Assumption {\bf (A)}, assume that 
$f$ and $g$ are Lipschitz continuous in $\mu$, uniformly in $x$. 
Assume also that
the sequence $(A_{N})_{N \in \NN^*}$ tends to $\infty$ with $N$. 
For a (deterministic) initial condition $X_{0}(\cdot) \in L^2(\SS^1;\RR^d)$, 
call $({\boldsymbol X}(\cdot),{\boldsymbol Y}(\cdot),{\boldsymbol Z}(\cdot))$
the solution to 
\eqref{eq:mkv:sde}.
Then, there exists a sequence of positive reals 
$(\varepsilon_{N})_{N \in \NN^*}$ converging to $0$ as $N$ tends to $\infty$ such that, with 
\begin{equation*}
\alpha^{\star k,j}_{t} = N \int_{(k-1)/N}^{k/N}
Y_{t}(x) dx, \quad t \in [0,T],
\end{equation*}
for all $k \in \{0,\cdots,N-1\}$ and $j \in \{1,\cdots,A_{N}\}$,
it holds, for any $k_{0} \in \{0,\cdots,N-1\}$ and 
$j_{0} \in \{1,\cdots,A_{N}\}$, 
and 
for any control ${\boldsymbol \alpha}^{k_{0},j_{0}}= (\alpha^{k_{0},j_{0}}_{t})_{0 \leq t \leq T}$,
\begin{equation*}
J^{k_{0},j_{0}}\bigl((\bbeta^{k,j})_{k=0,\cdots,N-1;j=1,\cdots,A_{N}}\bigr) \geq J^{k_{0},j_{0}} \bigl((\balpha^{\star k,j} _{k=0,\cdots,N-1;j=1,\cdots,A_{N}}\bigr)
- \varepsilon_{N},
\end{equation*}
where
$\bbeta^{k,j} = \balpha^{\star k,j}$ if $(k,j) \not = (k_{0},j_{0})$
and $\bbeta^{k_{0},j_{0}} = \balpha^{k_{0},j_{0}}$.
\end{theorem}

\begin{remark}
\label{rem:convergence:directe}
Theorem 
\ref{thm:approx:nash} must be regarded as a way to connect the 
problem 
\eqref{eq:mkv:sde:0} with a game of the same flavor as what appears in standard mean field game theory. 
In this regard, the assumption that $b$, $f(0,\cdot)$ and $g(0,\cdot)$ are at most of linear growth (with respect to $M_{2}(\mu)$)
is mostly for convenience. Also, 
it must be emphasized that it is not the only way to make the connection. Another way would be to construct an approximate Nash equilibrium in a closed-loop form, as usually done in mean field games. We assert that it should be indeed possible provided that we let:
\begin{equation*}
\alpha^{\star k,j}_{t} = N \int_{(k-1)/N}^{k/N}
{\mathcal U}\bigl(t,\bar{X}_{t}(\cdot)\bigr)(x) dx, \quad t \in [0,T],
\end{equation*}
with the notation
\begin{equation*}
\bar{X}_{t}(\cdot) = \sum_{k=0}^{N-1} \bar{X}_{t}^k {\mathbf 1}_{[k/N,(k+1)/N)}(\cdot)
=
\frac1N
\sum_{k=0}^{N-1}
\sum_{j=1}^N  X_{t}^{k,j} {\mathbf 1}_{[k/N,(k+1)/N)}(\cdot), \quad t \in [0,T],
\end{equation*}
which means that 
\begin{equation*}
 d X_{t}^{k,j} =
 \biggl( b \bigl( \bar{\mu}_{t}^N 
 \bigr) + 
 N \int_{(k-1)/N}^{k/N}
{\mathcal U}\bigl(t,\bar{X}_{t}(\cdot)\bigr)(x) dx
 + N^2
 \bigl( \bar X^{k+1}_{t} + \bar X^{k-1}_{t} - 2 \bar X^{k}_{t}
 \bigr) 
 \biggr) dt + \sqrt N dB_{t}^{k}.
 \end{equation*}

As the paper is already quite long, we feel better to focus on the construction of an approximated Nash equilibrium over open-loop form controls only, which is in fact slightly simpler. 

Another strategy would be to address the convergence of the Nash equilibria of the finite player game (if they do exist) to the solution of 
\eqref{eq:mkv:sde:0}. Describing the dynamics of the equilibria to the finite player game by means of Pontryagin's principle and then using the master equation 
\eqref{eq:se:2:PDE:infinite}, we could indeed implement the same strategy as that used in 
\cite{CardaliaguetDelarueLasryLions} for standard mean field games, but this would require first to improve Theorem 
\ref{thm:pde}
and to prove further regularity properties of ${\mathcal U}$. Again, we feel better to postpone this equation to further works. 

Last, we mention that the condition $A_{N} \rightarrow \infty$ is absolutely crucial. It is must be regarded as a way to freeze the influence of the local interaction in the dynamics between the particles; this is the key fact to restore a mean field limit despite the local interactions. 
\end{remark}

\section{Proofs of Theorems 
\ref{thm:existence:uniqueness}
and 
\ref{thm:pde}}
\label{se:proof}

We now prove Theorems 
\ref{thm:existence:uniqueness}
and 
\ref{thm:pde}.

\subsection{Small time analysis}
We start with the case when $T$ is small enough. 

\begin{theorem}
\label{eq:small:time}
There exists a constant $c$, only depending on the Lipschitz constant of the coefficients 
${\mathfrak b}_{0}$, ${\mathfrak F} = D {\mathfrak f}_{0}$
and ${\mathfrak G} = D {\mathfrak g}_{0} $ such that, 
for $T \leq c$, the system 
\eqref{eq:mkv:sde}
is uniquely solvable for any initial 
condition 
$X_{0}(\cdot) \in L^2(\Omega_{0},{\mathcal F}_{0,0},\P_{0};L^2(\SS^1;\R^d))$. 
This permits to define the decoupling field 
${\mathcal U}$ 
as in 
\eqref{eq:decoupling}. 
It maps 
$L^2({\mathbb S}^1;\R^d)$ into itself.
Then, there exists a constant $\Lambda$, only depending 
on the bound of the coefficients
${\mathfrak b}_{0}$, ${\mathfrak F} = D {\mathfrak f}_{0}$ such that, 
for $T \leq c$,
\begin{equation*}
\sup_{0 \leq t \leq T}
\sup_{\ell \in L^2({\mathbb S}^1;\R^d)}
\| {\mathcal U}(t,\ell) 
\|_{L^2({\mathbb S}^1;\R^d)} \leq 
\sup_{\ell \in L^2({\mathbb S}^1;\R^d)}
\| {\mathfrak G}(\ell) 
\|_{L^2({\mathbb S}^1;\R^d)}
+ \Lambda T^2. 
\end{equation*}
Moreover, there exists a constant $C$, only depending 
on the Lipschitz constant of the coefficients 
${\mathfrak b}_{0}$, ${\mathfrak F} = D {\mathfrak f}_{0}$
and ${\mathfrak G} = D {\mathfrak g}_{0}$ such that, 
for $T \leq c$, 
for any $t \in [0,T]$,
${\mathcal U}(t,\cdot)$ is $C$ Lipschitz continuous. In particular, 
${\mathcal U}$ satisfies Lemma \ref{lem:decoupling:field}.
\end{theorem}

\begin{remark}
\label{rem:changement:condition:terminale:T:petit}
We let the reader check that the above result remains true if ${\mathfrak G}$ is not given as
the gradient of ${\mathfrak g}_{0}$, but is a general bounded and Lipschitz continuous function
 from $L^2(\SS^1;\RR^d)$ into itself. 
 \end{remark}

\begin{proof}
The proof is quite standard in the finite dimensional 
framework. We give the sketch of it, insisting on 
the differences between the infinite-dimensional and finite-dimensional cases. 
\vskip 5pt

\textit{First step.} Existence and uniqueness in small time follow from the application of Picard's fixed point theorem. 
We consider 
the space ${\mathcal S}$ of 
processes $(\bX(\cdot),\bY(\cdot))=(X_{t}(\cdot),Y_{t}(\cdot))_{0 \leq t \leq T}$ 
with values in $L^2({\mathbb S}^1;\RR^d) \times 
L^2({\mathbb S}^1;\RR^d)$,
that are
${\mathbb F}_{0}$-adapted
with continuous paths 
and that satisfy
\begin{equation*}
{\mathbb E}_{0}
\bigl[
\sup_{0 \leq t \leq T}
\bigl(
\| X_{t}(\cdot) \|_{L^2(\SS^1;\RR^d)}^2
+ 
\| Y_{t}(\cdot) \|_{L^2(\SS^1;\RR^d)}^2
\bigr) 
\bigr] < \infty. 
\end{equation*}
Given the initial condition $X_{0}(\cdot) \in L^2(\Omega_{0},{\mathcal F}_{0,0},\P_{0};L^2(\SS^1;\R^d))$, 
we then call $\Phi$ the function that maps 
$(\bX(\cdot),\bY(\cdot))=(X_{t}(\cdot),Y_{t}(\cdot))_{0 \leq t \leq T}$ onto 
the pair 
$(\tilde{\bX}(\cdot),\tilde{\bY}(\cdot)) = (\tilde{X}_{t}(\cdot),\tilde{Y}_{t}(\cdot))_{0 \leq t \leq T}$
satisfying 
\begin{equation*}
\begin{split}
&d\tilde X_{t}^{n,\pm}
= 
\Bigl( 
{\mathbf 1}_{(n,\pm)=(0,+)}
{\mathfrak b}_{0}\bigl(\tilde X_{t}(\cdot)\bigr) 
- Y_{t}^{n,\pm} - (2 \pi n)^2 
 \tilde X_{t}^{n,\pm}
 \Bigr) dt + dW_{t}^{n,\pm},
 \\
 &d \tilde{Y}_{t}^{n,\pm}
 = 
 - D_{n,\pm} {\mathfrak f}_{0}\bigl(X_{t}(\cdot),\textrm{Leb}_{1}
 \circ X_{t}(\cdot)^{-1}\bigr) dt 
 + \sum_{k \in \NN} \tilde{Z}^{n,k,\pm} dW_{t}^{k,\pm},
 \end{split}
\end{equation*}
with the terminal condition 
$\tilde{Y}_{T}^{n,\pm}
=
D_{n,\pm} {\mathfrak g}_{0}(X_{T}(\cdot),\textrm{Leb}_{1}
 \circ X_{T}(\cdot)^{-1})$. 
Obviously,  
the backward equation may be rewritten under the form:
 \begin{equation*}
 \begin{split}
 \tilde{Y}_{t}^{n,\pm}
&= {\mathbb E}_{0}
\Bigl[ 
 D_{n,\pm} {\mathfrak g}_{0}
 \bigl(X_{T}(\cdot), \textrm{Leb}_{1}
 \circ X_{T}(\cdot)^{-1}\bigr)
+ \int_{t}^T
D_{n,\pm} {\mathfrak f}_{0}
\bigl(X_{s}(\cdot),\textrm{Leb}_{1}
 \circ X_{s}(\cdot)^{-1}\bigr)
ds
\, 
\big\vert \, {\mathcal F}_{0,t}
\Bigr]. 
\end{split}
 \end{equation*}
 Taking the square and summing over $n \in \NN$, we deduce
 that
 \begin{equation*}
 \begin{split}
\sum_{n \in \NN}
 \vert \tilde{Y}^{n,\pm}_{t}
 \vert^2
 &
 \leq 
 \sum_{n \in \NN}
 \E_{0}
 \Bigl[ 
  \bigl\vert 
 D_{n,\pm} {\mathfrak g}_{0}
 \bigl(X_{T}(\cdot), \textrm{Leb}_{1}
 \circ X_{T}(\cdot)^{-1}\bigr)
 \bigr\vert^2 
 \\
 &\hspace{15pt} + T 
 \int_{t}^T
  \bigl\vert
 D_{n,\pm} {\mathfrak f}_{0}
 \bigl(X_{s}(\cdot), \textrm{Leb}_{1}
 \circ X_{s}(\cdot)^{-1}\bigr)
 \bigr\vert^2 
 ds \, \big\vert \, {\mathcal F}_{0,t}
 \Bigr].
 \end{split}
\end{equation*}
Since 
$D {\mathfrak f}_{0}(\cdot,\textrm{Leb}_{1} \circ \cdot^{-1})$
and
$D {\mathfrak g}_{0}(\cdot,\textrm{Leb}_{1} \circ \cdot^{-1})$
are bounded, we deduce that 
 \begin{equation}
 \label{eq:bound:Y:L2}
\sum_{n \in \NN}
 \vert \tilde{Y}^{n,\pm}_{t}
 \vert^2
 \leq \sup_{\ell \in L^2({\mathbb S}^1;\R^d)}
\|
 D {\mathfrak g}_{0}(\ell,\textrm{Leb}_{1} \circ \ell^{-1})
\|_{L^2({\mathbb S}^1;\RR^d)}
+ \Lambda T^2,
\end{equation}
for some deterministic $\Lambda \geq 0$. 

Consider now 
another input $
(\bX'(\cdot),\bY'(\cdot))
=(X_{t}'(\cdot),Y_{t}'(\cdot))_{0 \leq t \leq T}$
in ${\mathcal S}$ and call 
$(\tilde \bX{}'(\cdot),\tilde \bY{}'(\cdot))=(\tilde X_{t}'(\cdot),\tilde Y_{t}'(\cdot))_{0 \leq t \leq T}$
its image by $\Phi$. 
By the same argument as above, using in addition 
B\"urkholder-Davis-Gundy inequalities,
we get 
\begin{equation*}
\begin{split}
&{\mathbb E}_{0}
\bigl[ \sup_{0 \leq t \leq T}
\| \tilde{Y}_{t}(\cdot) - \tilde{Y}_{t}'(\cdot)
\|_{L^2({\mathbb S}^1;\R^d)}^2
\bigr]
\\
&\leq 
{\mathbb E}_{0}
\Bigl[ 
\bigl\| D {\mathfrak g}_{0}
\bigl({X}_{T}(\cdot),\textrm{Leb}_{1} \circ X_{T}(\cdot)^{-1}\bigr)
-
D {\mathfrak g}_{0}
\bigl({X}_{T}'(\cdot),\textrm{Leb}_{1} \circ X_{T}'(\cdot)^{-1}\bigr)
\bigr\|_{L^2({\mathbb S}^1;\R^d)}^2
\Bigr]
\\
&\hspace{15pt} + T
\int_{0}^T
{\mathbb E}_{0}
\Bigl[ 
\bigl\| D {\mathfrak f}_{0}
\bigl({X}_{s}(\cdot),\textrm{Leb}_{1} \circ X_{s}(\cdot)^{-1}\bigr)
-
D {\mathfrak f}_{0}
\bigl({X}_{s}'(\cdot),\textrm{Leb}_{1} \circ X_{s}'(\cdot)^{-1}\bigr)
\bigr\|_{L^2({\mathbb S}^1;\R^d)}^2
\Bigr]
ds.
\end{split}
\end{equation*}
Observe that 
$D {\mathfrak f}_{0}$ 
and 
$D {\mathfrak g}_{0}$ 
are Lipschitz continuous (from $L^2(\SS^1;\R^d)$ into itself). 
Deduce 
that there exists a constant $C \geq 0$,
only depending on the Lipschitz constants of the coefficients, 
such that, for $T \leq 1$, 
\begin{equation}
\label{eq:fixed:point:1}
\begin{split}
&{\mathbb E}_{0}
\bigl[ \sup_{0 \leq t \leq T}
\| \tilde{Y}_{t}(\cdot) - \tilde{Y}_{t}'(\cdot)
\|_{L^2({\mathbb S}^1;\R^d)}^2
\bigr]
\leq 
C
\sup_{0 \leq t \leq T}
{\mathbb E}_{0}
\bigl[ 
\| {X}_{t}(\cdot)  - {X}_{t}'(\cdot)
\|_{L^2({\mathbb S}^1;\R^d)}^2
\bigr]. 
\end{split}
\end{equation}
Proceeding in a similar way with the forward equation
and using the fact that the factor $(2 \pi n)^2$ in the dynamics is 
affected with a sign minus (so that it is a friction term), we get 
\begin{equation}
\label{eq:fixed:point:2}
\begin{split}
&{\mathbb E}_{0}
\bigl[ \sup_{0 \leq t \leq T}
\| \tilde{X}_{t}(\cdot) - \tilde{X}_{t}'(\cdot)
\|_{L^2({\mathbb S}^1;\R^d)}^2
\bigr]
\leq 
C T
\sup_{0 \leq t \leq T}
{\mathbb E}_{0}
\bigl[ 
\| {Y}_{t}(\cdot) - {Y}_{t}'(\cdot)
\|_{L^2({\mathbb S}^1;\R^d)}^2
\bigr]. 
\end{split}
\end{equation}
We easily deduce that 
$\Phi$ is a contraction in small time, 
which shows the existence of a unique fixed point. 
This shows that the system \eqref{eq:mkv:sde}
is uniquely solvable when $T \leq c$, for a constant 
$c$ that only depends on the Lipschitz constants of the coefficients. 
\vskip 4pt

\textit{Second step.}
Now that existence and uniqueness are known to hold true, 
we can define the decoupling field ${\mathcal U}$ in a standard way.
The key point is to observe that the system
\eqref{eq:mkv:sde}, when regarded under the initial condition 
$X_{t} = \ell$ at time $t \in [0,T]$ for some $\ell \in L^2(\SS^1;\R^d)$,
is also uniquely solvable when $T \leq c$
and that its solution, denoted by $((\bX^{t,\ell,n,\pm}=(X^{t,\ell,n,\pm}_{s})_{t \leq s \leq T})_{n \in \NN},(\bY^{t,\ell,n,\pm}=(Y^{t,\ell,n,\pm}_{s})_{t \leq s \leq T})_{n \in \NN},(\bZ^{t,\ell,n,\pm,k,\pm}=(Z^{t,\ell,n,\pm,k,\pm}_{s})_{t \leq s \leq T})_{n \in \NN,k \in \NN})$ is adapted with respect to 
the completion of the filtration generated by the collection of Wiener processes 
$((W^0_{s}-W^0_{t})_{t \leq s \leq T},((W^{n,\pm}_{s}- W^{n,\pm}_{t})_{t \leq s \leq T})_{n \in \NN^*})$. 
In particular, for each $n \in \NN$, the random variable $Y_{t}^{n,\pm,t,\ell}$
is almost surely deterministic. We then let 
\begin{equation*}
\cU^{n,\pm}(t,\ell) = Y_{t}^{t,\ell,n,\pm},
\end{equation*}
and 
\begin{equation*}
\cU(t,\ell) = \sum_{n \in \NN} \cU^{n,\pm}(t,\ell) e^{n,\pm}(\cdot)
\in L^2(\SS^1;\RR^d), \quad t \in [0,T], \quad \ell \in L^2(\SS^1;\R^d).
\end{equation*}
The bound for $\cU$ is a straightforward consequence of 
 \eqref{eq:bound:Y:L2}. 

As for the Lispchitz constant of ${\mathcal U}$, it follows again from a straightforward adaptation 
of 
\eqref{eq:fixed:point:1} and 
\eqref{eq:fixed:point:2}.
Indeed, for any two solutions 
$({\boldsymbol X}(\cdot),{\boldsymbol Y}(\cdot))$
and
$({\boldsymbol X}'(\cdot),{\boldsymbol Y}'(\cdot))$
to \eqref{eq:mkv:sde}, we have
\begin{equation*}
\begin{split}
&{\mathbb E}_{0}
\bigl[ \sup_{0 \leq t \leq T}
\| {Y}_{t}(\cdot) - {Y}_{t}'(\cdot)
\|_{L^2({\mathbb S}^1;\R^d)}^2
\bigr]
\\
&\leq 
C
\sup_{0 \leq t \leq T}
{\mathbb E}_{0}
\bigl[ 
\| {X}_{t}(\cdot)  - {X}_{t}'(\cdot)
\|_{L^2({\mathbb S}^1;\R^d)}^2
\bigr]
\\
&\leq 
C 
\Bigl( 
{\mathbb E}_{0}
\bigl[
\| {X}_{0}(\cdot)  - {X}_{0}'(\cdot)
\|_{L^2({\mathbb S}^1;\R^d)}^2
\bigr]
+
T
\sup_{0 \leq t \leq T}
{\mathbb E}_{0}
\bigl[ 
\| {Y}_{t}(\cdot) - {Y}_{t}'(\cdot)
\|_{L^2({\mathbb S}^1;\R^d)}^2
\bigr] 
\Bigr),
\end{split}
\end{equation*}
and then, for $T$ small enough,
\begin{equation}
\label{eq:FB:stability}
\begin{split}
&{\mathbb E}_{0}
\bigl[ \sup_{0 \leq t \leq T}
\| {Y}_{t}(\cdot) - {Y}_{t}'(\cdot)
\|_{L^2({\mathbb S}^1;\R^d)}^2
\bigr]
\leq 
C
{\mathbb E}_{0}
\bigl[
\| {X}_{0}(\cdot)  - {X}_{0}'(\cdot)
\|_{L^2({\mathbb S}^1;\R^d)}^2
\bigr].
\end{split}
\end{equation}
By performing the analysis on the interval $[t,T]$ 
instead of $[0,T]$ and by choosing $\bX(\cdot) = \bX^{t,\ell}(\cdot)= \sum_{n \in \NN}
\bX^{t,\ell,n,\pm} e^{n,\pm}(\cdot)$ and 
$\bX'(\cdot) = \bX^{t,\ell'}(\cdot)= \sum_{n \in \NN}
\bX^{t,\ell',n,\pm} e^{n,\pm}(\cdot)$
for two $\ell,\ell' \in L^2(\SS^1;\R^d)$, we deduce that ${\mathcal U}$ is Lipschitz continuous in the space variable.  

It remains to check that Lemma 
\ref{lem:decoupling:field} is satisfied. 
The argument is standard in the finite dimensional case, see for instance 
\cite{Delarue02};  
as for the infinite dimensional case, we refer 
to \cite[Chapter 5]{CarmonaDelarue_book_II}. So, we just provide a sketch of the proof. 
In fact, by regarding $t$ in the formula
\eqref{eq:repres:formula}
 as 
the initial time of the forward process, 
it suffices to focus on the case $t=0$ and to prove that, for any $X_{0}(\cdot) \in L^2(\Omega, {\mathcal F}_{0,0},\PP;L^2(\SS^1;\RR^d))$,
the unique solution 
$({\boldsymbol X}(\cdot),{\boldsymbol Y}(\cdot))$
to \eqref{eq:mkv:sde} satisfies
\begin{equation*}
Y_{0}(\cdot) = {\mathcal U}\bigl(0,X_{0}(\cdot) \bigr),
\end{equation*}
which is already known to be true 
when $X_{0}(\cdot)$ is deterministic, that is $X_{0}(\cdot) = \ell \in L^2(\SS^1;\RR^d)$.
It is easily checked that it remains true when $X_{0}(\cdot)$ is 
a random variable of the form 
\begin{equation}
\label{eq:discrete}
X_{0}(\cdot) = \sum_{i=1}^n {\mathbf 1}_{A_{i}} \ell_{i},
\end{equation} 
with $A_{i} \in {\mathcal F}_{0,0}$ and $\ell_{i} \in L^2(\SS^1;\RR^d)$ for all 
$i \in \{1,\cdots,n\}$; indeed, in that case, $Y_{0}(\cdot) =  
\sum_{i=1}^n {\mathbf 1}_{A_{i}} Y_{0}^{0,\ell_{i}}(\cdot)$. 
When the support of 
the law of $X_{0}(\cdot)$ is included in a compact subset of $L^2(\SS^1;\RR^d)$, 
we can approximate $X_{0}(\cdot)$ in $L^2(\Omega, {\mathcal F}_{0,0},\PP;L^2(\SS^1;\RR^d))$
by a sequence 
of random variables of the form \eqref{eq:discrete}. Using the 
fact that the representation formula \eqref{eq:repres:formula} holds true along the 
approximation sequence and 
using the
stability property \eqref{eq:FB:stability}, we deduce that the representation formula holds true when 
the law of $X_{0}(\cdot)$ is compactly supported. When $X_{0}(\cdot)$ is 
a general element in $L^2(\Omega, {\mathcal F}_{0,0},\PP;L^2(\SS^1;\RR^d))$, 
we can play the same game: We
can approximate $X_{0}(\cdot)$ by a sequence of compactly supported initial conditions of the form 
$(\sum_{k=0}^n \vartheta_{n}(X_{0}^{k,\pm}) e^{k,\pm})_{n \in \NN}$, where 
$(\vartheta_{n})_{n \in \NN}$ is a sequence of cut-off functions from $\RR^d$ into itself converging to the identity
uniformly on compact sets. 
\end{proof}

\subsection{Road map to existence and uniqueness in arbitrary time}
Our strategy for proving existence and uniqueness in arbitrary time is completely inspired from the finite dimensional case. 
The point is to apply iteratively Theorem 
\ref{eq:small:time} and to provide an \textit{a priori} bound for the Lipschitz constant of the decoupling field $\cU$ that holds true all along the induction. We refer to \cite{Delarue02} for a complete description of the induction procedure in the finite dimensional case. 

\subsubsection*{Change of measure}
Below, we mostly focus on the derivation of the \textit{a priori} bound 
for the Lipschitz constant of $\cU$. We start with the following observation.
For $T \leq c$ as in the statement 
of Theorem 
\ref{eq:small:time}, we can define the probability $\tilde{\P}_{0}$ 
on $\Omega_{0}$
by 
\begin{equation*}
\begin{split}
\frac{d \tilde{\P}_{0}}
{d \P_{0}}
&= \exp \biggl( - \int_{0}^T
\bigl\langle {\mathfrak B}(X_{t}(\cdot)) - 
Y_{t}(\cdot) \bigr) , dW_{t}
\bigr\rangle_{L^2(\SS^1;\R^d)}
- \frac{1}{2}
\int_{0}^T 
\bigl\| {\mathfrak B}(X_{t}(\cdot)) - 
Y_{t}(\cdot) \bigr\|_{L^2({\mathbb S}^1;\R^d)}^2
dt \biggr)
\\
&=
\exp \biggl( 
- \sum_{n \in \NN}
\int_{0}^T 
\bigl( 
{\mathbf 1}_{(n,\pm)=(0,+)} {\mathfrak b}_{0}(X_{t}(\cdot)) - 
Y_{t}^{n,\pm} \bigr) \cdot dW_{t}^{n,\pm}
\\
&\hspace{45pt} - 
\frac12 \sum_{n \in \NN}
\int_{0}^T 
\bigl\vert 
{\mathbf 1}_{(n,\pm)=(0,+)}
{\mathfrak b}_{0}(X_{t}(\cdot)) - 
Y_{t}^{n,\pm} \bigr\vert^2 dt
\biggr),
\end{split}
\end{equation*}
where ${\mathfrak B}$
is as in \eqref{eq:coeff:mkv:infinite}. 
Since 
${\mathfrak b}_{0}$
is bounded
and 
$\bY$ satisfies
 \eqref{eq:bound:Y:L2}, $\t \PP_{0}$ is a probability measure equivalent to $\PP_{0}$. Observe in particular that, for any $p \geq 1$,
 \begin{equation}
 \label{eq:bound:Girsanov:density}
\EE_{0}
\Bigl[ \Bigl( \frac{d \tilde{\PP}_{0}}{d \PP_{0}} \Bigr)^p \Bigr] < \infty.
\end{equation}
Of course, the bound  \eqref{eq:bound:Y:L2} remains true
under $\tilde{\PP}_{0}$. Observe also from the identity
\begin{equation*}
\sum_{k \in \NN}
\int_{0}^T Z^{n,k,\pm}_{s} dW_{s}^{k,\pm}
= Y_{T}^{n,\pm} - Y_{0}^{n,\pm} + \int_{0}^T
D {\mathfrak f}_{0}^{n,\pm}\bigl(X_{t}(\cdot),\textrm{\rm Leb}_{1}
\circ X_{t}(\cdot)^{-1}\bigr) dt
\end{equation*}
that, for any $p \geq 1$, 
\begin{equation*}
\EE_{0} \biggl[ \biggl( \sum_{k \in \NN} \int_{0}^T \vert Z^{n,k,\pm}_{s} \vert^2 
ds\biggr)^p \biggr] < \infty.
\end{equation*}
By  \eqref{eq:bound:Girsanov:density}, the same is true under $\t \PP_{0}$, that is 
\begin{equation}
\label{eq:bound:dZ:Girsanov}
\t \EE_{0} \biggl[ \biggl( \sum_{k \in \NN} \int_{0}^T \vert Z^{n,k,\pm}_{s} \vert^2 
ds\biggr)^p \biggr] < \infty.
\end{equation}
Now, we let 
\begin{equation*}
\tilde{W}^{n,\pm}_{t} = W^{n,\pm}_{t} + \int_{0}^t
\bigl( {\mathbf 1}_{(n,\pm)=(0,+)} {\mathfrak b}_{0}(X_{s}(\cdot)) - 
Y_{s}^{n,\pm} \bigr)
ds, \quad t \in [0,T].
\end{equation*}
Under $\tilde{\PP}_{0}$, 
the processes 
$((\tilde{W}^{n,\pm}_{t})_{0 \leq t \leq T})_{n \in \NN}$
are independent Brownian motions and
the forward component of 
the solution
to \eqref{eq:mkv:sde}
satisfies 
\begin{equation*}
d {X}_{t}^{n,\pm}
= - (2 \pi n)^2 {X}_{t}^{n,\pm}
dt + d\tilde{W}_{t}^n, \quad t \in [0,T],
\end{equation*}
and is thus an Ornstein-Ulhenbeck process, with $X_{0}^{n,\pm}$ as initial condition. Also, under $\tilde{\P}_{0}$, the backward equation takes 
the form:
\begin{equation}
\label{eq:girsanov:bsde}
\begin{split}
d Y^{n,\pm}_{t}
&= 
\Bigl[ 
- D_{n,\pm} {\mathfrak f}_{0}\bigl(X_{t}(\cdot) ,
\textrm{Leb}_{1} \circ X_{t}(\cdot)^{-1}
\bigr) 
- 
\sum_{k \in \NN}
Z_{t}^{n,k,\pm}
\bigl(
{\mathbf 1}_{(k,\pm)=(0,+)}
{\mathfrak b}_{0}
\bigl(X_{t}(\cdot)\bigr) - Y_{t}^{k,\pm}
\bigr) 
\Bigr]
dt 
\\
&\hspace{15pt} + 
\sum_{k \in \NN}
Z_{t}^{n,k,\pm} d \tilde{W}_{t}^{k,\pm}.
\end{split}
\end{equation}
By \eqref{eq:bound:dZ:Girsanov}, the drift has finite moments of any order under $\t \PP_{0}$.

According to the standard theory of
backward SDEs (or, equivalently, by a formal application of It\^o's formula), we expect
\begin{equation}
\label{eq:gradient}
Z_{t}^{n,k,\pm}
= D_{k,\pm} {\mathcal U}^{n,\pm}(t,X_{t}(\cdot))
\quad  \P_{0} \ (\textrm{or} \ \tilde{\P}_{0}) \ \textrm{almost everywhere}. 
\end{equation}
Initializing the process $(X_{s})_{0 \leq s \leq T}$
at some $\ell \in L^2({\mathbb S}^1;\R^d)$
and at some $t \in [0,T]$
and
taking the expectation 
in 
\eqref{eq:girsanov:bsde}
under $\tilde{\P}_{0}$, we conjecture (and this in fact the purpose of Theorem 
\ref{thm:pde} to make the statement clear) that:
\begin{equation}
\label{eq:mild:pde}
\begin{split}
{\mathcal U}^{n,\pm}(t,\cdot)
&= {\mathcal P}_{T-t}
\Bigl(
 D_{n,\pm} {\mathfrak g}_{0}(\cdot,
\textrm{Leb}_{1} \circ \cdot^{-1})
\Bigr)
\\
&\hspace{15pt}
+ \int_{t}^T 
{\mathcal P}_{s-t}
\Bigl[
 D_{n,\pm} {\mathfrak f}_{0}(\cdot,
\textrm{Leb}_{1} \circ \cdot^{-1})
+
\bigl\langle
D {\mathcal U}^{n,\pm}(s,\cdot) ,
{\mathfrak B}(\cdot) - 
{\mathcal U}(s,\cdot) 
\bigr\rangle_{L^2({\mathbb S}^1;\R^d)}
\Bigr] ds,
\end{split}
\end{equation}
where, differently from 
\eqref{eq:girsanov:bsde}, 
 we used the more compact notation ${\mathfrak B}$
for the drift coefficient. 
Here the notation
$\langle
D {\mathcal U}^{n,\pm}(s,\cdot),
{\mathfrak B}(\cdot) - 
{\mathcal U}(s,\cdot)
\rangle_{L^2({\mathbb S}^1;\R^d)}$ may be slightly confusing and should be understood as a function from 
$L^2({\mathbb S}^1;\R^d)$ into $\RR^d$ defined by:
\begin{equation*}
\begin{split}
\big\langle
 D {\mathcal U}^{n,\pm}(s,\cdot),
 {\mathfrak B}(\cdot) - 
{\mathcal U}(s,\cdot)
\big\rangle_{L^2({\mathbb S}^1;\R^d)}
: L^2({\mathbb S}^1;\R^d) \ni \ell
&\mapsto  
\big\langle
 D {\mathcal U}^{n,\pm}(s,\ell),
 {\mathfrak B}(\ell) - 
{\mathcal U}(s,\ell)
\big\rangle_{L^2({\mathbb S}^1;\R^d)}
\\
&= \sum_{k \in \NN}  D_{k,\pm}
{\mathcal U}^{n,\pm}(s,\ell)
\bigl( 
{\mathfrak B}^{k,\pm}(\ell)
-
{\mathcal U}^{k,\pm}(s,\ell)
\bigr),
\end{split}
\end{equation*}
the summand in the right-hand side reading as the product of a matrix 
of size $d \times d$ by 
a vector of size $d$.
Identity 
\eqref{eq:mild:pde}
is the cornerstone of the \textit{a priori} bound on the Lipschitz constant of ${\mathcal U}$ (in space).

\subsubsection*{Galerkin approximation} The problem with the formula 
\eqref{eq:mild:pde}
is that 
we do not know yet whether 
${\mathcal U}$ is Fr\'echet differentiable. 
In order to proceed, we take advantage of the 
stability properties of 
the solutions to 
\eqref{eq:mkv:sde}
in small time, which can be shown by a mere variation of 
the arguments used in the proof of Theorem
\ref{eq:small:time}. Indeed, 
we can 
use a Galerkin approximation and approximate 
${\mathfrak F}=D {\mathfrak f}_{0}$
and ${\mathfrak G}=D {\mathfrak g}_{0}$ by 
coefficients
${\mathfrak F}^{(N)}$
and 
${\mathfrak G}^{(N)}$
 with a truncated Fourier expansion, namely
\begin{equation}
\label{eq:galerkin}
\begin{split}
&{\mathfrak B}^{(N)}(\ell) 
= {\mathfrak B}
\Bigl( \sum_{k=0}^N \ell^{k,\pm} e^{k,\pm}(\cdot)
\Bigr)
\\
&{\mathfrak F}^{(N),n,\pm}(\ell) 
= {\mathfrak F}^{n,\pm}
\Bigl( \sum_{k=0}^N \ell^{k,\pm} e^{k,\pm}(\cdot)
\Bigr) {\mathbf 1}_{\{ n \leq N\}}, 
\quad 
{\mathfrak G}^{(N),n,\pm}(\ell) 
= {\mathfrak G}^{n,\pm}
\Bigl( \sum_{k=0}^N \ell^{k,\pm} e^{k,\pm}(\cdot)
\Bigr) {\mathbf 1}_{\{ n \leq N\}},
\end{split}
\end{equation} 
for $n \in \NN$, 
where we 
refer to \eqref{eq:mkv:sde:coefficients}
for the definitions of 
${\mathfrak F}$
and 
${\mathfrak G}$. 
It is clear that ${\mathfrak F}^{(N)}$ and 
${\mathfrak G}^{(N)}$ 
are bounded by the same constants 
as 
${\mathfrak F}$ and 
${\mathfrak G}$
and satisfy the same Lipschitz property. 
 Therefore, we can solve, for $T \leq c$
 with the same $c$ as in 
 Theorem
 \ref{eq:small:time},
  the forward backward system
 \begin{equation}
\label{eq:mkv:sde:N}
\begin{split}
&dX_{t}^{(N),n,\pm} = \Bigl( {\mathfrak B}^{(N),n,\pm}\Bigl(
\sum_{k=0}^N
X_{t}^{(N),k,\pm} e^{k,\pm}(\cdot)
\Bigr) - Y_{t}^{(N),n,\pm}
- (2 \pi n)^2 X_{t}^{(N),n,\pm}
\Bigr) dt + dW_{t}^{n,\pm}, 
\\
&dY_{t}^{(N),n,\pm} =
- {\mathfrak F}^{(N),n,\pm}\Bigl(
\sum_{k=0}^N
X_{t}^{(N),k,\pm}e^{k,\pm}(\cdot)
\Bigr)
dt + \sum_{k \in \NN}
Z_{t}^{(N),n,k,\pm} dW_{t}^{k,\pm},
\quad n \in \NN,
\end{split}
\end{equation}
with
$X_{0}^{(N),n,\pm}=X_{0}^{n,\pm}$
as initial condition 
and
$Y_{T}^{(N),n,\pm} = {\mathfrak G}^{(N),n,\pm}(\sum_{k=0}^N
X_{T}^{(N),k,\pm} e^{k,\pm}(\cdot)
)$
as
terminal condition.  
Observe in particular that $\bY^{(N),n,\pm}$ and $\bZ^{(N),n,\pm}$
are null for $n>N$.
Denoting by ${\mathcal U}^{(N)}$
the corresponding decoupling field, it is then well-checked that 
${\mathcal U}^{(N)}(t,\ell)$, for $\ell \in L^2({\mathbb S}^1;\R^d)$,
is a function of $(\ell^{n,\pm})_{0 \leq n \leq N}$ only, meaning that 
\begin{equation}
\label{eq:definition:cU(N)}
{\mathcal U}^{(N)}(t,\ell)
= {\mathcal U}^{(N)}
\Bigl(t,\sum_{k=0}^N
\ell^{k,\pm} e^{k,\pm}(\cdot)
\Bigr).
\end{equation}
Also, ${\mathcal U}^{(N),n,\pm}$ is zero when $n >N$. 

In words, the system \eqref{eq:mkv:sde:N} reduces to a finite dimensional 
system of $2N+1$ equations (i.e. up to the order $n = N$) on $\RR^{(2N+1)d}$.
By standard results
for non-degenerate forward-backward equations, 
see for instance \cite{DelarueGuatteri} (in order to fit the framework of the latter paper, 
notice that 
the linear term $-(2 \pi n)^2 X_{t}^{(N),n,\pm}$ can be easily removed by 
considering $\exp( (2 \pi n)^2 t)X_{t}^{(N),n,\pm}$ instead of 
$X_{t}^{(N),n,\pm}$),
we know that ${\mathcal U}^{(N)}(t,\cdot)$ is differentiable 
in $(\ell^{n,\pm})_{0\leq n \leq N}$
for $t<T$ and that, for $n \leq N$, 
\eqref{eq:gradient}
holds true 
with $Z^{n,k,\pm}$ replaced by 
$Z^{(N),n,k,\pm}$
and ${\mathcal U}^{n,\pm}$ replaced
by 
${\mathcal U}^{(N),n,\pm}$.

By stability in small time of the solutions to 
\eqref{eq:mkv:sde} (the proof of which works on the model of the proof of Theorem \ref{eq:small:time}), we can check that, for $T \leq c$,
\begin{equation}
\label{eq:stability:fbsde:galerkin}
\begin{split}
&\EE_{0} \Bigl[ \sup_{0 \leq t \leq T} 
\Bigl(
\| X_{t}(\cdot) - X_{t}^{(N)}(\cdot) \|^2_{L^2(\SS^1;\R^d)} + 
\| Y_{t}(\cdot) - Y_{t}^{(N)} (\cdot) \|^2_{L^2(\SS^1;\R^d)}
\Bigr) 
\Bigr]
\\
&\hspace{40pt} + \EE_{0} 
 \biggl[
\sum_{n,k \in \NN} 
 \int_{0}^T
 \vert Z_{t}^{n,k,\pm} - Z_{t}^{(N),n,k,\pm} \vert^2
 dt 
 \biggr]
 \\
 &\hspace{15pt} \leq
\EE_{0}
\biggl[
\bigl\| \bigl( {\mathfrak G}^{(N)} - {\mathfrak G} \bigr)\bigl(X_{T}(\cdot)\bigr) 
\bigr\|^2_{L^2(\SS^1;\R^d)}
\\
&\hspace{40pt}
+
\int_{0}^T 
\Bigl(
\bigl\| \bigl( {\mathfrak B}^{(N)} - {\mathfrak B} \bigr)\bigl(X_{t}(\cdot)\bigr) 
\bigr\|^2_{L^2(\SS^1;\R^d)}
+
\bigl\| \bigl( {\mathfrak F}^{(N)} - {\mathfrak F} \bigr)\bigl(X_{t}(\cdot)\bigr) 
\bigr\|^2_{L^2(\SS^1;\R^d)}
\Bigr)
 dt
 \biggr].
\end{split}
\end{equation}
Observe now that, for all $\ell \in L^2(\SS^1;\R^d)$, 
\begin{equation}
\label{eq:FN-F}
\begin{split}
\| {\mathfrak F}^{(N)}(\ell) - {\mathfrak F}(\ell) \|_{L^2(\SS^1;\R^d)}^2
&=
\sum_{n = 0}^N
\Bigl\vert 
{\mathfrak F}^{n,\pm}
\Bigl( 
\sum_{n=0}^N \ell^{n,\pm} e^{n,\pm}(\cdot)
\Bigr)
 - {\mathfrak F}^{n,\pm}(\ell)
 \Bigr\vert^2
+ \sum_{n \geq N+1}
\bigl\vert {\mathfrak F}^{n,\pm}(\ell)\bigr\vert^2
\\
&\leq 
\Bigl\| 
{\mathfrak F}
\Bigl( 
\sum_{n=0}^N \ell^{n,\pm} e^{n,\pm}(\cdot)
\Bigr)
 - {\mathfrak F}(\ell)
 \Bigr\|_{L^2(\SS^1;\R^d)}^2
+ \sum_{n \geq N+1}
\bigl\vert {\mathfrak F}^{n,\pm}(\ell)\bigr\vert^2
\\
&\leq C \sum_{n \geq N+1} | \ell^{n,\pm} \vert^2
+ \sum_{n \geq N+1}
\bigl\vert {\mathfrak F}^{n,\pm}(\ell)\bigr\vert^2,
\end{split}
\end{equation}
from which we get that the left-hand side tends to $0$. 
Proceeding in the same way with 
${\mathfrak B}^{(N)}- {\mathfrak B}$
and 
${\mathfrak G}^{(N)}-{\mathfrak G}$
and combining with Lebesgue's dominated convergence theorem, 
we deduce that the right-hand side in 
\eqref{eq:stability:fbsde:galerkin} tends to $0$ as $N$ tends to $\infty$. We deduce that the left-hand side also tends to $0$. 
And then, 
\begin{equation*}
\lim_{N \rightarrow +\infty}
{\mathcal U}^{(N)}(t,\ell) 
=
{\mathcal U}(t,\ell), 
\quad t \in [0,T], \ \ell \in L^2({\mathbb S}^1;\R^d),
\end{equation*}
and, for a given initial condition $\ell$ in 
\eqref{eq:mkv:sde:N},
\begin{equation*}
\lim_{N \rightarrow +\infty}
\E_{0}
\sum_{n,k \in \NN}
\int_{0}^T 
\vert Z_{s}^{\ell;(N),n,k,\pm}
-Z_{s}^{\ell;n,k,\pm}
\vert^2 
ds = 0, 
\quad t \in [0,T], \ \ell \in L^2({\mathbb S}^1;\R^d),
\end{equation*}
where we added the superscript $\ell$ 
in the notations to emphasize the fact that 
$X_{0}^{(N)}(\cdot)$ and $X_{0}(\cdot)$ were both equal
to $\ell$.
This says that, to prove 
\eqref{eq:gradient}
and 
the statement of
Theorem \ref{thm:pde}, 
we can focus first on the Galerkin approximation and then pass to the limit
as $N$ tends to $+\infty$. We shall come back to this point later on.

\subsubsection*{Smoothing estimates for the OU semi-group}
The long time analysis relies on the smoothing properties of the OU semi-group $(\cP_{t})_{t \geq 0}$ we introduced earlier, see 
\eqref{eq:SG}.

The following lemma is standard in the literature, see for instance
\cite[Section 5]{MR1731796}, see also 
\cite{MR1840644}. It will play a key role in the proof of Theorem \ref{thm:existence:uniqueness}.
\begin{lemma}
\label{lem:ibp}
Let $\cV$ be a bounded and measurable function from $L^2(\SS^1;\RR^d)$ into $\RR$. Then, for any $t
\in (0,T]$, 
${\mathcal P}_{t} {\mathcal V}$
is Fr\'echet differentiable and, for all $\ell \in L^2(\SS^1;\RR^d)$, 
\begin{equation*}
\bigl\|
D
{\mathcal P}_{t} {\mathcal V}(\ell)
\bigr\|_{L^2({\mathbb S}^1;\R^d)} \leq C t^{-1/2}
\EE_{0} \bigl[ \vert \cV(U_{t}^\ell)
\vert^2 \bigr]^{1/2}
\leq C t^{-1/2} 
\| {\mathcal V}\|_{\infty},
\end{equation*}
for a constant $C$ independent of $t \in (0,T]$. 
If $\cV$ is Lipschitz continuous on $L^2(\SS^1;\R^d)$, then, 
for any $t \in ( 0,T]$ and any $\ell \in L^2(\SS^1;\R^d)$,
\begin{equation*}
\bigl\|
D
{\mathcal P}_{t} {\mathcal V}(\ell)
\bigr\|_{L^2({\mathbb S}^1;\R^d)} 
\leq 
\textrm{\rm Lip}({\mathcal V}),
\end{equation*}
where 
$\textrm{\rm Lip}({\mathcal V})$ is the Lipschitz constant of ${\mathcal V}$. 
%
%
\end{lemma}
The second inequality in the statement is just a consequence of the fact that the function 
$L^2(\SS^1;\RR^d) \ni \ell \mapsto {\mathbb E}[{\mathcal V}(U_{t}^{\ell})]$
is $\textrm{\rm Lip}(\cV)$-Lipschitz continuous. 

\subsection{Analysis of the Galerkin approximation}
For a given fixed $T>0$, we consider the Galerkin approximation of the coefficients, 
as defined in 
\eqref{eq:galerkin}, together with the corresponding Galerkin approximation of the forward-backward system, 
as defined in \eqref{eq:mkv:sde:N}. 

As we already explained, the system \eqref{eq:mkv:sde:N} is already known to be uniquely solvable, for any given initial condition for 
$\bX^{(N)}$, whatever the time duration $T$ is. 
Also, we know from 
\cite{DelarueGuatteri} that the decoupling field
$\cU^{(N)}$, when  regarded as a function from 
$[0,T] \times \RR^{(2N+1)d}$ into $\RR^{(2N+1)d}$ 
 satisfies a system of $(2N+1)$ PDEs in dimension $(2N+1)d$.
 By identifying the Fr\'echet derivative $D \cU^{(N)}$ 
 of $\cU^{(N)}$ with the derivatives in $\RR^{2N+1}$ through the
 formula:
 \begin{equation*}
 D \cU^{(N)}(t,\ell) = \sum_{n=0}^N \partial_{\ell^{n,\pm}}
 \cU^{(N)}\Bigl(t, 
 \sum_{k=0}^N \ell^{k,\pm} e^{k,\pm}(\cdot)
 \Bigr)
  e^{n,\pm}(\cdot), 
 \end{equation*}
the system of PDEs satisfied by the decoupling field of 
\eqref{eq:mkv:sde:N}
coincides, in the mild form, with 
\eqref{eq:mild:pde}, but with 
$D {\mathfrak f}_{0}$ and $D {\mathfrak g}_{0}$
and replaced by 
${\mathfrak F}^{(N)}$ and ${\mathfrak G}^{(N)}$. 
Namely, we have:
\begin{equation}
\label{eq:galerkin:mild:formulation}
\begin{split}
{\mathcal U}^{(N),n,\pm}(t,\cdot)
&= {\mathcal P}_{T-t}
\bigl(
  {\mathfrak G}^{(N),n,\pm}
 \bigr)
\\
&\hspace{15pt}
+ \int_{t}^T 
{\mathcal P}_{s-t}
\Bigl[
{\mathfrak F}^{(N),n,\pm}(\cdot)
+
\bigl\langle
 D {\mathcal U}^{(N),n,\pm}(s,\cdot),
{\mathfrak B}^{(N)}(\cdot) - 
{\mathcal U}^{(N)}(s,\cdot)\bigr\rangle_{L^2({\mathbb S}^1;\R^d)}
\Bigr] ds,
\end{split}
\end{equation}
the identity holding true in $L^2(\SS^1;\R^d)$, for any $t \in [0,T]$. 
\vspace{4pt}

Following 
 \eqref{eq:bound:Y:L2}, we claim first:
\begin{lemma}
\label{lem:bound:pp}
There exists a constant $C$ such that, for all $N \in \NN^*$, 
\begin{equation*}
\sup_{t \in [0,T]} \sup_{\ell \in L^2(\SS^1;\RR^d)}
\| \cU^{(N)}(t,\ell) \|_{L^2(\SS^1;\RR^d)} \leq C. 
\end{equation*}
\end{lemma}

The following lemma provides a uniform bound for the Fr\'echet derivative of 
the Galerkin approximation:
\begin{lemma}
\label{lem:lemma:ibp:UN}
There
exists a constant $C$ independent of $N$ 
such that, for all $t \in [0,T)$ and 
all $N \in \NN^*$,
\begin{equation*}
\sup_{\ell \in L^2(\SS^1;\RR^d)}
\vvvert D \cU^{(N)}(t,\ell)
\vvvert_{L^2(\SS^1;\R^d) \times L^2(\SS^1;\R^d)} 
\leq C,
\end{equation*}
where
\begin{equation*}
\vvvert D \cU^{(N)}(t,\ell)
\vvvert_{L^2(\SS^1;\R^d) \times L^2(\SS^1;\R^d)} 
=
 \sup_{h \in L^2({\mathbb S}^1;\RR^d) :  \| h\|_{L^2({\mathbb S}^1;\RR^d)} \leq 1} 
 \Bigl\|
  D \bigl[  \langle \cU^{(N)}(t,\cdot), h \rangle_{L^2({\mathbb S}^1;\RR^d)}
  \bigr]_{\vert \cdot = \ell}
  \Bigr\|_{L^2({\mathbb S}^1;\RR^d)},
 \end{equation*}
 the notation $D [ \varphi(\cdot)]_{\vert \cdot =\ell}$ indicating 
 the fact that the differential is computed with respect to the argument 
 $\cdot$ and then taken at point $\ell$.  
\end{lemma}
\begin{proof}
We start from 
\eqref{eq:galerkin:mild:formulation}. For $h \in L^2({\mathbb S}^1)$,
\begin{equation}
\label{eq:estimation:fine:1}
\begin{split}
\big\langle  {\mathcal U}^{(N)}(t,\cdot), h \big\rangle
&= {\mathcal P}_{T-t}
\Bigl[ 
\bigl\langle{\mathfrak G}^{(N)}(\cdot),h \rangle_{L^2({\mathbb S}^1;\RR^d)}
\Bigr]
+ \int_{t}^T {\mathcal P}_{s-t}
\Bigl[ \big\langle {\mathfrak F}^{(N)}(\cdot),
h
\big\rangle_{L^2(\SS^1;\RR^d)}
\Bigr] ds
\\
&\hspace{15pt}
+
\int_{t}^T {\mathcal P}_{s-t}
\Bigl[  {\mathfrak b}_{0}^{(N)}(\cdot)  \cdot 
D_{0} \langle {\mathcal U}^{(N)}(s,\cdot),
h \rangle_{L^2({\mathbb S}^1;\R^d)}
\Bigr] ds
\\
&\hspace{15pt}
- \int_{t}^T {\mathcal P}_{s-t}
\Bigl[ \big\langle D \langle {\mathcal U}^{(N)}(s,\cdot),h
\rangle_{L^2({\mathbb S}^1;\R^d)}, {\mathcal U}^{(N)}(s,\cdot)
 \big\rangle_{L^2({\mathbb S}^1;\R^d)}
\Bigr] ds
\\
&= T_{1} + T_{2} + T_{3}. 
\end{split}
\end{equation}
Apply now 
Lemma 
\ref{lem:ibp}
when $\| h \|_{L^2({\mathbb S}^1;\RR^d)} \leq 1$. 
Deduce that
\begin{equation*}
\begin{split}
&\sup_{\ell \in L^2({\mathbb S}^1;\R^d)}
\sup_{\| h \|_{L^2({\mathbb S}^1;\R^d)} \leq 1}
\bigl\| 
\bigl[D \big\langle {\mathcal U}^{(N)}(t,\cdot),h \big\rangle_{L^2({\mathbb S}^1;\RR^d)}
\bigr]_{\vert \cdot=\ell}
\bigr\|_{L^2(\SS^1;\R^d)} 
\\
&\leq
C \biggl\{
\sup_{\| h \|_{L^2({\mathbb S}^1;\R^d)} \leq 1}
\textrm{\rm Lip}\Bigl(  \bigl\langle {\mathfrak G}^{(N)}(\cdot), h
\bigr\rangle_{L^2(\SS^1;\R^d)}
\Bigr)
\\
&\hspace{15pt}
+
\int_{t}^T 
\frac{1}{\sqrt{s-t}}
\sup_{\ell \in L^2({\mathbb S}^1;\R^d)}
\sup_{\| h \|_{L^2({\mathbb S}^1;\R^d)} \leq 1}
\Big|{\mathfrak b}_{0}^{(N)}(\ell)
\cdot
\bigl[ D_{0} \langle {\mathcal U}^{(N)}(s,\cdot),h \rangle_{L^2({\mathbb S}^1)}
\bigr]_{\vert \cdot = \ell}
 \Big|
ds
\\
&\hspace{15pt}
+
\int_{t}^T 
\frac{1}{\sqrt{s-t}}
\sup_{\ell \in L^2({\mathbb S}^1;\R^d)}
\sup_{\| h \|_{L^2({\mathbb S}^1;\R^d)} \leq 1}
\Big|
\big\langle D \bigl[ \langle {\mathcal U}^{(N)}(s,\cdot),h \rangle_{L^2({\mathbb S}^1;\RR^d)}
\bigr]_{\vert \cdot = \ell}, {\mathcal U}^{(N)}(s,\ell) \big\rangle_{L^2({\mathbb S}^1;\R^d)}
\Big|
ds
\\
&\hspace{15pt}
+
\int_{t}^T 
\frac{1}{\sqrt{s-t}}
\sup_{\ell \in L^2({\mathbb S}^1;\R^d)}
\big\|
  {\mathfrak F}^{(N)}(\ell) 
\big\|_{L^2({\mathbb S}^1;\R^d)} ds \biggr\},
\end{split}
\end{equation*}
for a constant $C$ whose value may change from line to line. 
Recall now that 
\begin{equation*}
\sup_{\| h \|_{L^2({\mathbb S}^1;\R^d)} \leq 1}
\textrm{\rm Lip}\Bigl(  \bigl\langle {\mathfrak G}^{(N)}(\cdot), h
\bigr\rangle_{L^2(\SS^1;\R^d)}\Bigr) \leq C,
\end{equation*} 
 and that 
\begin{equation*}
\begin{split}
&\sup_{\ell \in L^2({\mathbb S}^1;\R^d)}
\Bigl\{
\big\|
  {\mathfrak F}^{(N)}(\ell) 
\big\|_{L^2({\mathbb S}^1;\R^d)}, 
\sup_{t \in [0,T]}
\big\|
  {\cU}^{(N)}(t,\ell) 
\big\|_{L^2({\mathbb S}^1;\R^d)}
\Bigr\}
\leq C.
\end{split}
\end{equation*}
We deduce that
\begin{equation*}
\begin{split}
&\sup_{\ell \in L^2({\mathbb S}^1;\R^d)}
\sup_{\| h \|_{L^2({\mathbb S}^1;\R^d)} \leq 1}
\bigl\| D \bigl[  \big\langle {\mathcal U}^{(N)}(t,\cdot),h\big\rangle_{L^2({\mathbb S}^1;\R^d)}
\bigr]_{\vert \cdot = \ell}
\bigr\|_{L^2(\SS^1;\R^d)} 
\\
&\leq
C
+
\int_{t}^T 
\frac{C}{\sqrt{s-t}}
\sup_{\ell \in L^2({\mathbb S}^1;\R^d)}
\sup_{\| h \|_{L^2({\mathbb S}^1;\R^d)} \leq 1}
\Big\|
D \bigl[ \langle {\mathcal U}^{(N)}(s,\cdot),h \rangle_{L^2({\mathbb S}^1;\R^d)}
\bigr]_{\vert \cdot = \ell}
 \Big\|_{L^2(\SS^1;\R^d)}
ds.
\end{split}
\end{equation*}
By a variant of Gronwall's lemma, see Lemma 
\ref{lem:4:2:5} right below, we complete the proof. 
\end{proof}

Using a similar argument, we claim:
\begin{lemma}
\label{lem:convergence:gradients:galerkin}
For any  
compact subset $\cK \subset L^2(\SS^1;\RR^d)$,
there exist a constant 
$C$ and real $\varepsilon >0$,
such that, 
for all $t \in [0,T]$ and all $N,M \in {\mathbb N}^*$, 
\begin{equation*}
\begin{split}
&\sup_{\ell \in \cK}
\bigl\vvvert 
D {\mathcal U}^{(N)}(t,\ell)-
D {\mathcal U}^{(M)}(t,\ell)
\bigr\vvvert_{L^2(\SS^1;\R^d) \times L^2(\SS^1;\R^d)} 
\\ 
 &\leq 
\frac{C}{\sqrt{T-t}}
\biggl[
\Bigl( \sup_{\ell \in \cK}
\sum_{n >N \wedge M} \vert \ell^{n,\pm} \vert^2 
+ 
\sup_{h \in \cK^{\varepsilon}}
\sum_{n >N \wedge M} \vert {\mathfrak F}^{n,\pm}(h) \vert^2 
+ 
\sup_{h \in \cK^{\varepsilon}}
\sum_{n >N \wedge M} \vert {\mathfrak G}^{n,\pm}(h) \vert^2 
\\
&\hspace{100pt} +
\sup_{s \in [{0},T]}
\sup_{h \in \cK^{\varepsilon}}
\bigl\|
\bigl(
{\mathcal U}^{(N)}-
{\mathcal U}^{(M)}
\bigr)
(s,h)
 \bigr\|_{L^2({\mathbb S}^1)}^2
+
\sup_{\ell \in \cK}
 \sup_{r \in [0,T]} \PP \bigl( U_{r}^\ell \not \in \cK^{\varepsilon} \bigr)
\Bigr)^{1/2} \biggr].
\end{split}
\end{equation*}
 where
\begin{equation*}
\begin{split}
&\vvvert 
D {\mathcal U}^{(N)}(t,\ell)-
D {\mathcal U}^{(M)}(t,\ell)
\vvvert_{L^2(\SS^1;\R^d) \times L^2(\SS^1;\R^d)}
\\
&=
\sup_{h \in L^2(\SS^1;\RR^d) : \| h \|_{L^2(\SS^1;\R^d)} \leq 1}
\bigl\| 
D \bigl[ \langle {\mathcal U}^{(N)}(t,\cdot)-
{\mathcal U}^{(M)}(t,\cdot), h \rangle
\bigr]_{\cdot = \ell}
\bigr\|_{L^2(\SS^1;\R^d)},
\end{split}
\end{equation*} 
and
\begin{equation*}
\sup_{\ell \in \cK}
 \sup_{r \in [0,T]} \PP_{0} \bigl( U_{r}^\ell \not \in \cK^{\varepsilon} \bigr)
\Bigr) \leq \varepsilon.
\end{equation*}
\end{lemma}

\begin{proof}
Throughout the proof, we use the fact that $\cK$ is compact in 
$L^2(\SS^1;\R^d)$ if and only if $\cK$ is closed and, for any $\epsilon >0$, there exists 
$n \in \NN$, such that for all $h \in \cK$, $\sum_{k \geq n}
\vert h^{k,\pm} \vert^2 \leq \varepsilon$. 
\vskip 4pt

\textit{First step.}
Also, we recall that ${\mathfrak G}$ and ${\mathfrak F}$
are continuous from $L^2(\SS^1;\R^d)$ into itself. Hence, 
${\mathfrak G}(\cK)$ and ${\mathfrak F}(\cK)$ are compact subsets of 
$L^2(\SS^1;\R^d)$. In particular, for all $\varepsilon>0$, there exists $n \in \NN^*$, such that for all $h \in \cK$, 
\begin{equation*}
\sum_{k \geq n}
\vert {\mathfrak F}^{k,\pm}(h) \vert^2 \leq \varepsilon, 
\quad 
\sum_{k \geq n}
\vert {\mathfrak G}^{k,\pm}(h) \vert^2 \leq \varepsilon. 
\end{equation*}

Also, we observe that that, for any compact subset $\cK$
and any $\varepsilon >0$, there exists another
compact subset $\cK_{\varepsilon}$ such that, for all 
$\ell \in \cK$, for all $t \in [0,T]$,
\begin{equation}
\label{eq:proof:lemma:4.16:0}
\PP_{0} \bigl[ U_{t}^\ell \in \cK^{\varepsilon} 
\bigr] \geq 1 - \varepsilon. 
\end{equation}
The proof is quite straightforward. We give it for the sake of completeness. Indeed, we recall that:
\begin{equation}
\label{eq:proof:lemma:4.16:1}
\begin{split}
U_{t}^\ell
= \sum_{n \in \NN}
\Bigl(
e^{ - (2 \pi n)^2 t }
\ell^{n,\pm}
+ \int_{0}^t e^{ - (2 \pi n)^2 (t-s)}
dW^{n,\pm}_{s}
\Bigr) e^{n,\pm}(\cdot).
\end{split}
\end{equation}
Obviously, we have, for any $n \in \NN$, 
\begin{equation}
\label{eq:proof:lemma:4.16:2}
\sum_{k \geq n}
\bigl\vert 
e^{ - (2 \pi k)^2 t }
\ell^{k,\pm}
\bigr\vert^2
\leq \sum_{k \geq n}
\bigl\vert 
\ell^{k,\pm}
\bigr\vert^2,
\end{equation}
which can be made as small as desired by choosing $n$ large enough, uniformly in $\ell \in \cK$. Also, for any $n \in \NN$,
\begin{equation*}
\begin{split}
\sum_{k \geq n}
\EE_{0} \biggl[
\biggl\vert 
\int_{0}^t e^{ - (2 \pi k)^2 (t-s)}
dW^{k,\pm}_{s}
\biggr\vert^2 \biggr] = \sum_{k \geq n}
\int_{0}^t e^{ - 2 (2 \pi k)^2 (t-s)} ds \leq \sum_{k \geq n} \frac{1}{2 (2 \pi k)^2}.
\end{split}
\end{equation*}
In particular, we can find a universal constant $c>0$ such that:
\begin{equation}
\label{eq:proof:lemma:4.16:3}
\begin{split}
\sum_{k \geq n}
\EE_{0} \biggl[
\biggl\vert 
\int_{0}^t e^{ - (2 \pi k)^2 (t-s)}
dW^{k,\pm}_{s}
\biggr\vert^2 \biggr]  \leq \frac{c}{n}.
\end{split}
\end{equation}
We deduce that
\begin{equation*}
\PP_{0}
\biggl[ 
\sum_{k \geq n^3}
\biggl\vert 
\int_{0}^t e^{ - (2 \pi k)^2 (t-s)}
dW^{k,\pm}_{s}
\biggr\vert^2
\geq \frac1{n}
\biggr]
\leq \frac{c}{n^2},
\end{equation*}
and then, by Borel-Cantelli's Lemma, 
we obtain:
\begin{equation*}
\begin{split}
&\lim_{p \rightarrow \infty}
\PP_{0}
\biggl(
\bigcap_{n \geq p}
\biggl\{
\sum_{k \geq n^3}
\biggl\vert 
\int_{0}^t e^{- (2 \pi k)^2 (t-s)}
dW^{k,\pm}_{s}
\biggr\vert^2
\leq \frac1{n}
\biggr\}
\biggr)
\\
&=
\PP_{0}
\biggl(
\bigcup_{p \geq 1} \bigcap_{n\geq p}
\biggl\{
\sum_{k \geq n^3}
\biggl\vert 
\int_{0}^t e^{ - (2 \pi k)^2 (t-s)}
\bigr)
dW^{k,\pm}_{s}
\biggr\vert^2
\leq \frac1{n}
\biggr\}
\biggr)
= 1.
\end{split}
\end{equation*}
It remains to observe that, for any $p \geq 1$, the set 
\begin{equation*}
\bigcap_{n \geq p}
\Bigl\{h \in L^2(\SS^1;\R^d) : 
\sum_{k \geq n^3}
\vert h^{k,\pm} \vert^2
\leq \frac1{n}
\Bigr\}
\end{equation*}
is compact in $L^2(\SS^1;\RR^d)$. 
\vskip 4pt

\textit{Second step.}
Following \eqref{eq:FN-F}, we observe that 
there exists a constant $C \geq 0$ such that, 
for all $N \in \NN^*$, $t \in [0,T]$ and $\ell \in \cK$,
\begin{equation}
\label{eq:proof:lemma:4.16:5}
\begin{split}
{\mathbb E}_{0}
\Bigl[
\bigl\| {\mathfrak G}(U_{T-t}^{\ell}) - {\mathfrak G}^{(N)}(U_{T-t}^{\ell}) 
\bigr\|_{L^2(\SS^1;\R^d)}^2
\Bigr] 
&\leq C {\mathbb E}_{0}
\Bigl[
\sum_{n > N}
\bigl| 
(U_{T-t}^{\ell})^{n,\pm}
\bigr|^2
\Bigr] + C \EE_{0} \Bigl[
\sum_{n >N} 
\vert 
{\mathfrak G}^{n,\pm}(U_{T-t}^{\ell}) 
\vert^2
\Bigr].
\end{split}
\end{equation}
By 
\eqref{eq:proof:lemma:4.16:1},
\eqref{eq:proof:lemma:4.16:2}
and
\eqref{eq:proof:lemma:4.16:3}, we have
\begin{equation*}
{\mathbb E}_{0}
\Bigl[
\sum_{n > N}
\bigl| 
(U_{T-t}^{\ell})^{n,\pm}
\bigr|^2
\Bigr] \leq \sum_{n >N} \vert \ell^{n,\pm} \vert^2 + \frac{c}{N}. 
\end{equation*}
Also, using the same notation $\cK^{\varepsilon}$ as in 
\eqref{eq:proof:lemma:4.16:0}, we have, for any $\varepsilon >0$ and for all $N \in \NN^*$, $t \in [0,T]$ and 
$\ell \in L^2(\SS^1;\R^d)$: 
\begin{equation}
\begin{split}
\EE_{0} \Bigl[
\sum_{n >N} 
\vert 
{\mathfrak G}^{n,\pm}(U_{T-t}^{\ell}) 
\vert^2
\Bigr]
&\leq 
\EE_{0} \Bigl[
{\mathbf 1}_{\{U_{T-t}^\ell \in \cK^{\varepsilon}\}}
\sum_{n >N} 
\vert 
{\mathfrak G}^{n,\pm}(U_{T-t}^{\ell}) 
\vert^2
\Bigr]
+  C \PP_{0} \bigl( U_{T-t}^\ell \not \in \cK^{\varepsilon} \bigr)
\\
&\leq \sup_{l \in \cK^{\varepsilon}}
\sum_{n >N} \vert {\mathfrak G}^{n,\pm}(l) \vert^2 
+  C \PP_{0} \bigl( U_{T-t}^\ell \not \in \cK^{\varepsilon} \bigr),
\end{split}
\end{equation}
where we used the fact that 
${\mathfrak G}$ is bounded and where we allowed the constant $C$ to increase from line to line. 

Therefore, 
\eqref{eq:proof:lemma:4.16:5} yields
\begin{equation}
\label{eq:proof:lemma:4.16:5:b}
\begin{split}
&{\mathbb E}_{0}
\Bigl[
\bigl\| {\mathfrak G}(U_{T-t}^{\ell}) - {\mathfrak G}^{(N)}(U_{T-t}^{\ell}) 
\bigr\|_{L^2(\SS^1;\R^d)}^2
\Bigr] 
\\
&\hspace{15pt} \leq C
\Bigl(  
\sup_{l \in \cK}
\sum_{n >N} \vert l^{n,\pm} \vert^2 + 
\sup_{l \in \cK^{\varepsilon}}
\sum_{n >N} \vert {\mathfrak G}^{n,\pm}(l) \vert^2 
+  \PP_{0} \bigl( U_{T-t}^\ell \not \in \cK^{\varepsilon} \bigr)
+ \frac1{N}
\Bigr). 
\end{split}
\end{equation}
Similarly, 
\begin{equation}
\label{eq:rate:F}
\begin{split}
&\sup_{s \in [0,T-t]}
{\mathbb E}_{0}
\Bigl[
\bigl\| {\mathfrak F}(U^{\ell}_{s}) - {\mathfrak F}^{(N)}(U^{\ell}_{s}) 
\bigr\|_{L^2(\SS^1;\R^d)}^2
\Bigr] 
\\
&\hspace{15pt}
\leq 
  C
\Bigl(  
\sup_{l \in \cK}
\sum_{n >N} \vert l^{n,\pm} \vert^2 + 
\sup_{l \in \cK^{\varepsilon}}
\sum_{n >N} \vert {\mathfrak F}^{n,\pm}(l) \vert^2 
+  \sup_{s \in [0,T]}\PP_{0} \bigl( U_{s}^\ell \not \in \cK^{\varepsilon} \bigr)
+ \frac1{N}
\Bigr).
\end{split}
\end{equation}
Obviously, the same bound holds true when replacing 
${\mathfrak F}$ by ${\mathfrak B}$.
We now return to 
\eqref{eq:estimation:fine:1}
and we write:
\begin{equation}
\label{eq:convergence:N:M:proof:1}
\begin{split}
&\big\langle  \bigl( {\mathcal U}^{(N)}-
{\mathcal U}^{(M)}
\bigr)
(t,\cdot), h \big\rangle
\\
&= {\mathcal P}_{T-t}
\Bigl[ 
\bigl\langle
\bigl( {\mathfrak G}^{(N)}
-{\mathfrak G}^{(M)}
\bigr)
(\cdot),h \bigr\rangle_{L^2({\mathbb S}^1;\R^d)}
\Bigr]
+ \int_{t}^T {\mathcal P}_{s-t}
\Bigl[ \big\langle 
\bigl( {\mathfrak F}^{(N)}-
{\mathfrak F}^{(M)}
\bigr)
(\cdot),
h
\big\rangle_{L^2(\SS^1;\R^d)}
\Bigr] ds
\\
&\hspace{15pt}
+
\int_{t}^T {\mathcal P}_{s-t}
\Bigl[  \bigl( {\mathfrak b}_{0}^{(N)}
-
{\mathfrak b}_{0}^{(M)}
\bigr)
(\cdot) \cdot D_{0} \langle {\mathcal U}^{(N)}(s,\cdot),
h \rangle_{L^2({\mathbb S}^1;\R^d)}
\Bigr] ds
\\
&\hspace{15pt}+
\int_{t}^T
{\mathcal P}_{s-t} \Bigl[{\mathfrak b}_{0}^{(M)}
(\cdot)  \cdot D_{0} \langle \bigl( {\mathcal U}^{(N)}-{\mathcal U}^{(M)}\bigr)(s,\cdot),
h \rangle_{L^2({\mathbb S}^1;\R^d)}
\Bigr] ds
\\
&\hspace{15pt}
- \int_{t}^T {\mathcal P}_{s-t}
\Bigl[ \big\langle  D \langle 
( {\mathcal U}^{(N)}- {\mathcal U}^{(M)})(s,\cdot),h
\rangle_{L^2({\mathbb S}^1;\R^d)}, {\mathcal U}^{(N)}(s,\cdot)
 \big\rangle_{L^2({\mathbb S}^1;\R^d)}
\Bigr] ds
\\
&\hspace{15pt}
- \int_{t}^T {\mathcal P}_{s-t}
\Bigl[ \big\langle D \langle {\mathcal U}^{(M)}
(s,\cdot),h
\rangle_{L^2({\mathbb S}^1;\R^d)}, 
\bigl(
{\mathcal U}^{(N)}-
{\mathcal U}^{(M)}
\bigr)
(s,\cdot)
 \big\rangle_{L^2({\mathbb S}^1;\R^d)}
\Bigr] ds. 
\end{split}
\end{equation}
We then make use of 
Lemma \ref{lem:ibp}.
We can find a constant $C$ such that, for all 
$N,M \geq 1$, $\ell \in \cK$, $h \in L^2(\SS^1;\R^d)$ 
with $\|h \|_{L^2(\SS^1;\R^d)} \leq 1$,
 and $t \in [0,T]$, 
\begin{equation*}
\begin{split}
&\Bigl\|
D \Bigl[ {\mathcal P}_{T-t}
\bigl[ 
\bigl\langle
\bigl( {\mathfrak G}^{(N)}
-{\mathfrak G}^{(M)}
\bigr)
(\cdot),h \bigr\rangle_{L^2({\mathbb S}^1;\R^d)}
\bigr]
\Bigr]_{\vert \cdot = \ell}
\Bigr\|_{L^2(\SS^1;\R^d)}
\\
&\leq \frac{C}{\sqrt{T-t}}
\EE_{0}
\Bigl[
\bigl\|
\bigl( {\mathfrak G}^{(N)}
-{\mathfrak G}^{(M)}
\bigr)
(U_{T-t}^\ell)
 \bigr\|_{L^2({\mathbb S}^1;\R^d)}^2 
\Bigr]^{1/2},
\end{split}
\end{equation*}
where we used the fact that 
$\EE_{0}
[
|
\langle
( {\mathfrak G}^{(N)}
-{\mathfrak G}^{(M)}
)
(U_{T-t}^\ell),h \rangle_{L^2({\mathbb S}^1;\R^d)}
|^2 
]^{1/2}$
is less than 
$\EE_{0}
[\|
( {\mathfrak G}^{(N)}
-{\mathfrak G}^{(M)}
)(U_{T-t}^\ell)
\|_{L^2({\mathbb S}^1;\R^d)}^2 
]^{1/2}$. 
If, instead of $\ell$, we choose the realization of the random variable 
$U_{t-t_{0}}^\ell$, for $t_{0} \in [0,t]$, we get by the flow property of the Ornstein-Uhlenbeck process:
\begin{equation}
\label{eq:convergence:N:M:proof:2}
\begin{split}
&\EE_{0} \biggl[ \Bigl\|
D \Bigl[
 {\mathcal P}_{T-t}
\bigl[ 
\bigl\langle
\bigl( {\mathfrak G}^{(N)}
-{\mathfrak G}^{(M)}
\bigr)
(\cdot),h \bigr\rangle_{L^2({\mathbb S}^1;\R^d)}
\bigr]
\Bigr]_{\cdot = U_{t-t_{0}}}
\Bigr\|_{L^2(\SS^1;\R^d)}^2 \biggr]^{1/2}
\\
&\leq
\frac{C}{\sqrt{T-t}}
\EE_{0}
\Bigl[
\bigl\|
\bigl( {\mathfrak G}^{(N)}
-{\mathfrak G}^{(M)}
\bigr)
(U_{T-t_{0}}^\ell)
 \bigr\|_{L^2({\mathbb S}^1;\R^d)}^2 
\Bigr]^{1/2}.
\end{split}
\end{equation}
By \eqref{eq:proof:lemma:4.16:5:b} and \eqref{eq:proof:lemma:4.16:0}, we obtain:
\begin{equation*}
\begin{split}
&{\mathbb E}_{0}
\Bigl[
\bigl\| \bigl( {\mathfrak G}^{(N)} - {\mathfrak G}^{(M)}\bigr)(U_{T-t_{0}}^{\ell}) 
\bigr\|_{L^2(\SS^1;\R^d)}^2
\Bigr] 
\\
&\hspace{15pt} \leq C
\Bigl( 
\sup_{l \in \cK}
\sum_{n >N \wedge M} \vert l^{n,\pm} \vert^2 + 
\sup_{l \in \cK^{\varepsilon}}
\sum_{n >N \wedge M} \vert {\mathfrak G}^{n,\pm}(l) \vert^2 
+
\sup_{l \in \cK}
 \sup_{r \in [0,T]} \PP_{0} \bigl( U_{r}^l \not \in \cK^{\varepsilon} \bigr)
 + \frac{1}{N \wedge M}
\Bigr). 
\end{split}
\end{equation*}
Therefore, 
\begin{equation*}
\begin{split}
&\EE_{0} \biggl[ \Bigl\|
D \Bigl[ {\mathcal P}_{T-t}
\bigl[ 
\bigl\langle
\bigl( {\mathfrak G}^{(N)}
-{\mathfrak G}^{(M)}
\bigr)
(\cdot),h \bigr\rangle_{L^2({\mathbb S}^1;\R^d)}
\bigr]
\Bigr]_{\vert \cdot = U_{t-t_{0}}^\ell}
\Bigr\|_{L^2(\SS^1;\R^d)}^2 \biggr]^{1/2}
\\
&\leq
\frac{C}{\sqrt{T-t}}
\biggl( \sup_{l \in \cK}
\sum_{n >N \wedge M} \vert l^{n,\pm} \vert^2 + 
\sup_{l \in \cK^{\varepsilon}}
\sum_{n >N \wedge M} \vert {\mathfrak G}^{n,\pm}(l) \vert^2 
+
\sup_{l \in \cK}
 \sup_{r \in [0,T]} \PP_{0} \bigl( U_{r}^l \not \in \cK^{\varepsilon} \bigr)
  + \frac{1}{N \wedge M}
\biggr)^{1/2}. 
\end{split}
\end{equation*}
By the same argument, 
\begin{equation*}
\begin{split}
&\EE_{0} \biggl[ \Bigl\|
D \Bigl[ {\mathcal P}_{s-t}
\bigl[ 
\bigl\langle
\bigl( {\mathfrak F}^{(N)}
-{\mathfrak F}^{(M)}
\bigr)
(\cdot),h \bigr\rangle_{L^2({\mathbb S}^1;\R^d)}
\bigr]
\Bigr]_{\vert \cdot =  U_{t-t_{0}}^\ell}
\Bigr\|_{L^2(\SS^1;\R^d)}^2 \biggr]^{1/2}
\\
&\leq
\frac{C}{\sqrt{s-t}}
\biggl( \sup_{l \in \cK}
\sum_{n >N \wedge M} \vert l^{n,\pm} \vert^2 + 
\sup_{l \in \cK^{\varepsilon}}
\sum_{n >N \wedge M} \vert {\mathfrak F}^{n,\pm}(l) \vert^2 
+
\sup_{l \in \cK}
 \sup_{r \in [0,T]} \PP_{0} \bigl( U_{r}^l \not \in \cK^{\varepsilon} \bigr)
   + \frac{1}{N \wedge M}
\biggr)^{1/2}. 
\end{split}
\end{equation*}
Similarly, using 
Lemma \ref{lem:lemma:ibp:UN}, it holds that
\begin{equation*}
\begin{split}
&\EE_{0} \biggl[ \Bigl\|
D \Bigl[ {\mathcal P}_{s-t}
\bigl[
 \bigl( {\mathfrak b}_{0}^{(N)}
-
{\mathfrak b}_{0}^{(M)}
\bigr)
(\cdot) \cdot D_{0} \langle {\mathcal U}^{(N)}(s,\cdot),
h \rangle_{L^2({\mathbb S}^1;\RR^d)}
\bigr]
\Bigr]_{\vert \cdot = U_{t-t_{0}}^\ell}
\Bigr\|_{L^2(\SS^1;\R^d)}^2 \biggr]^{1/2}
\\
&\leq
\frac{C}{\sqrt{s-t}}
\biggl( \sup_{l \in \cK}
\sum_{n >N \wedge M} \vert l^{n,\pm} \vert^2 
+
\sup_{l \in \cK}
 \sup_{r \in [0,T]} \PP_{0} \bigl( U_{r}^l \not \in \cK^{\varepsilon} \bigr)
 + \frac{1}{N \wedge M}
\biggr)^{1/2}. 
\end{split}
\end{equation*}
We now turn to the term on the third line in 
\eqref{eq:convergence:N:M:proof:1}. 
Following 
\eqref{eq:convergence:N:M:proof:2}, we have
\begin{equation*}
\begin{split}
&\EE_{0} \biggl[ \Bigl\|
D \Bigl[ {\mathcal P}_{s-t}
\bigl[ 
{\mathfrak b}_{0}^{(M)}
(\cdot)  \cdot D_{0} \langle \bigl( {\mathcal U}^{(N)}-{\mathcal U}^{(M)}\bigr)(s,\cdot),
h \rangle_{L^2({\mathbb S}^1;\R^d)}
\bigr]
\Bigr]_{\vert \cdot = U_{t-t_{0}}^\ell}
\Bigr\|_{L^2(\SS^1;\R^d)}^2 \biggr]^{1/2}
\\
&\leq
\frac{C}{\sqrt{s-t}}
\EE_{0}
\biggl[
\Bigl\|
D \bigl[ \big\langle  \bigl( {\mathcal U}^{(N)}-
{\mathcal U}^{(M)}
\bigr)
(s,\cdot), h \big\rangle
 \bigr]_{\vert \cdot = U_{s-t_{0}}^\ell}
 \Bigr\|_{L^2({\mathbb S}^1;\R^d)}^2 
\biggr]^{1/2}.
\end{split}
\end{equation*}
Obviously, the same holds for the term on the fourth line in 
\eqref{eq:convergence:N:M:proof:1}.
\begin{equation*}
\begin{split}
&\EE_{0} \biggl[ \Bigl\|
D \Bigl[ {\mathcal P}_{s-t}
\bigl[ 
 \big\langle  D \langle 
( {\mathcal U}^{(N)}- {\mathcal U}^{(M)})(s,\cdot),h
\rangle_{L^2({\mathbb S}^1)}, {\mathcal U}^{(N)}(s,\cdot)
 \big\rangle_{L^2({\mathbb S}^1;\R^d)}
\bigr]
\Bigr]_{\vert \cdot = U_{t-t_{0}}^\ell}
\Bigr\|_{L^2(\SS^1;\R^d)}^2
 \biggr]^{1/2}
\\
&\leq
\frac{C}{\sqrt{s-t}}
\EE_{0}
\biggl[
\Bigl\|
D 
\bigl[
\big\langle  \bigl( {\mathcal U}^{(N)}-
{\mathcal U}^{(M)}
\bigr)
(s,\cdot), h \big\rangle
 \bigr]_{\vert \cdot = U_{s-t_{0}}^\ell}
 \Bigr\|_{L^2({\mathbb S}^1;\R^d)}^2 
\biggr]^{1/2}.
\end{split}
\end{equation*}
Finally, 
\begin{equation*}
\begin{split}
&\EE_{0} \biggl[ \Bigl\|
D \Bigl[ 
{\mathcal P}_{s-t}
\bigl[ 
 \big\langle D \langle {\mathcal U}^{(M)}
(s,\cdot),h
\rangle_{L^2({\mathbb S}^1;\R^d)}, 
\bigl(
{\mathcal U}^{(N)}-
{\mathcal U}^{(M)}
\bigr)
(s,\cdot)
 \big\rangle_{L^2({\mathbb S}^1;\R^d)}
\bigr]
\Bigr]_{\vert \cdot = U_{t-t_{0}}^\ell} 
\Bigr\|_{L^2(\SS^1;\R^d)}^2 \biggr]^{1/2}
\\
&\leq
\frac{C}{\sqrt{s-t}}
\EE_{0}
\Bigl[
\bigl\|
\bigl(
{\mathcal U}^{(N)}-
{\mathcal U}^{(M)}
\bigr)
(s,U_{s-t_{0}}^\ell)
 \bigr\|_{L^2({\mathbb S}^1;\R^d)}^2 
\Bigr]^{1/2}.
\end{split}
\end{equation*}
Collecting all the bounds and plugging them into 
\eqref{eq:convergence:N:M:proof:1}, we get
\begin{equation*}
\begin{split}
&\EE_{0} \biggl[ \Bigl\|
D
\bigl[
\big\langle  \bigl( {\mathcal U}^{(N)}-
{\mathcal U}^{(M)}
\bigr)
(t,\cdot), h \big\rangle
\bigr]_{\vert \cdot = U_{t-t_{0}}^\ell}
\Bigr\|_{L^2(\SS^1;\R^d)}^2
 \biggr]^{1/2}
\\ 
 &\leq 
\frac{C}{\sqrt{T-t}}
\biggl( \sup_{l \in \cK}
\sum_{n >N \wedge M} \vert l^{n,\pm} \vert^2 
+ 
\sup_{l \in \cK^{\varepsilon}}
\sum_{n >N \wedge M} \vert {\mathfrak F}^{n,\pm}(l) \vert^2 
+ 
\sup_{l \in \cK^{\varepsilon}}
\sum_{n >N \wedge M} \vert {\mathfrak G}^{n,\pm}(l) \vert^2 
\\
&\hspace{200pt} +
\sup_{l \in \cK}
 \sup_{r \in [0,T]} \PP_{0} \bigl( U_{r}^l \not \in \cK^{\varepsilon} \bigr)
 + \frac{1}{N \wedge M}
\biggr)^{1/2}
\\
&\hspace{15pt} + \int_{t}^T 
\frac{C}{\sqrt{s-t}}
\EE_{0}
\biggl[
\Bigl\|
D \bigl[ \big\langle  \bigl( {\mathcal U}^{(N)}-
{\mathcal U}^{(M)}
\bigr)
(s,\cdot), h \big\rangle
\bigr]_{\vert \cdot = U_{s-t_{0}}^\ell}
\Bigr\|_{L^2({\mathbb S}^1;\R^d)}^2 
\biggr]^{1/2} ds
\\
&\hspace{15pt} + \int_{t}^T 
\frac{C}{\sqrt{s-t}}
\EE_{0}
\Bigl[
\bigl\|
\bigl(
{\mathcal U}^{(N)}-
{\mathcal U}^{(M)}
\bigr)
(s,U_{s-t_{0}}^\ell)
 \bigr\|_{L^2({\mathbb S}^1;\R^d)}^2 
\Bigr]^{1/2} ds.
\end{split}
\end{equation*}
By Lemma \ref{lem:4:2:5:bis} below, we get
\begin{equation*}
\begin{split}
&\EE_{0} \biggl[ \Bigl\|
D
\bigl[
\big\langle  \bigl( {\mathcal U}^{(N)}-
{\mathcal U}^{(M)}
\bigr)
(t,\cdot), h \big\rangle
\bigr]_{\vert \cdot = U_{t-t_{0}}^\ell} 
\Bigr\|_{L^2(\SS^1;\R^d)}^2
 \biggr]^{1/2}
\\ 
 &\leq 
\frac{C}{\sqrt{T-t}}
\biggl[
\biggl( \sup_{l \in \cK}
\sum_{n >N \wedge M} \vert l^{n,\pm} \vert^2 
+ 
\sup_{l \in \cK^{\varepsilon}}
\sum_{n >N \wedge M} \vert {\mathfrak F}^{n,\pm}(l) \vert^2 
+ 
\sup_{l \in \cK^{\varepsilon}}
\sum_{n >N \wedge M} \vert {\mathfrak G}^{n,\pm}(l) \vert^2 
\\
&\hspace{200pt} +
\sup_{l \in \cK}
 \sup_{r \in [0,T]} \PP_{0} \bigl( U_{r}^l \not \in \cK^{\varepsilon} \bigr)
 + \frac{1}{N \wedge M}
\biggr)^{1/2}
\\
&\hspace{30pt} +  \sup_{s \in [t_{0},T]}
\EE_{0}
\Bigl[
\bigl\|
\bigl(
{\mathcal U}^{(N)}-
{\mathcal U}^{(M)}
\bigr)
(s,U_{s-t_{0}}^\ell)
 \bigr\|_{L^2({\mathbb S}^1;\R^d)}^2 
\Bigr]^{1/2} \biggr].
\end{split}
\end{equation*}
And then, using the boundedness of $\cU^{(N)}$ (and $\cU^{(M)}$), 
we obtain
\begin{equation*}
\begin{split}
&\EE_{0} \biggl[ \Bigl\|
D \bigl[
\big\langle  \bigl( {\mathcal U}^{(N)}-
{\mathcal U}^{(M)}
\bigr)
(t,\cdot), \ell \big\rangle
\bigr]_{\vert \cdot = U_{t-t_{0}}^h} 
\Bigr\|_{L^2(\SS^1;\R^d)}^2
 \biggr]^{1/2}
\\ 
 &\leq 
\frac{C}{\sqrt{T-t}}
\biggl[
\biggl( \sup_{l \in \cK}
\sum_{n >N \wedge M} \vert l^{n,\pm} \vert^2 
+ 
\sup_{l \in \cK^{\varepsilon}}
\sum_{n >N \wedge M} \vert {\mathfrak F}^{n,\pm}(l) \vert^2 
+ 
\sup_{l \in \cK^{\varepsilon}}
\sum_{n >N \wedge M} \vert {\mathfrak G}^{n,\pm}(l) \vert^2 
\\
&\hspace{70pt} +
\sup_{s \in [{0},T]}
\sup_{l \in \cK^{\varepsilon}}
\bigl\|
\bigl(
{\mathcal U}^{(N)}-
{\mathcal U}^{(M)}
\bigr)
(s,l)
 \bigr\|_{L^2({\mathbb S}^1;\R^d)}^2
+
\sup_{l \in \cK}
 \sup_{r \in [0,T]} \PP_{0} \bigl( U_{r}^l \not \in \cK^{\varepsilon} \bigr)
+ \frac1{N \wedge M}
\biggr)^{1/2} \biggr],
\end{split}
\end{equation*}
which completes the proof by taking $t=t_{0}$.
\end{proof}

\begin{corollary}
\label{cor:time:regularity:decoupling:field}
For any compact subset $\cK \subset L^2(\SS^1;\R^d)$, there exists a 
function $w : \RR_{+} \rightarrow \RR_{+}$ satisfying 
$\lim_{\delta \searrow 0} w(\delta) =0$ such that, for any $N \in \NN^*$, any $s,t \in [0,T]$
and any $\ell \in L^2(\SS^1;\R^d)$,
\begin{equation*}
\| \cU^{(N)}(t,\ell) - \cU^{(N)}(s,\ell) \|_{L^2(\SS^1;\R^d)}
\leq C w\bigl( \vert s-t \vert \bigr).  
\end{equation*}
\end{corollary}

\begin{proof}
Without any loss of generality, we can assume $t<s$. 
We then consider the solution $(X^{(N),t,\ell}_{r}(\cdot),Y^{(N),t,\ell}_{r}(\cdot),Z^{(N),t,\ell}_{r}(\cdot))_{t \leq r \leq T}$ of the forward-backward system 
\eqref{eq:mkv:sde:N} on the interval $[t,T]$ with $X_{t}^{(N),t,\ell}(\cdot)=\ell$ as initial condition.

We then have:
\begin{equation*}
\begin{split}
\cU^{(N)}(t,\ell) = \EE_{0}
\biggl[ \cU^{(N)}\bigl(s,X_{s}^{(N),t,\ell}(\cdot)\bigr)
+ \int_{t}^s {\mathfrak F}^{(N)}\bigl(X_{r}^{(N),t,\ell}(\cdot)\bigr) dr \biggr],
\end{split}
\end{equation*}
so that
\begin{equation*}
\begin{split}
\cU^{(N)}(t,\ell) - \cU^{(N)}(s,\ell) = 
\EE_{0}
\biggl[
\Bigl(
 \cU^{(N)}\bigl(s,X_{s}^{(N),t,\ell}(\cdot)\bigr)
-
\cU^{(N)}(s,\ell)
\Bigr)
+ \int_{t}^s {\mathfrak F}^{(N)}\bigl(X_{r}^{(N),t,\ell}(\cdot)\bigr) dr \biggr].
\end{split}
\end{equation*}
Recalling that the functions $({\mathfrak F}^{(N)})_{N \in \NN^*}$ are bounded, uniformly in $N \in \NN^*$, and invoking Lemma 
\ref{lem:lemma:ibp:UN}, we deduce that there exists a constant 
$C$ such that, for any $N \in \NN^*$, any $t \in [0,T]$
and any $\ell \in L^2(\SS^1;\R^d)$,
\begin{equation}
\label{eq:time:regularity:proof:1}
\| \cU^{(N)}(t,\ell) - \cU^{(N)}(s,\ell) \|_{L^2(\SS^1;\R^d)}
\leq C \Bigl( \vert s-t \vert + \EE_{0}
\bigl[ \| X_{s}^{(N),t,\ell} - \ell \|_{L^2(\SS^1;\R^d)} \bigr]
\Bigr). 
\end{equation}
We now recall   the
Fourier expansion of the 
 forward equation in 
\eqref{eq:mkv:sde:N}:
\begin{equation*}
\begin{split}
&dX_{r}^{(N),n,\pm} 
= \Bigl( 
{\mathbf 1}_{(n,\pm)=(0,+)}
{\mathfrak B}^{(N),0,+}\bigl(X_{r}^{(N)}(\cdot)\bigr) 
 - \cU^{(N),n,\pm}\bigl(r,X_{r}^{(N)}(\cdot)\bigr)
- (2 \pi n)^2 X_{r}^{(N),n,\pm}
\Bigr) ds + dW_{r}^{n,\pm}, 
\end{split}
\end{equation*}
for 
$r \in [t,T]$,
where, for the sake of simplicity, we omitted the indices $(t,\ell)$ in the notation and we just indicated the mode indices. We get:
\begin{equation}
\label{eq:time:regularity:proof:2}
\begin{split}
X_{s}^{(N),n,\pm} &= e^{ - (2 \pi n)^2 (s-t)} \ell^{n,\pm}
+ 
 \int_{t}^s 
  e^{  (2 \pi n)^2 (r-s) }
{\mathbf 1}_{(n,\pm)=(0,+)}
{\mathfrak B}^{(N),0,+}\bigl(X_{r}^{(N)}(\cdot)\bigr)
dr
\\
&\hspace{15pt}-  
 \int_{t}^s 
 e^{  (2 \pi n)^2 (r-s) }
 \cU^{(N),n,\pm}\bigl(r,X_{r}^{(N)}(\cdot)\bigr) dr
  + 
\int_{t}^s 
 e^{  (2 \pi n)^2 (r-s) }
dW_r^{n,\pm}. 
\end{split}
\end{equation}
Since the functions ${\mathfrak B}^{(N)}$ and ${\mathcal U}^{(N)}$
can be bounded independently of $N$, we deduce that:
\begin{equation}
\label{eq:time:regularity:proof:3}
\begin{split}
\EE_{0} \bigl[ \|X_{s}^{(N)} - \ell \|_{L^2}^2 \bigr]
&\leq
C \biggl[ \vert s-t \vert^2 
+
 \sum_{n \in \NN} \vert \ell^{n,\pm} \vert^2 
\Bigl( e^{ - (2 \pi n)^2 (s-t)} - 1 \Bigr)^2
+
\sum_{n \in \NN} \int_{t}^s e^{ 2 (2 \pi n)^2 (r-s)} dr
\biggr].
\end{split}
\end{equation}
Now,
\begin{equation}
\label{eq:time:regularity:proof:4}
\begin{split}
\sum_{n \in \NN} \vert \ell^{n,\pm} \vert^2 
\Bigl( e^{ - (2 \pi n)^2 (s-t) } - 1 \Bigr)^2
&\leq 
C \sum_{n \in \NN} \vert \ell^{n,\pm} \vert^2 
\bigl[ 1 \wedge \bigl( n^2 (s-t) \bigr) \bigr]^2
\\
&\leq 
C 
\biggl[ \vert s-t \vert
\sum_{n \in \NN} \vert \ell^{n,\pm} \vert^2 
+
\sum_{n \geq (s-t)^{-1/4}} \vert \ell^{n,\pm} \vert^2 \biggr].
\end{split}
\end{equation}
Also, allowing the constant $C$ to change from line to line, we get
\begin{equation}
\label{eq:time:regularity:proof:5}
\begin{split}
\sum_{n \in \NN} 
\int_{t}^s e^{ - 2 (2 \pi n)^2 (s-r)} dr
&\leq C (s-t) + C 
\int_{0}^{\infty }\int_{t}^s e^{ - 2 (2 \pi x)^2 (s-r)} dr dx
\leq C(s-t)^{1/2}.
\end{split}
\end{equation}
Collecting 
\eqref{eq:time:regularity:proof:1},
\eqref{eq:time:regularity:proof:3},
\eqref{eq:time:regularity:proof:4}
and
\eqref{eq:time:regularity:proof:5}, 
we finally obtain:
\begin{equation*}
\| \cU^{(N)}(t,\ell) - \cU^{(N)}(s,\ell) \|_{L^2(\SS^1;\R^d)}
\leq C 
\bigl( 1 + \sup_{l \in \cK} \| l \|_{L^2(\SS^1;\R^d)}^2
\bigr)
\Bigl( \vert s-t \vert^{1/4} + 
\sup_{l \in \cK}
\sum_{n \geq (s-t)^{-1/4}} \vert l^{n,\pm} \vert^2
\Bigr),
\end{equation*}
which completes the proof.
\end{proof}

Here are now the two variants of Gronwall's lemma we appealed to right above. 

\begin{lemma}
\label{lem:4:2:5}
Consider two bounded measurable functions $g_1,g_2 : [0,T] \rightarrow \R_+$ such that
\begin{equation}
\label{eq:4:2:30}
g_1(t) \leq C_1 + C_2 \int_t^T \frac{g_2(s)}{\sqrt{s-t}} ds,
\end{equation}
for some constants $C_1,C_2 \geq 0$. Then there exist $\lambda,\mu > 0$, depending on $C_2$ and $T$ only, such that
\begin{equation}
\label{eq:4:2:31}
\begin{split}
&\int_0^T  g_1(t) \exp(\lambda t) dt \leq \mu C_1 + \frac{1}{2} \int_0^T g_2(t) \exp(\lambda t) dt,
\\
&\sup_{0 \leq t \leq T} \bigl[g_{1}(t)\bigr] \leq \mu C_1 + 2 C_{2}^2 \int_0^T g_2(t) dt +  \frac{1}{2} \sup_{0 \leq t \leq T} \bigl[g_{2}(t)\bigr].
\end{split}
\end{equation}
In particular, if $g_1=g_2$, then $g_1$ is bounded by $\mu' C_1$, for a constant 
$\mu'$ depending on $C_2$ and $T$ only.
 \end{lemma}
 
 \begin{lemma}
\label{lem:4:2:5:bis}
Consider two bounded measurable functions $g_1,g_2 : [0,T] \rightarrow \R_+$ such that
\begin{equation}
\label{eq:4:2:30:bis}
g_1(t) \leq \frac{C_{1}}{\sqrt{T-t}} + C_2 \int_t^T \frac{g_2(s)}{\sqrt{s-t}} ds,
\end{equation}
for some constants $C_1,C_2 \geq 0$. Then there exist $\lambda,\mu > 0$, depending on $C_2$ and $T$ only, such that
\begin{equation}
\label{eq:4:2:31:bis}
\begin{split}
&\int_0^T  g_1(t) \exp(\lambda t) dt \leq \mu C_1 + \frac{1}{2} \int_0^T g_2(t) \exp(\lambda t) dt,
\\
&\sup_{0 \leq t \leq T} \bigl[\sqrt{T-t} \, g_{1}(t)\bigr] \leq \mu C_1 + \mu \int_0^T g_2(t) dt +  \frac{1}{2} \sup_{0 \leq t \leq T} \bigl[\sqrt{T-t} \, g_{2}(t)\bigr].
\end{split}
\end{equation}
In particular, there exists a 
constant 
$\mu'$ depending on $C_2$ and $T$ only such that, whenever $g_{1}=g_{2}$, 
\begin{equation*}
\sup_{0 \leq t \leq T} \bigl[\sqrt{T-t} \, g_{1}(t)\bigr] \leq \mu' C_1.
\end{equation*}
 \end{lemma}

We just prove the second statement. The proof of the first one may be found 
in \cite[Lemma 2.13]{MR2973334}.

\begin{proof}
The first part of Lemma \ref{lem:4:2:5:bis}
may be proved as in \cite[Lemma 2.13]{MR2973334}.
So, we focus on the second inequality. For any $\varepsilon>0$, \eqref{eq:4:2:30:bis} yields
\begin{equation*}
\begin{split}
&(T-t)^{1/2}  g_1(t) 
\\
&\leq C_1 + C_2 \int_t^{(t+\varepsilon) \wedge T} \frac{(T-t)^{1/2}}{(s-t)^{1/2} (T-s)^{1/2}} (T-s)^{1/2} g_2(s) ds + C_2 \varepsilon^{-1/2} \int_{(t+\varepsilon) \wedge T}^T g_2(s) ds
\\
&\leq C_1 + C_2 \varepsilon^{-1/2} \int_{0}^T g_{2}(s) ds + C_2 
 \sup_{0 \leq s \leq T} \bigl[(T-s)^{1/2} g_2(s)\bigr]
 \int_t^{(t+\varepsilon) \wedge T} \frac{(T-t)^{1/2}}{(s-t)^{1/2} (T-s)^{1/2}}  ds.
\end{split}
\end{equation*}
Now,
\begin{equation*}
\begin{split}
\int_t^{(t+\varepsilon) \wedge T} \frac{(T-t)^{1/2}}{(s-t)^{1/2} (T-s)^{1/2}}  ds
 &= \int_0^{\varepsilon \wedge (T-t)} \frac{(T-t)^{1/2}}{s^{1/2} (T-t-s)^{1/2}}  ds
 \\
&= (T-t)^{1/2} \int_0^{1 \wedge [\varepsilon/(T-t)]} \frac{1}{s^{1/2} (1-s)^{1/2}}  ds
 \end{split}
 \end{equation*}
If $\varepsilon^{1/2} \leq T-t$, then
\begin{equation*}
\begin{split}
\int_t^{(t+\varepsilon) \wedge T} \frac{(T-t)^{1/2}}{(s-t)^{1/2} (T-s)^{1/2}}  ds
 &\leq  T^{1/2} \int_0^{1 \wedge \varepsilon^{1/2}} \frac{1}{s^{1/2} (1-s)^{1/2}}  ds.
 \end{split}
 \end{equation*} 
Otherwise, $T-t \leq \varepsilon^{1/2}$ and 
\begin{equation*}
\begin{split}
\int_t^{(t+\varepsilon) \wedge T} \frac{(T-t)^{1/2}}{(s-t)^{1/2} (T-s)^{1/2}}  ds
 &\leq  \varepsilon^{1/4} \int_0^{1} \frac{1}{s^{1/2} (1-s)^{1/2}}  ds.
 \end{split}
 \end{equation*} 
 So, we can find a function $\delta : \RR_{+} \rightarrow \RR_{+}$ converging to $0$ in $0$ such that 
 \begin{equation*}
\begin{split}
&(T-t)^{1/2}  g_1(t) 
\leq C_1 + C_2 \varepsilon^{-1/2} \int_{0}^T g_{2}(s) ds + C_2 
\delta(\varepsilon)
 \sup_{0 \leq s \leq T} \bigl[(T-s)^{1/2} g_2(s)\bigr].
\end{split}
\end{equation*}
The proof of the second claim is easily completed.
Whenever $g_{1}=g_{2}$, 
\begin{equation*}
\int_{0}^T g_{1}(t) \exp(\lambda t) dt \leq 2 C_{1} \mu,
\end{equation*}
and then, choosing $\varepsilon$ small enough in the second claim, we get by the first part of the statement:
\begin{equation*}
\begin{split}
\sup_{0 \leq t \leq T} \bigl[ \sqrt{T-t} \, g_{1}(t) 
\bigr] &\leq 2 \mu C_{1} + 2 \mu \int_{0}^T g_{1}(t) dt
\\
&\leq 2 \mu C_{1} + 2 \mu \int_{0}^T g_{1}(t) \exp(\lambda t) dt
\leq 2 \mu C_{1} + 4 C_{1} \mu^2,  
\end{split}
\end{equation*}
which completes the proof.
\end{proof}

\subsection{End of the proof of Theorem 
\ref{thm:existence:uniqueness}}
We now turn to the proof of Theorem \ref{thm:existence:uniqueness}. 
To this end, we recall the constant $C$ from Lemma \ref{lem:lemma:ibp:UN}. Without any loss of generality, we assume that the Lipschitz constants of the coefficients 
${\mathfrak b}_{0}$, ${\mathfrak F}$ and ${\mathfrak G}$ are less than the same constant $C$.  
We then call $c$ the constant in the statement 
of Theorem \ref{eq:small:time}
when the Lipschitz constant of the coefficients is less than $C$.

We let 
$N=\lceil T/c\rceil$
and $\tau_{n} = T - (N-n) c$ for $n \in \{1,\dots,N\}$
and $\tau_{0}=0$. 
We know from Theorem \ref{eq:small:time} that, for any 
square-integrable 
$\cF_{0,\tau_{N-1}}$-measurable 
initial condition $X^{(N-1)}(\cdot)$ with values in $L^2(\SS^1;\R^d)$, 
the forward-backward system 
\eqref{eq:mkv:sde}
is uniquely solvable. 
Following Lemma 
\ref{lem:decoupling:field}, this permits to define the decoupling field 
$\cU$ on $[\tau_{N-1},T] \times L^2(\SS^1;\R^d)$. 
By 
\eqref{eq:stability:fbsde:galerkin}, we know that, for any 
$(t,\ell) \in [\tau_{N-1},T] \times L^2(\SS^1;\R^d)$, the sequence  
$(\cU^{(N)}(t,\ell))_{N \in \NN^*}$, defined 
as the sequence of decoupling fields of the systems 
\eqref{eq:mkv:sde:N}, converges to $\cU(t,\ell)$. 
In particular, we deduce from Lemma \ref{lem:lemma:ibp:UN} 
that $\cU$ is $C$-Lipschitz in the space variable on $[\tau_{N-1},T] 
\times L^2(\SS^1;\R^d)$. 

Since $\cU(\tau_{N-1},\cdot)$ is $C$-Lipschitz, we can iterate the argument and apply Theorem \ref{eq:small:time} on the interval 
$[\tau_{N-2},\tau_{N-1}]$. This permits to extend the definition 
of the decoupling field $\cU$ to the set $[\tau_{N-2},\tau_{N-1}] \times L^2(\SS^1;\R^d)$. By invoking \eqref{eq:stability:fbsde:galerkin}
once again but on $[\tau_{N-2},\tau_{N-1}]$, we deduce that, for any $(t,\ell) \in [\tau_{N-2},\tau_{N-1}] \times L^2(\SS^1;\R^d)$, 
the sequence  
$(\cU^{(N)}(t,\ell))_{N \in \NN^*}$ converges to $\cU(t,\ell)$, which permits to iterate the argument and, in the end, to construct a candidate $\cU$ for being the decoupling field on the entire $[0,T] \times L^2(\SS^1;\R^d)$. 
Once $\cU$ has been constructed, the proof is completed as in the finite dimensional case, see for instance 
\cite{Delarue02}
and \cite[Chapter 4]{CarmonaDelarue_book_I}.

\subsection{Proof of Theorem \ref{thm:pde}}

\textit{First Step.}
As a by-product of the 
analysis achieved in the previous subsection to complete 
the proof of Theorem \ref{thm:existence:uniqueness}, we claim that, for any $(t,\ell) \in [0,T] \times L^2(\SS^1;\R^d)$,
\begin{equation*}
\lim_{M,N \rightarrow \infty} \| (\cU^{(N)} - \cU^{(M)})(t,\ell) \|_{L^2(\SS^1;\R^d)} = 0.
\end{equation*}
Recall from Lemmas
\ref{lem:bound:pp}
and  \ref{lem:lemma:ibp:UN}
and Corollary \ref{cor:time:regularity:decoupling:field} that the mappings $(\cU^{(N)})_{N \in \NN^*}$ are uniformly bounded and uniformly continuous on any compact subset of $L^2(\SS^1;\R^d)$. Hence, we have:
\begin{equation*}
\lim_{M,N \rightarrow \infty}
\sup_{t \in [0,T]}
\sup_{\ell \in \cK}
\| (\cU^{(N)} - \cU^{(M)})(t,\ell) \|_{L^2(\SS^1;\R^d)}
= 0.
\end{equation*}
We now invoke Lemma 
\ref{lem:convergence:gradients:galerkin}, from which we deduce that
for any compact subset 
of $[0,T) \times L^2(\SS^1;\R^d)$, il holds that:
\begin{equation*}
\lim_{M,N \rightarrow \infty}
\sup_{(t,\ell) \in \cK}
\bigl\vvvert (D \cU^{(N)} - D\cU^{(M)})(t,\ell) \bigr\vvvert_{L^2(\SS^1;\RR^d) \times L^2(\SS^1;\RR^d)}
= 0,
\end{equation*}
which shows that the sequence $(D \cU^{(N)})_{N \in \NN^*}$ converges, uniformly on compact subsets of $[0,T) \times L^2(\SS^1;\R^d)$. 
Since each $D \cU^{(N)}$ is continuous on 
$[0,T) \times L^2(\SS^1;\R^d)$, we deduce that the limit, denoted by $D \cU$ is continuous and is the Fr\'echet derivative of $\cU$ in the space variable.  Of course, $D \cU$ satisfies Lemma 
\ref{lem:lemma:ibp:UN}.
Passing to the limit in 
\eqref{eq:galerkin:mild:formulation}, we deduce that 
$\cU$ is a mild solution of the system of PDEs \eqref{eq:se:2:PDE:infinite}, as formulated in the statement of Theorem 
\ref{thm:pde}. 
\vskip 4pt

\section{Construction of an approximated Nash equilibrium}
\label{se:construction:approximate:proof}

The purpose of this section is to prove Theorem 
\ref{thm:approx:nash}. To do so, we use the same setting as in
Subsection \ref{subse:connection:game:statement}, a short reminder of which is recalled below.

The game consists of $N A_{N}$ particles that are uniformly distributed along the points
(which we call roots)
$(e^{i 2 \pi k/N})_{k=0,\cdots,N-1}$ of the unit circle, with $i^2=-1$ and  with exactly $A_{N}$ particles per root, where 
$A_{N} \in \NN^*$.  
States of the particles at time $t$ are denoted by 
$(X_{t}^{k,j})_{k=0,\cdots,N-1;j=1,\cdots,A_{N}}$, where $k$ stands for the index of the root occupied by the particle and $j$ for its label among the collection of particles located at the same site. As already explained in Subsection 
\ref{subse:connection:game:statement}, 
we put $X_{t}^{k+\ell N,j}=X_{t}^{k,j}$,
for $k \in \{0,\dots,N-1\}$ and $\ell \in {\mathbb Z}$. 

Each particle $(k,j)$ 
has dynamics of the following form:
 \begin{equation}
 \label{eq:sec:5:definition:game}
 d X_{t}^{k,j} =
 \biggl\{ b \bigl( \bar{\mu}_{t}^N 
 \bigr) + \alpha_{t}^{k,j}
 + N
\sum_{l=1}^{A_{N}}
 \bigl(   X^{k+1,l+1}_{t} +
X^{k+1,l-1}_{t}
 -
2 X^{k+1,l}_{t}
 \bigr) 
 \biggr\} dt + \sqrt N dB_{t}^{k},
 \end{equation}
 for $t \in [0,T]$,  
with the initial condition 
$X_{0}^{k,j} = \bar X_{0}^k$, 
 where $(\bar X_{0}^{k})_{k=0,\cdots,N-1}$
 are given by:
 \begin{equation}
 \label{se:5:approximation:initial:condition}
 \bar X_{0}^k = N \int_{k/N}^{(k+1)/N}
 X_{0}(x) dx, \quad k =0,\cdots,N-1,
 \end{equation}
whilst the noises 
 $(\bmf{B}^{k}= (B_{t}^{k})_{0 \leq t \leq T})_{k=0,\cdots,N-1}$ are independent $d$-dimensional Brownian motions on 
 the interval $[0,T]$ with the following definition:
 \begin{equation*}
 B_{t}^k = \sqrt{N} \int_{k/N}^{(k+1)/N}
W_{t}(dx). 
 \end{equation*}
We recall that $\bar{\mu}^N_{t}$ denotes the empirical distribution: 
 \begin{equation*}
 \bar{\mu}^N_{t} = 
 \frac{1}{N A_{N}}
 \sum_{k=0}^{N-1}
 \sum_{j=1}^{A_{N}}
 \delta_{X_{t}^{k,j}}. 
 \end{equation*}
The processes $(\balpha^{k,j}= (\alpha^{k,j}_{t})_{0 \leq t \leq T})_{k=0,\cdots,N-1;j=1,\cdots,A_{N}}$
 are constructed on $(\Omega_{0},{\mathcal A}_{0},\PP_{0})$
 and are 
$\RR^d$-valued progressively-measurable 
controls with respect to 
 the filtration generated by the cylindrical white noise $(W_{t}(\cdot))_{0 \le t \le T}$
 satisfying the condition:
 \begin{equation*} 
\EE_{0} \int_{0}^T  \vert \alpha_{t}^{k,j} \vert^2 dt < \infty.
\end{equation*}
The cost functional to player $(k,j)$ is then given by: 
\begin{equation*}
J^{k,j}\bigl((\balpha^{k',j'})_{k'=0,\cdots,N-1;j'=1,\cdots,A_{N}}\bigr)
= \E_{0} \biggl[ g\bigl(X_{T}^{k,j},\bar{\mu}^N_{T}\bigr)
+ \int_{0}^T 
\Bigl(
f\bigl(X_{t}^{k,j}, \bar{\mu}^N_{t}\bigr)
+ \frac12 \vert \alpha_{t}^{k,j} \vert^2
\Bigr) dt \biggr].
\end{equation*}
Following the statement of Theorem 
\ref{thm:approx:nash}, we introduce the collection of controls:
\begin{equation}
 \label{eq:sec:5:definition:game:Y}
\alpha^{\star k,j}_{t} = 
\bar Y_{t}^k, \quad 
\bar Y_{t}^k = 
N \int_{(k-1)/N}^{k/N}
Y_{t}(x) dx, \quad t \in [0,T],
\end{equation}
for all $k \in \{0,\cdots,N-1\}$ and $j \in \{1,\cdots,A_{N}\}$. Then, for some
$k_{0} \in \{0,\cdots,N-1\}$ and $j_{0} \in \{1,\cdots,A_{N}\}$ and for some
$\RR^d$-valued 
process ${\boldsymbol \gamma}=(\gamma_{t})_{0 \le t \le T}$ that is 
progressively-measurable 
with respect to 
 the filtration generated by the cylindrical white noise ${\boldsymbol W}(\cdot)=(W_{t}(\cdot))_{0 \le t \le T}$
(that is, the filtration generated by the processes $(\langle W_{t}(\cdot),h \rangle_{L^2(\SS^1;\RR^d)})_{0 \leq t \leq T}$
for $h \in L^2(\SS^1;\R^d)$)  and that satisfies the condition
 \begin{equation*} 
\EE_{0} \int_{0}^T  \vert \gamma_{t} \vert^2 dt < \infty,
\end{equation*}
we let
$\bbeta^{\star k,j} = \balpha^{\star k,j}$, 
for $k \in \{0,\cdots,N-1\}$ and 
$j \in \{1,\cdots,A_{N}\}$, 
with $(k,j) \not = (k_{0},j_{0})$. 
When $k=k_{0}$ and $j=j_{0}$, 
we let 
$\bbeta^{\star k_{0},j_{0}} = {\boldsymbol \gamma}$.

The goal of this section is to prove that there exists a sequence of positive reals $(\varepsilon_{N})_{N \in \NN^*}$, converging to $0$, independent of 
${\boldsymbol \gamma}$, $k_{0}$ and $j_{0}$, such that 
\begin{equation*}
J^{k_{0},j_{0}}\bigl((\bbeta^{\star k,j})_{k=0,\cdots,N-1;j=1,\cdots,A_{N}}\bigr)
\geq 
J^{k_{0},j_{0}}\bigl((\balpha^{\star k,j})_{k=0,\cdots,N-1;j=1,\cdots,A_{N}}\bigr)
- \varepsilon_{N}. 
\end{equation*}

Throughout the analysis, we assume that, 
on top of 
Assumption {\bf (A)}, 
$f$ and $g$ are Lipschitz continuous in $\mu$, uniformly in $x$. 
In particular, $f$ and $g$ are Lipschitz in $(x,\mu)$.

\subsection{Distance between discrete and continuous systems}
Most of the proof relies on a stability property under discretization for SPDEs of the form:
\begin{equation}
\label{eq:SPDE:alpha}
\partial_{t} X_{t}(x) = \alpha_{t}(x)  + \Delta X_{t}(x)  + \dot{W}_{t}(x), \quad (t,x) \in [0,T] 
\times \SS^1,
\end{equation}
with some initial condition $X_{0}(\cdot) \in L^2(\SS^1;\RR^d)$. Above, the process 
$\balpha(\cdot) = (\alpha_{t}(\cdot))_{0 \le t \le T}$
is an $L^2(\SS^1;\R^d)$-valued progressively-measurable process with respect to the filtration generated by $(W_{t}(\cdot))_{0 \leq t \leq T}$. We assume it to satisfy 
\begin{equation*}
\EE_{0} \int_{0}^T \| \alpha_{t}(\cdot) \|^2_{L^2(\SS^1;\R^d)} dt < \infty. 
\end{equation*}
The solution to 
\eqref{eq:SPDE:alpha} will be denoted 
$(X_{t}^{(\balpha)}(\cdot))_{0 \le t \le T}$. For another 
$L^2(\SS^1;\R^d)$-valued
progressively-measurable process
$\bbeta(\cdot) = (\beta_{t}(\cdot))_{0 \le t \le T}$ satisfying 
\begin{equation*}
\EE_{0} \int_{0}^T \| \beta_{t}(\cdot) \|^2_{L^2(\SS^1;\R^d)} dt < \infty, 
\end{equation*}
we let
\begin{equation}
\label{se:5:approximation:beta}
\bar \beta_{t}^k
=
N \int_{k/N}^{(k+1)/N} \beta_{t}(x) dx, 
\quad t \in [0,T], \quad k \in \{0,\cdots,N-1\},  
\end{equation}
and we consider the discretized version 
\begin{equation}
\label{eq:discrete:eq}
d\bar X_{t}^k = \bar \beta_{t}^k dt + N^2
\bigl(\bar X_{t}^{k+1}+\bar   X_{t}^{k-1}-2 \bar X_{t}^k
\bigr) dt + \sqrt{N} dB_{t}^k,
\end{equation}
for 
$t \in [0,T]$ and $k \in \{0,\cdots,N-1\}$,
with the same convention as before that 
$\bar X_{t}^{-1}=\bar X_{t}^{N-1}$
and
$\bar X_{t}^{N}=\bar X_{t}^{0}$. Above the 
initial condition is given by the same approximation as in 
\eqref{se:5:approximation:initial:condition}.
The solution to 
\eqref{eq:discrete:eq}
will be denoted $((\bar X_{t}^{(\bbeta),k})_{k=0,\cdots,N-1})_{0 \le t \le T}$. 
With this solution, we associate the periodic function
\begin{equation*}
\bar X_{t}^{(\bbeta)}(\cdot)
=
\sum_{k=0}^{N-1}
\bar X_{t}^{(\bbeta),k} {\mathbf 1}_{[k/N,(k+1)/N)+{\mathbb Z}}(\cdot), \quad t \in [0,T].
\end{equation*}
Notice that (and this is the key point of the proof) the equation 
\eqref{eq:discrete:eq} is just indexed by the label $k$ of the root (and not by the label $j$ we used before to denote a particle).


\subsubsection*{Mild solution of the discrete equation}
Equation \eqref{eq:discrete:eq} forms a system of stochastic differential equations, the solution of which may be put under a discrete mild form,  the mild formulation being based upon the following operator:
\begin{equation*}
\Delta^{(N)} \Bigl( 
\sum_{k=0}^{N-1}
\bar \lambda ^{k} {\mathbf 1}_{[k/N,(k+1)/N) + \ZZ}(\cdot)
\Bigr)
= 
\sum_{k=0}^{N-1}
N^2 \bigl( \bar \lambda^{k+1}+ \bar \lambda^{k-1} - 
2 \bar \lambda^{k}  \bigr) 
{\mathbf 1}_{[k/N,(k+1)/N) + \ZZ}(\cdot),
\end{equation*}
for any sequence $(\bar \lambda ^{k})_{k=0,\cdots,N-1}$. 
Obviously, $\Delta^{(N)}$ is acting on piecewise constant functions 
from the torus $\SS^1$ into $\RR$
(or, more generally, into $\R^d$)
 with 
$(k/N +  \ZZ)_{k=0,\cdots,N-1}$ as mesh. 
We often identify these functions with piecewise constant functions from $[0,1)$ into $\RR$ (or $\R^d$) with 
$(k/N +  \ZZ)_{k=0,\cdots,N-1}$ as mesh, in which case the above identity becomes
(with a slight abuse of notation):
\begin{equation*}
\Delta^{(N)} \Bigl( 
\sum_{k=0}^{N-1}
\bar \lambda^{k} {\mathbf 1}_{[k/N,(k+1)/N)}(\cdot)
\Bigr)
= 
\sum_{k=0}^{N-1}
N^2 \bigl( \bar \lambda^{k+1}+ \bar \lambda^{k-1} - 
2 \bar \lambda^{k}  \bigr) 
{\mathbf 1}_{[k/N,(k+1)/N)}(\cdot),
\end{equation*}
Throughout the analysis, we shall use the following convention. For a collection of weights 
$(\bar \lambda^k)_{k=0,\cdots,N-1}$ (with values in $\RR$ or in $\R^d$), we call
\begin{equation}
\label{se:5:approximation:lambda}
\bar{\lambda}(\cdot) = \sum_{k=0}^{N-1} \bar{\lambda}^k {\mathbf 1}_{[k/N,(k+1)/N) + \ZZ}(\cdot)
\end{equation}
the corresponding piecewise constant step functions on the torus. Observe that, for the sake of convenience, we omitted to specify the dependence of the function $\bar{\lambda}(\cdot)$ upon the discretization parameter $N$. Also, according to the previous convention, we shall identify the function $\bar \lambda(\cdot)$ with the function
$\sum_{k=0}^{N-1}
\bar{\lambda}^k {\mathbf 1}_{[k/N,(k+1)/N) }(\cdot)$ from $[0,1)$ into $\RR$. 
With this convention of notation, the solution to \eqref{eq:discrete:eq} may be written under the form:
\begin{equation}
\label{eq:discrete:mild:formulation}
\bar{X}_{t}^{(\bbeta)}(\cdot)
= e^{t\Delta^{(N)}} \bar X_{0}(\cdot)
+ \int_{0}^t 
e^{(t-s)\Delta^{(N)}} \bar \beta_{s}(\cdot) ds +
\int_{0}^t 
e^{(t-s)\Delta^{(N)}} 
\Bigl( 
\sum_{n \in {\mathbb N}}
\bar e^{n,\pm}(\cdot) 
dW_{s}^{n,\pm}
\Bigr),
 \end{equation}
with the same   convention as before for the notation $\bar{e}^{n,\pm}(\cdot)$, namely:
\begin{equation*}
\bar{e}^{n,\pm}(\cdot)  = \sum_{k=0}^{N-1}
\bar{e}^{n,\pm,k} 
{\mathbf 1}_{[k/N,(k+1)/N)}
(\cdot), \quad 
\textrm{\rm with}
\quad 
\biggl(
\bar{e}^{n,\pm,k}
=
N \int_{k/N}^{(k+1)/N}e^{n,\pm}(x) 
dx \biggr)_{k=0,\cdots,N-1},
\end{equation*} 
which is to say that $\bar{e}^{n,\pm}(\cdot)$ is the piecewise constant step function associated with the family of weights $(\bar e^{n,k,\pm})_{k=0,\cdots,N-1}$. 
 
The above writing of the stochastic integral is justified by the fact that
 \begin{equation*} 
\sqrt{N} \bigl( \sum_{k=0}^{N-1}
 B_{t}^k 
{\mathbf 1}_{[k/N,(k+1)/N)}(\cdot)
\bigr) = \sum_{n \in \NN} W_{t}^{n,\pm} 
\biggl[
\sum_{k=0}^{N-1} \bar e^{n,\pm,k} {\mathbf 1}_{[k/N,(k+1)/N)}(\cdot)
\biggr],
\end{equation*} 
which follows from a straightforward application of the decomposition of $W$ in Fourier modes, namely
\begin{equation*}
\sqrt{N}
 {B}_{t}^k = N \int_{k/N}^{(k+1)/N}
W_{t}(dx)  = \sum_{n \in \NN} W_{t}^{n,\pm}
\bar{e}^{n,\pm,k}. 
\end{equation*} 

\subsubsection*{Distance between $X^{(\balpha)}$ and $\bar X^{(\bbeta)}$}

For the sake of completeness, we recall the mild formulation of the 
SPDE 
\eqref{eq:SPDE:alpha}:
\begin{equation}
\label{eq:SPDE:mild:formulation}
 {X}_{t}^{(\balpha)}(\cdot)
= e^{t\Delta }  X_{0}(\cdot)
+ \int_{0}^t 
e^{(t-s)\Delta } \alpha_{s}(\cdot) ds +
\int_{0}^t 
e^{(t-s)\Delta} 
\Bigl( 
\sum_{n \in {\mathbb N}}
e^{n,\pm}(\cdot) 
dW_{s}^{n,\pm}
\Bigr), \quad t \in [0,T]. 
\end{equation}

Here is the main statement of this subsection. 
\begin{theorem}
\label{thm:discretization}
There exist a constant $C$ 
together with a sequence $(\varepsilon_{N})_{N \in \NN^*}$, converging to $0$ as $N$ tends to $\infty$,
such that for any initial condition 
$X_{0}(\cdot) \in L^2(\SS^1;\R^d)$, any two square-integrable progressively-measurable process $(\alpha_{t}(\cdot))_{0 \le t \le T}$
and 
$(\beta_{t}(\cdot))_{0 \le t \le T}$
with values in $L^2(\SS^1;\R^d)$
and any integer $N \in \NN^*$, 
it holds
\begin{equation}
\label{eq:discretization:2}
\begin{split}
\sup_{x \in \SS}
\EE_{0} \bigl[
\vert \bar{X}_{t}^{(\bbeta)}(x) - X_{t}^{(\balpha)}(x) 
\vert^2
\bigr] \Bigr\}
&\leq C 
\biggl( 1 + \frac{1}{t^{3/4}}
\| X_{0}(\cdot)
\|_{L^2(\SS^1;\R^d)}^2 + 
\EE_{0} \int_{0}^t 
\| \alpha_{s}(\cdot) \|_{L^2(\SS^1;\R^d)}^2 ds 
\biggr) \varepsilon_{N}
\\
&\hspace{15pt}
+ C 
\EE_{0} \int_{0}^t 
\bigl\| \bigl( \alpha_{s} - \beta_{s}\bigr)(\cdot) \bigr\|_{L^2(\SS^1;\R^d)}^2 ds,
\end{split}
\end{equation}
for all $t \in (0,T]$.
\end{theorem}

\begin{proof}
The proof is split in several steps. The goal 
is to compare \eqref{eq:discrete:mild:formulation}
and 
\eqref{eq:SPDE:mild:formulation}. Basically, each step of the proof corresponds 
to the comparison of a pair of terms in the right-hand sides of 
\eqref{eq:discrete:mild:formulation}
and 
\eqref{eq:SPDE:mild:formulation}.
\vspace{4pt}

\textit{Preliminary Step.}
As a preliminary step, we have the following two standard results, the proofs of which are postponed to the end of the subsection. 

The first identity is
\begin{equation}
\label{eq:preliminary:identity:1}
\begin{split}
&\bar{e}^{n}(\cdot)
=
e^{i \pi \frac{n}{N}}
\frac{\sin(\pi n/N)}{\pi n/N}
\sum_{k=0}^{N-1}
e^{i 2\pi \frac{k n}{N}}
 {\mathbf 1}_{[k/N,(k+1)/N)}(\cdot), \quad \textrm{\rm with} \ \bar{e}^n(\cdot) = \frac{
  \bar{e}^{n,+} + i \bar{e}^{n,-}}{\sqrt{2}}(\cdot),
\end{split}
\end{equation}
and $i^2=-1$.
The second one is
\begin{equation}
\label{eq:preliminary:identity:2}
\begin{split}
&\Delta^{(N)}
\biggl[
\sum_{k=0}^{N-1}
e^{i 2\pi \frac{k n}{N}}
 {\mathbf 1}_{[k/N,(k+1)/N)}(\cdot)
\biggr]
=- 2 N^2 \bigl[ 
1- \cos\bigl( \frac{2 \pi  n}{N}
\bigr) 
\bigr]
\biggl[
\sum_{k=0}^{N-1}
e^{i 2\pi \frac{k n}{N}}
 {\mathbf 1}_{[k/N,(k+1)/N)}(\cdot)
\biggr],
\end{split}
\end{equation}
which shows that the function
$\sum_{k=0}^{N-1}
e^{i 2\pi \frac{k n}{N}}
 {\mathbf 1}_{[k/N,(k+1)/N)}(\cdot)$ is an eigenvector of $\Delta^{(N)}$. In particular, we have
\begin{equation*}
\begin{split}
&e^{(t-s)\Delta^{(N)}}
\biggl[
\sum_{k=0}^{N-1}
e^{i 2\pi \frac{k n}{N}}
 {\mathbf 1}_{[k/N,(k+1)/N)}(\cdot)
\biggr]
=e^{-2N^2 [ 
1- \cos( 2 \pi  n/N) ]
(t-s) }
\biggl[
\sum_{k=0}^{N-1}
e^{i 2\pi \frac{k n}{N}}
 {\mathbf 1}_{[k/N,(k+1)/N)}(\cdot)
\biggr],
\end{split}
\end{equation*} 
for any $s,t \in [0,T]$, with $0 \leq s \leq t$. Combining with the first identity
\eqref{eq:preliminary:identity:1}, we get:
\begin{equation*}
\begin{split}
e^{(t-s)
\Delta^{(N)}} \bar e^{n}(\cdot) 
&= 
e^{i \pi \frac{n}{N}}
\frac{\sin(\pi n/N)}{\pi n/N}
e^{-2 N^2 [ 
1- \cos( 2 \pi  n/N) ]
(t-s) }
\biggl[
\sum_{k=0}^{N-1}
e^{i 2\pi \frac{k n}{N}}
 {\mathbf 1}_{[k/N,(k+1)/N)}(\cdot)
\biggr]
\\
&= 
e^{-2 N^2 [ 
1- \cos( 2 \pi  n/N) ]
(t-s) }
 \bar e^{n}(\cdot),
 \end{split}
\end{equation*}
which shows that $\bar e^{n}(\cdot)$ is also an eigenvector of $\Delta^{(N)}$. 
Taking the real and imaginary parts, the same holds for 
$\bar e^{n,+}(\cdot)$
and
$\bar e^{n,-}(\cdot)$.
\vspace{4pt}

\textit{Second Step.}
We now compare the martingale terms in 
\eqref{eq:discrete:mild:formulation}
and
\eqref{eq:SPDE:mild:formulation}.
We start with \eqref{eq:discrete:mild:formulation}. 
Thanks to the preliminary step, it may be rewritten  under the form:
\begin{equation*}
\begin{split}
&\int_{0}^t 
e^{(t-s)\Delta^{(N)}} 
\Bigl( 
\sum_{n \in \NN}
\bar e^{n,\pm}(\cdot) 
dW_{s}^{n,\pm}
\Bigr)
 = \sum_{n \in {\mathbb N}}
\int_{0}^t 
e^{-2 N^2 [1 - \cos(2\pi n/N)](t-s)}  \bar{e}^{n,\pm}(\cdot)
dW_{s}^{n,\pm}.
\end{split}
\end{equation*}
We then observe that there exists a universal constant $C$ such that 
\begin{equation*}
\sup_{x \in \SS^1}
\EE_{0}\biggl[ 
\biggl\vert
\sum_{n \geq {N}^{1/4}}
\int_{0}^t 
e^{-2 N^2 [1 - \cos(2\pi n/N)](t-s)}  \bar{e}^{n,\pm}(x)  
dW_{s}^{n,\pm}
\biggr\vert^2 \biggr] \leq \frac{C}{{N}^{1/4}}. 
\end{equation*}
Indeed, the left hand side is equal to 
\begin{equation}
\label{eq:approximation:terms:above:N14}
\begin{split}
&\sup_{x \in \SS^1} \EE_{0} \biggl[ \biggl\vert
\sum_{  n  \geq {N}^{1/4}}
\int_{0}^t 
e^{-2 N^2 [1 - \cos(2\pi n/N)](t-s)}  \bar{e}^{n,\pm}(x)  
dW_{s}^{n,\pm}
\biggr\vert^2 \biggr]
\\
&= \sup_{x \in \SS^1} \sum_{n  \geq {N}^{1/4}}
\int_{0}^t 
e^{-4 N^2 [1 - \cos(2\pi n/N)](t-s)}  \vert \bar{e}^{n,\pm}(x)  \vert^2 ds
\\
&= \sum_{n \geq {N}^{1/4}}
\frac{\sin^2(\pi n/N)}{(\pi n/N)^2}
\int_{0}^t 
e^{-4 N^2 [1 - \cos(2\pi n/N)](t-s)}  ds,
\end{split}
\end{equation}
so that 
\begin{equation}
\label{eq:approximation:terms:above:N14:3}
\begin{split}
&\sup_{x \in \SS^1}
\EE_{0} \biggl[ \biggl\vert
\sum_{n  \geq {N}^{1/4}}
\int_{0}^t 
e^{-2 N^2 [1 - \cos(2\pi n/N)](t-s)}  \bar{e}^{n,\pm}(x) 
dW_{s}^{n,\pm}
\biggr\vert^2 \biggr]
\\
&\leq \sum_{\vert n \vert \geq {N}^{1/4}}
\frac{\sin^2(\pi n/N)}{(\pi n/N)^2}
\frac1{4 N^2 [1 - \cos(2\pi n/N)]}.
\end{split}
\end{equation}
We then observe that the function 
$\varrho : \SS^1 \ni x \mapsto \varrho(x) = \sin^2(x)/(1- \cos(2x))$ 
is equal to $1/2$ as $\cos(2x) = 2 \cos^2(x) - 1 =
1 - 2 \sin^2(x)$. So, the above ratio 
$\sin^2(\pi n/N)/[1 - \cos(2\pi n/N)]$ is bounded
 by a universal constant $c$. 
In the sequel, this constant $c$ may vary from line to line as long as it remains universal. Then, 
\begin{equation}
\label{eq:approximation:terms:above:N14:2}
\begin{split}
&\sup_{x \in \SS^1} \EE_{0} \biggl[ \biggl\vert
\sum_{n \geq N^{1/4}}
\int_{0}^t 
e^{-2 N^2 [1 - \cos(2\pi n/N)](t-s)}  \bar{e}^{n,\pm}(x)  
dW_{s}^{n,\pm}
\biggr\vert^2 \biggr]
\leq c \sum_{n  \geq N^{1/4}} \frac1{n^2} 
\leq \frac{c}{{N}^{1/4}}. 
\end{split}
\end{equation} 
Actually, the same bound holds for the solution of the SPDE, namely:
\begin{equation*}
\sup_{x \in \SS^1}
{\mathbb E}_{0}
\biggl[ \biggl\vert
\int_{0}^t 
e^{(t-s)\Delta} 
\biggl( 
\sum_{n \geq {N}^{1/4}}
e^{n,\pm} (\cdot)
dW_{s}^{n,\pm}
\biggr)(x) 
\biggr\vert^2 \biggr] 
\leq \frac{c}{{N}^{1/4}}, 
\end{equation*}
which may be proved in the same way by recalling that 
$e^{(t-s)\Delta} 
e^{n,\pm} 
=
- (2 \pi n)^2 e^{n,\pm}$, 
for all $n \in \NN$. 

We now handle the difference
\begin{equation*}
\sum_{0 \leq n< {N}^{1/4}}
\biggl( \int_{0}^t 
e^{-2 N^2 [1 - \cos(2\pi n/N)](t-s)}  \bar{e}^{n,\pm}(x) 
dW_{s}^{n,\pm}
- \int_{0}^t 
e^{-(2 \pi n)^2 (t-s)} e^{n,\pm}(x) 
dW_{s}^{n,\pm}
\biggr). 
\end{equation*}
Taking the $L^2$ norm of the modulus, we obtain:
\begin{align}
&
\sup_{x \in \SS^1}
{\mathbb E}
\biggl[ \biggl\vert
\sum_{0 \leq n < {N}^{1/4}}
\biggl( \int_{0}^t 
e^{-2 N^2 [1 - \cos(2\pi n/N)](t-s)}  \bar{e}^{n,\pm}(x) 
dW_{s}^{n,\pm}
- \int_{0}^t 
e^{-(2 \pi n)^2 (t-s)} e^{n,\pm}(x) 
dW_{s}^{n,\pm}
\biggr) \biggr\vert^2
\biggr]\nonumber
\\
&=\sup_{x \in \SS^1} \sum_{0 \leq n < {N}^{1/4}}
 \int_{0}^t
\Bigl\vert 
e^{-2 N^2 [1 - \cos(2\pi n/N)](t-s)}  \bar{e}^{n,\pm}(x) 
-
e^{-(2\pi n)^2(t-s)}  {e}^{n,\pm}(x)  
\Bigr\vert^2 ds\nonumber
\\
&\leq 
4 \sum_{0 \leq n < {N}^{1/4}}
 \int_{0}^t
\Bigl\vert 
e^{-2 N^2 [1 - \cos(2\pi n/N)](t-s)}
-
e^{-(2\pi n)^2(t-s)}  
\Bigr\vert^2 ds\label{eq:second:step:approximation:(i)+(ii)}
\\
&\hspace{15pt}+
2	\sup_{x \in \SS^1}
\sum_{0 \leq n < {N}^{1/4}}
 \int_{0}^t
e^{-2(2\pi n)^2(t-s)} 
\bigl\vert  \bar{e}^{n,\pm}(x)  
-
 {e}^{n,\pm} (x)
\bigr\vert^2 ds \nonumber
\\
&= (i) + (ii). \nonumber
\end{align}
As for the first term $(i)$, we proceed as follows. We use the following two facts. First, we observe that, for $0 \leq n 
\leq {N}^{1/4}$,
\begin{equation*}
\begin{split}
N^2\bigl[1 - \cos\bigl(\frac{2 \pi n}{N} \bigr)
\bigr]
&= N^2 \bigl[ \frac12 \bigl( \frac{2 \pi n}{N}\bigr)^2 + O\bigl( \frac{n^4}{N^4} \bigr)\bigr]
= \frac{(2 \pi n)^2}{2} + O \bigl( \frac{1}{N} \bigr).
\end{split}
\end{equation*}
Therefore, for any
$0 \leq s \leq t \leq T$, 
\begin{equation*}
\exp \Bigl(
-2
N^2\bigl[1 - \cos\bigl(\frac{2 \pi n}{N} \bigr)
\bigr] (t-s)
\Bigr)
= \exp \bigl( 
- (2 \pi n)^2 (t-s)
\bigr) \bigl( 1+ O(\frac1N) \bigr),
\end{equation*}
where the Landau symbol is uniform in $s,t \in [0,T]$, with $s \leq t$, and in $0 \leq n \leq N^{1/4}$. 
Therefore,
\begin{equation*}
\vert (i) \vert   \leq
\frac{c}N \sum_{0 \leq n \leq N^{1/4}}
\int_{0}^t \exp\bigl(-(2 \pi n)^2 (t-s)\bigr) ds 
\leq \frac{c}N \Bigl( 1+ \sum_{n \in \NN^*} \frac1{n^2} \Bigr),
\end{equation*}
which is less   than $c/N$. 

In order to handle $(ii)$, we notice 
that, for $0 \leq n \leq N^{1/4}$, 
\begin{equation*}
\sup_{x \in \SS^1}   \vert e^{n,\pm}(x)- \bar{e}^{n,\pm}(x) \vert \leq \frac{c n}{N} \leq \frac{c}{N^{1/4}}.
\end{equation*}
We easily deduce that $\vert (ii)\vert$ is less than 
$c/N^{1/4}$.

So, the conclusion of this second step is that
there exists a sequence $(\varepsilon_{N})_{N \in \NN^*}$, independent of the data, converging to $0$ as $N$ tends to $\infty$, such that 
\begin{equation*}
\begin{split}
&\sup_{0 \le t \le T}
\sup_{x \in \SS^1}
{\mathbb E}_{0}
\biggl[ \biggl\vert
\int_{0}^t 
e^{(t-s)\Delta^{(N)}} 
\Bigl( 
\sum_{n \in {\mathbb N}}
\bar e^{n,\pm}(\cdot) 
dW_{s}^{n,\pm}
\Bigr)(x)
-
\int_{0}^t 
e^{(t-s)\Delta} 
\Bigl( 
\sum_{n \in {\mathbb N}}
e^{n,\pm}(\cdot) 
dW_{s}^{n,\pm}
\Bigr)(x)
 \biggr\vert^2
\biggr]
\\
&= \sup_{0 \le t \le T}
\sup_{x \in \SS^1}
{\mathbb E}_{0}
\biggl[ \biggl\vert
\sum_{n \in \NN}
\biggl( \int_{0}^t 
e^{-2 N^2 [1 - \cos(2\pi n/N)](t-s)}  \bar{e}^{n,\pm}(x)  
dW_{s}^{n,\pm}
- \int_{0}^t 
e^{-(2 \pi n)^2 (t-s)} e^{n,\pm}(x) 
dW_{s}^{n,\pm}
\biggr) \biggr\vert^2
\biggr]
\\
&\leq \varepsilon_{N},
\end{split}
\end{equation*}
which proves that the
two martingale terms in 
\eqref{eq:discrete:mild:formulation}
and
\eqref{eq:SPDE:mild:formulation}
get closer as $N$ tends to $\infty$, uniformly in time (and in the data).  
\vspace{4pt}

\textit{Third Step.} We now provide a similar analysis but for the control terms in
\eqref{eq:discrete:mild:formulation}
and
\eqref{eq:SPDE:mild:formulation}. 
 We start with the case when $\balpha(\cdot)=\bbeta(\cdot)$. 
 To do so, we call $(\alpha^{n,\pm}_{t})_{n \in \NN}$
the sequence of Fourier coefficients of 
each $\alpha_{t}(\cdot)$, seen as a (random) element of 
$L^2(\SS^1;\RR^d)$. 
Similar to
\eqref{se:5:approximation:beta},
we also define the sequence 
$((\bar{\alpha}^k_{t})_{0 \leq t \leq T})_{k=0,\cdots,N-1}$:
\begin{equation*}
\bar \alpha_{t}^k
=
N \int_{k/N}^{(k+1)/N} \alpha_{t}(x) dx, 
\quad t \in [0,T], \quad k \in \{0,\cdots,N-1\},  
\end{equation*} 
and we define $(\bar \alpha_{t}(\cdot))_{0 \leq t \leq T}$ accordingly, see 
\eqref{se:5:approximation:lambda}, namely
\begin{equation*}
\bar{\alpha}_{t}(\cdot) = \sum_{k=0}^{N-1} \bar{\alpha}_{t}^k {\mathbf 1}_{[k/N,(k+1)/N)}(\cdot).
\end{equation*}
With this notation, we have the following identity:
\begin{equation*}
\begin{split}
\bar{\alpha}_{t}(\cdot)
= 
\sum_{n \in \NN}
\alpha^{n,\pm}_{t}
\biggl[
\sum_{k=0}^{N-1}
\biggl( N \int_{k/N}^{(k+1)/N}
e^{n,\pm}(x) dx \biggr)
{\mathbf 1}_{[k/n,(k+1)/N)}(\cdot)
\biggr]
= \sum_{n \in \NN}
\alpha^{n,\pm}_{t} \bar{e}^{n,\pm}(\cdot).
\end{split}
\end{equation*}
So, using the preliminary step, we deduce that, for any $s,t \in [0,T]$ with $s \leq t$, 
\begin{equation*}
e^{(t-s) \Delta^{(N)}}
\bar{\alpha}_{s}(\cdot)
= 
e^{(t-s) \Delta^{(N)}}
\biggl[
\sum_{n \in \NN}
\alpha^{n,\pm}_{s} \bar{e}^{n,\pm}(\cdot)
\biggr]
= 
\sum_{n \in \NN}
\alpha^{n,\pm}_{s}  
e^{-2N^2[1- \cos(2\pi n/N)](t-s)}
\bar{e}^{n,\pm}(\cdot), 
\end{equation*}
and then
\begin{equation*}
\begin{split}
\int_{0}^t 
e^{(t-s) \Delta^{(N)}}
\bar{\alpha}_{s}(\cdot)
ds
&=
\sum_{n \in \NN}
\biggl( 
\int_{0}^t
\alpha^{n,\pm}_{s}  
e^{-2N^2[1- \cos(2\pi n/N)](t-s)}
ds 
\biggr) \bar{e}^{n,\pm}(\cdot).
\end{split}
\end{equation*}
Proceeding as in the second step, we first focus on 
\begin{equation*}
\begin{split}
\sum_{n  \geq N^{1/4}}
\biggl( 
\int_{0}^t
\alpha^{n,\pm}_{s}  
e^{-2N^2[1- \cos(2\pi n/N)](t-s)}
ds 
\biggr) \bar{e}^{n,\pm}(\cdot).
\end{split}
\end{equation*}
By Cauchy Schwartz inequality, 
we have
\begin{equation*}
\begin{split}
&
\sup_{x \in \SS^1}
\biggl\vert
\sum_{ n \geq N^{1/4}}
\biggl( 
\int_{0}^t
\alpha^{n,\pm}_{s}  
e^{-2N^2[1- \cos(2\pi n/N)](t-s)}
ds 
\biggr) \bar{e}^{n,\pm}(x)  \biggr\vert^2
\\
&\leq
\biggl( \sum_{n \geq N^{1/4}}
\int_{0}^t
\vert \alpha^{n,\pm}_{s} \vert^2 ds   \biggr)
\biggl( 
\sup_{x \in \SS^1} \sum_{n \geq N^{1/4}}
\int_{0}^t
e^{-4N^2[1- \cos(2\pi n/N)](t-s)}
\vert \bar{e}^{n,\pm}(x) \vert^2
ds 
\biggr). 
\end{split}
\end{equation*}
Take now expectation and deduce that:
\begin{equation*}
\begin{split}
&
\sup_{x \in \SS^1}
{\mathbb E}_{0}
\biggl[
\biggl\vert
\sum_{n \geq N^{1/4}}
\biggl( 
\int_{0}^t
\alpha^{n,\pm}_{s}  
e^{-2N^2[1- \cos(2\pi n/N)](t-s)}
ds 
\biggr) \bar{e}^{n,\pm}(x) \biggr\vert^2
\biggr]
\\
&\leq {\mathbb E}_{0}\biggl[ \sum_{ n  \geq N^{1/4}}
\int_{0}^t
\vert \alpha^{n,\pm}_{s} \vert^2 ds   \biggr]
\biggl( 
\sup_{x \in \SS^1}
 \sum_{n  \geq N^{1/4}}
\int_{0}^t
e^{-4N^2[1- \cos(2\pi n/N)](t-s)}
\vert \bar{e}^{n,\pm}(x) \vert^2
ds 
\biggr).
\end{split}
\end{equation*}
By Parseval's identity, the first term is bounded by 
$\EE_{0} \int_{0}^t \| \alpha_s(\cdot) \|_{L^2(\SS^1;\R^d)}^2
ds$. The second one may be handled as in 
\eqref{eq:approximation:terms:above:N14}
and 
\eqref{eq:approximation:terms:above:N14:2}.
We deduce that:
\begin{equation*}
\begin{split}
\sup_{x \in \SS^1}
{\mathbb E}_{0}
\biggl[
\biggl\vert
\sum_{n  \geq N^{1/4}}
\biggl( 
\int_{0}^t
\alpha^{n,\pm}_{s}  
e^{-2N^2[1- \cos(2\pi n/N)](t-s)}
ds 
\biggr) \bar{e}^{n,\pm}(x)  \biggr\vert^2
\biggr]
\leq \frac{c}{N^{1/4}} \EE_{0} \int_{0}^t \| \alpha_s(\cdot) \|_{L^2(\SS^1;\RR^d)}^2
ds.
\end{split}
\end{equation*}
Similarly, we have
\begin{equation*}
\int_{0}^t 
e^{(t-s) \Delta} \alpha_{s}(\cdot)
ds = \sum_{n \in \NN} \int_{0}^t 
\alpha^{n,\pm}_{s} e^{-(2 \pi n)^2(t-s)} e^{n,\pm}(\cdot) ds,
\end{equation*}
and then,
\begin{equation*}
\begin{split}
&\sup_{x \in \SS^1}\EE_{0} \biggl[ 
\biggl\vert
\sum_{n \geq N^{1/4}}
\int_{0}^t 
\alpha^{n,\pm}_{s} e^{-(2 \pi n)^2(t-s)} e^{n,\pm}(x)  ds
\biggr\vert^2 \biggr]
\\
&\leq \EE_{0}
\biggl[
\sum_{\vert n \vert \geq N^{1/4}}
\int_{0}^t 
\vert \alpha^{n,\pm}_{s} \vert^2 ds
\biggr]
\biggl( 
\sum_{n \geq N^{1/4}}
\int_{0}^t 
 e^{-2(2 \pi n)^2(t-s)}  ds
 \biggr),
\end{split}
\end{equation*}
and again, it is less than $(c/N^{1/4})
\EE_{0} \int_{0}^t \| \alpha_s(\cdot) \|_{L^2(\SS^1)}^2
ds$. 
We now handle the difference
\begin{equation*}
\begin{split}
&\sum_{0 \leq n < N^{1/4}}
\biggl( 
\int_{0}^t
\alpha^{n,\pm}_{s}  
e^{-2N^2[1- \cos(2\pi n/N)](t-s)}
ds 
\biggr) \bar{e}^{n,\pm}(\cdot)
-
\sum_{0 \leq n  < N^{1/4}}
\biggl(
\int_{0}^t 
\alpha^{n,\pm}_{s} e^{-(2 \pi n)^2(t-s)} ds \biggr) e^{n,\pm}(\cdot) 
\\
&= 
\sum_{\vert n \vert < N^{1/4}}
\int_{0}^t
\alpha^{n,\pm}_{s}  
\Bigl( 
e^{-2N^2[1- \cos(2\pi n/N)](t-s)}
 \bar{e}^{n,\pm}(\cdot)
-
 e^{-(2 \pi n)^2(t-s)} e^{n,\pm}(\cdot) \Bigr) ds.
\end{split}
\end{equation*}
By Cauchy-Schwarz inequality, 
\begin{equation*}
\begin{split}
&
\sup_{x \in \SS^1}
\EE_{0} \biggl[ \biggl\vert
\sum_{0 \leq  n  < N^{1/4}}
\int_{0}^t
\alpha^{n,\pm}_{s}  
e^{-2N^2[1- \cos(2\pi n/N)](t-s)}
 \bar{e}^{n,\pm}(x)
 ds 
\\
&\hspace{150pt} -
\sum_{0 \leq n < N^{1/4}}
\int_{0}^t 
\alpha^{n,\pm}_{s} e^{-(2 \pi n)^2(t-s)} e^{n,\pm}(x)  ds
\biggr\vert^2
\biggr]
\\
&\leq 
\EE_{0} \biggl[
\sum_{0 \leq n  < N^{1/4}}
\int_{0}^t
\vert 
\alpha^{n,\pm}_{s}
\vert^2
ds
\biggr]
\\
&\hspace{15pt} \times
\sup_{x \in \SS^1}
\sum_{0 \leq n < N^{1/4}}
\int_{0}^t
\vert 
e^{-2N^2[1- \cos(2\pi n/N)](t-s)} \bar e^{n,\pm}(x)
-
e^{-(2 \pi n)^2(t-s)} e^{n,\pm}(x)
\vert^2
 ds.
\end{split}
\end{equation*} 
We then follow 
\eqref{eq:second:step:approximation:(i)+(ii)}. We deduce that there exist a constant $C$ and 
a sequence $(\varepsilon_{N})_{N \in \NN^*}$, independent of the data, the sequence $(\varepsilon_{N})_{N \in \NN^*}$ converging to $0$ as $N$ tends to $\infty$, such that, for all 
$t \in [0,T]$, 
\begin{align}
\notag
&\sup_{x \in \SS^1}
\EE_{0}
\biggl[
\biggl\vert 
\sum_{n \in \NN}
\biggl( 
\int_{0}^t
\alpha^{n,\pm}_{s}  
e^{-2N^2[1- \cos(2\pi n/N)](t-s)}
ds 
\biggr) \bar{e}^{n,\pm}(x)
- 
\sum_{n \in \NN}
\biggl( 
\int_{0}^t
\alpha^{n,\pm}_{s}  
e^{-(2 \pi n)^2(t-s)}
ds 
\biggr) {e}^{n,\pm}(x)
\biggr\vert^2 
\biggr] 
\\
&\hspace{15pt} \leq C 
\biggl(
\EE_{0} \int_{0}^t
\| \alpha_{s}(\cdot) \|_{L^2(\SS^1;\R^d)}^2
ds
\biggr) \varepsilon_{N},
\label{eq:bd:by:epsilonN}
\end{align}
which proves in particular that, whenever 
$\balpha(\cdot) = \bbeta(\cdot)$, the control terms in
\eqref{eq:discrete:mild:formulation}
and
\eqref{eq:SPDE:mild:formulation}
get closer as $N$ tends to $\infty$, uniformly in time.  

Now, in order to handle the general case when 
$\balpha(\cdot) \not = \bbeta(\cdot)$, it suffices to handle the term:
\begin{equation*}
\sup_{x \in \SS^1}
\EE_{0}
\biggl[
\biggl\vert 
\sum_{n \in \NN}
\biggl( 
\int_{0}^t
\bigl(
\alpha^{n,\pm}_{s} 
- \beta^{n,\pm}_{s}\bigr) 
e^{-2N^2[1- \cos(2\pi n/N)](t-s)}
ds 
\biggr) \bar{e}^{n,\pm}(x) 
\biggr\vert^2
\biggr]. 
\end{equation*}
By Cauchy-Schwarz' inequality and then by Parseval's identity, it is less than 
\begin{equation*}
\begin{split}
&\sup_{x \in \SS^1}
\EE_{0}
\biggl[
\biggl\vert 
\sum_{n \in \NN}
\biggl( 
\int_{0}^t
\bigl(
\alpha^{n,\pm}_{s} 
- \beta^{n,\pm}_{s}\bigr) 
e^{-2N^2[1- \cos(2\pi n/N)](t-s)}
ds 
\biggr) \bar{e}^{n,\pm}(x)
\biggr\vert^2
\biggr]
\\
&\leq 
\sup_{x \in \SS^1}
\biggl\{
\EE_{0}
\biggl[
\sum_{n \in \NN}
\int_{0}^t
\vert
\alpha^{n,\pm}_{s} 
- \beta^{n,\pm}_{s}
\vert^2 
ds
\biggr]
\biggl[
\sum_{n \in \NN}
\int_{0}^t
e^{-4N^2[1- \cos(2\pi n/N)](t-s)}
\vert 
 \bar{e}^{n,\pm}(x) \vert^2 
 ds
\biggr] \biggr\}
\\
&\leq 
\EE_{0}
\biggl[
\int_{0}^t
\bigl\|
\bigl(\alpha_{s} 
- \beta_{s}\bigr)
(\cdot)
\bigr\|_{L^2(\SS^1)}^2 
ds
\biggr]
\biggl[
\sup_{x \in \SS^1}
\sum_{n \in \NN}
\int_{0}^t
e^{-4N^2[1- \cos(2\pi n/N)](t-s)}
\vert 
 \bar{e}^{n,\pm}(x) \vert^2 
 ds
\biggr].
\end{split}
\end{equation*}
Following 
\eqref{eq:approximation:terms:above:N14}
and
\eqref{eq:approximation:terms:above:N14:2}, 
we can easily bound the second factor. We deduce that 
\begin{equation*}
\begin{split}
&\sup_{x \in \SS^1}
\EE_{0}
\biggl[
\biggl\vert
\sum_{n \in \NN}
\biggl( 
\int_{0}^t
\bigl(
\alpha_{s}^{n,\pm} 
- \beta_{s}^{n,\pm} \bigr) 
e^{-2N^2[1- \cos(2\pi n/N)](t-s)}
ds 
\biggr) \bar{e}^{n,\pm}(x)
\biggr\vert^2
\biggr]
\\
&\hspace{15pt} \leq C
\EE_{0}
\biggl[
\int_{0}^t
\bigl\|
\bigl(\alpha_{s} 
- \beta_{s}\bigr)(\cdot)
\bigr\|_{L^2(\SS^1;\R^d)}^2 
ds
\biggr]. 
\end{split}
\end{equation*}
And then, combining with 
\eqref{eq:bd:by:epsilonN},
\begin{equation*}
\begin{split}
&\sup_{x \in \SS^1}
\EE_{0}
\biggl[
\sum_{n \in \NN}
\biggl( 
\int_{0}^t
\beta^{n,\pm}_{s}  
e^{-2N^2[1- \cos(2\pi n/N)](t-s)}
ds 
\biggr) \bar{e}^{n,\pm}(x) 
- 
\sum_{n \in \NN}
\biggl( 
\int_{0}^t
\alpha^{n,\pm}_{s}  
e^{-(2 \pi n)^2(t-s)}
ds 
\biggr) {e}^{n,\pm}(x) 
\biggr\vert^2 
\biggr] 
\\
&\hspace{15pt} \leq C 
 \varepsilon_{N}
\EE_{0} \int_{0}^t
\| \alpha_{s}(\cdot) \|_{L^2(\SS^1;\R^d)}^2
ds
+ 
C \EE_{0}
\biggl[
\int_{0}^t
\bigl\|
\bigl(\alpha_{s} 
- \beta_{s}\bigr)(\cdot)
\bigr\|_{L^2(\SS^1;\R^d)}^2 
ds
\biggr].
\end{split}
\end{equation*}

\vspace{5pt}

\textit{Fourth Step.}
We now handle the initial condition on the same principle. As before, we denote by $(X_{0}^{n,\pm})_{n \in \NN}$ the Fourier coefficients of $X_{0}(\cdot)$. Then, we let
\begin{equation*}
\bar X_{0}(\cdot) = 
\sum_{k = 0}^N
\biggl( N \int_{k/N}^{(k+1)/N} X_{0}(x)  dx \biggr)
{\mathbf 1}_{[k/N,(k+1)/N)}(\cdot) 
=
\sum_{n \in \NN} X^{n,\pm}_{0} \bar{e}^{n,\pm}(\cdot). 
\end{equation*}
Therefore, 
\begin{equation*}
e^{t \Delta^{(N)}} \bar X_{0}(\cdot) = 
\sum_{n \in \NN}
X^{n,\pm}_{0} e^{-2N^2[1- \cos(2\pi n/N)] t}
\bar{e}^{n,\pm}(\cdot). 
\end{equation*}
Proceeding as above, 
\begin{equation*}
\begin{split}
&\sup_{x \in \SS^1} \biggl\vert 
\sum_{ n  \geq N^{1/4}}
X^{n,\pm}_{0}
e^{-2N^2[1- \cos(2\pi n/N)] t}
\bar{e}^{n,\pm}(x) 
\biggr\vert^2
\\
&\hspace{15pt} \leq 
2 \sum_{n \in \NN}
\vert X^{n,\pm}_{0} \vert^2
\times 
\sum_{n  \geq N^{1/4}}
\frac{\sin^2(\pi n/N)}{(\pi n/N)^2}
e^{-2N^2[1- \cos(2\pi n/N)] t},
\end{split}
\end{equation*}
which yields to a somewhat different bound from what we obtained in the two previous steps. In order to recover the same kind of bounds, we use the following trick:
\begin{equation}
\label{eq:approximations:trick:condition:initiale}
\begin{split}
&\sup_{x \in \SS^1}
\biggl\vert 
\sum_{n  \geq N^{1/4}}
X^{n,\pm}_{0} 
e^{-2N^2[1- \cos(2\pi n/N)] t}
\bar{e}^{n,\pm}(x) 
\biggr\vert^2
\\
&\leq \frac1{t^{3/4}}
\sum_{n \in \NN}
\vert X^{n,\pm}_{0} \vert^2
\cdot
\sum_{n \geq N^{1/4}}
\frac{\sin^2(\pi n/N)}{(\pi n/N)^2}
t^{3/4} e^{-2N^2[1- \cos(2\pi n/N)] t}
\\
&\leq \frac{c}{t^{3/4}}
\sum_{n \in \NN}
\vert X^{n,\pm}_{0} \vert^2
\cdot
\sum_{n  \geq N^{1/4}}
\frac{\sin^2(\pi n/N)}{(\pi n/N)^2}
\frac{1}{(N^2[1- \cos(2\pi n/N)])^{3/4}},
\end{split}
\end{equation}
for a new value of the universal constant $c$. 
Recalling that the function $\RR \ni x \mapsto 
\sin(x)/x$ is bounded by $1$, we deduce that
\begin{equation*}
\begin{split}
&\sup_{x \in \SS^1}
\biggl\vert 
\sum_{n  \geq N^{1/4}}
X^{n,\pm}_{0} 
e^{-2N^2[1- \cos(2\pi n/N)] t}
\bar{e}^{n,\pm}(x) 
\biggr\vert^2
\\
&\leq \frac{c}{t^{3/4}}
\sum_{n \in \NN}
\vert X^{n,\pm}_{0} \vert^2
\cdot
\sum_{n \geq N^{1/4}}
\Bigl( 
\frac{\sin^2(\pi n/N)}{(\pi n/N)^2}
\Bigr)^{3/4}
\frac{1}{(N^2[1- \cos(2\pi n/N)])^{3/4}},
\end{split}
\end{equation*}
and then following the argument used to pass
from 
\eqref{eq:approximation:terms:above:N14:3}
to 
\eqref{eq:approximation:terms:above:N14:2}, 
we deduce that: 
\begin{equation*}
\begin{split}
&\sup_{x \in \SS^1}
\biggl\vert 
\sum_{n  \geq N^{1/4}}
X^n_{0} 
e^{-2N^2[1- \cos(2\pi n/N)] t}
\bar{e}^{n,\pm}(x) 
\biggr\vert^2
\leq \frac{c}{t^{3/4}}
\sum_{n \in \NN}
\vert X^{n,\pm}_{0} \vert^2
\cdot
\sum_{n \geq N^{1/4}}
\frac{1}{n^{3/2}}
\leq \frac{c}{t^{3/4} N^{1/8}}.
\end{split}
\end{equation*}
It is well-checked that a similar bound holds 
true for
\begin{equation*}
\begin{split}
\sup_{x \in \SS^1}
\biggl\vert 
\sum_{n \geq N^{1/4}}
X^{n,\pm}_{0} 
e^{-(2 \pi n)^2 t}
{e}^{n,\pm}(x) 
\biggr\vert^2.
\end{split}
\end{equation*}
So, in order to compare $e^{t \Delta}
X_{0}$ and $e^{t \Delta^{(N)}} \bar X_{0}$, 
see
\eqref{eq:discrete:mild:formulation}
and
\eqref{eq:SPDE:mild:formulation},
it remains to handle the difference
\begin{equation*}
\sum_{0 \leq n  < N^{1/4}}
\Bigl(
X^{n,\pm}_{0} 
e^{-N^2
[1- \cos(2\pi n/N)] t}
\bar{e}^{n,\pm}(\cdot)
-
X^{n,\pm}_{0}
e^{-(2 \pi n)^2 t}
{e}^{n,\pm}(\cdot) \Bigr).
\end{equation*}
By Cauchy-Schwarz inequality, we have the following bound. 
\begin{equation*}
\begin{split}
&\sup_{x \in \SS^1 }  \biggl\vert
\sum_{0 \leq n  < N^{1/4}}
X^{n,\pm}_{0} 
e^{-N^2
[1- \cos(2\pi n/N)] t}
\bar{e}^{n,\pm}(x) 
-
X^{n,\pm}_{0} 
e^{-(2 \pi n)^2 t}
{e}^{n,\pm}(x) 
\biggr\vert^2  
\\
&\leq
\biggl[
\sum_{n \in \NN}
\vert X^n_{0} \vert^2
\biggr]
\sup_{x \in \SS^1} 
\biggl[
\sum_{0\leq n < N^{1/4}}
\bigl\vert 
e^{-N^2
[1- \cos(2\pi n/N)] t}
\bar{e}^{n,\pm}(x) 
-
e^{-(2 \pi n)^2 t}
{e}^{n,\pm}(x) \bigr\vert^2
\biggr].
\end{split}
\end{equation*}
Following the analysis of 
\eqref{eq:second:step:approximation:(i)+(ii)}
and using the same trick as in 
\eqref{eq:approximations:trick:condition:initiale}, we deduce that there exist a constant $C$ and a sequence $(\varepsilon_{N})_{N \in \NN^*}$ converging to $0$ as $N$ tends to $\infty$, both the constant and the sequence being independent of the data, such that 
\begin{equation*}
\sup_{x \in \SS^1}
\bigl\vert
\bigl( e^{t \Delta^{(N)}} \bar X_{0} -  e^{t \Delta } X_{0}\bigr)(x) \bigr\vert^2 \leq 
\frac{\varepsilon_{N}}{t^{3/4}} \| X_{0}(\cdot) \|_{L^2(\SS^1;\RR^d)}^2.
\end{equation*}

\textit{Fifth Step.} By combining the three previous steps, we easily deduce 
\eqref{eq:discretization:2}.
\end{proof}

\subsubsection*{Proof of 
the two auxiliary identities 
\eqref{eq:preliminary:identity:1}
and
\eqref{eq:preliminary:identity:2}}

We now prove the identity
\eqref{eq:preliminary:identity:1}. We start with
\begin{equation*}
\begin{split}
\bar{e}^{n}(\cdot) &=
\sum_{k=0}^{N-1}
\biggl( N \int_{k/N}^{(k+1)/N} e^{i 2\pi n x} dx 
\biggr) {\mathbf 1}_{[k/N,(k+1)/N)}(\cdot)
\\
&=
\biggl( N \int_{0}^{1/N} 
e^{i 2\pi n x} dx 
\biggr)
\sum_{k=0}^{N-1}
e^{i 2\pi \frac{k n}{N}}
 {\mathbf 1}_{[k/N,(k+1)/N)}(\cdot)
 \\
&=
e^{i \pi \frac{n}{N}}
\frac{\sin(\pi n/N)}{\pi n/N}
\sum_{k=0}^{N-1}
e^{i 2\pi \frac{k n}{N}}
 {\mathbf 1}_{[k/N,(k+1)/N)}(\cdot).
\end{split}
\end{equation*}
We now check the second identity \eqref{eq:preliminary:identity:2}. Implementing the definition of $\Delta^{(N)}$, we get:
\begin{equation*}
\begin{split}
\Delta^{(N)}
\biggl[
\sum_{k=0}^{N-1}
e^{i 2\pi \frac{k n}{N}}
 {\mathbf 1}_{[k/N,(k+1)/N)}(\cdot)
\biggr]
&= 
N^2 \sum_{k=0}^{N-1}
\bigl( 
e^{i 2\pi \frac{(k+1) n}{N}}
+
e^{i 2\pi \frac{(k-1) n}{N}}
-
2
e^{i 2\pi \frac{k n}{N}}
\bigr)
 {\mathbf 1}_{[k/N,(k+1)/N)}(\cdot)
\\
&= 
- 2 N^2 \bigl[ 
1- \cos\bigl( \frac{2 \pi  n}{N}
\bigr) 
\bigr]
\biggl[
\sum_{k=0}^{N-1}
e^{i 2\pi \frac{k n}{N}}
 {\mathbf 1}_{[k/N,(k+1)/N)}(\cdot)
\biggr],
\end{split}
\end{equation*}

\subsection{Application to games}
We now return to 
 \eqref{eq:sec:5:definition:game} with 
 $(\balpha^{k,j} = \balpha^{\star k,j} = -  \bar{\boldsymbol Y}^{k})_{k=0,\cdots,N-1;j=1,\cdots,A_{N}}$ as defined in 
 \eqref{eq:sec:5:definition:game:Y}
 where 
  $({\boldsymbol X}(\cdot),{\boldsymbol Y}(\cdot),{\boldsymbol Z}(\cdot))$
  now denotes the solution
  to \eqref{eq:mkv:sde}.
  We denote the corresponding solution by
  $({\boldsymbol X}^{\star,k,j})_{k=0,\cdots,N-1;j=1,\cdots,A_{N}}$. Since 
$\balpha^{\star,k,j}$ does not depend on $j$, we have
${\boldsymbol X}^{\star,k,j} = 
\bar {\boldsymbol X}^{\star,k}$
for any 
$k \in \{0,\cdots,N-1\}$, with 
$\bar {\boldsymbol X}^{\star,k}
=1/A_{N} \sum_{j=1}^{A_{N}}
\bar {\boldsymbol X}^{\star,k,j}$. 

Also, we notice that $(\bar{X}_{t}^{\star,0},\cdots,
\bar{X}_{t}^{\star,N-1})_{0 \le t \le T}$ solves the system of SDEs:
\begin{equation*}
d \bar X_{t}^{\star,k} =
 \Bigl\{ b \bigl( \bar{\mu}_{t}^{\star,N} 
 \bigr) - \bar Y_{t}^k  + N^2
 \bigl( \bar X^{\star,k+1}_{t} + \bar X^{\star,k-1}_{t} - 2 \bar X^{\star,k}_{t}
 \bigr) 
 \Bigr\} dt + \sqrt N dB_{t}^{k},
 \end{equation*}
 for $t \in [0,T]$,  
with the same initial condition $\bar X_{0}^k$ as before and for $k \in \{0,\cdots,N-1\}$. 
The above system fits the form of 
\eqref{eq:discrete:eq}. To make it clear, we use the following notations:
\begin{equation*}
\bar{X}^\star_{t}(\cdot) = \sum_{k=0}^{N-1}
\bar X_{t}^{\star,k}
{\mathbf 1}_{[k/N,(k+1)/N)}(\cdot),
\quad \textrm{\rm and}
\quad 
\bar{\mu}^{\star,N}_{t} = 
\frac1N \sum_{k=0}^{N-1} \delta_{\bar{X}_{t}^{\star,k}},
\end{equation*} 
We then apply Theorem \ref{thm:discretization}
with $\balpha^\star(\cdot)
= b(\textrm{\rm Leb}_{\SS^1} \circ {\boldsymbol X}^{-1}(\cdot)) -{\boldsymbol Y}(\cdot)$
and $\bbeta^\star
= b(\bar{\boldsymbol \mu}^{\star,N}) - {\boldsymbol Y}(\cdot)$
and thus 
$(\bbeta^{\star,k}= b(\bar{\boldsymbol \mu}^{\star,N}) - \bar{\boldsymbol Y}^{k})_{k=0,\cdots,N-1}$. Then, the SPDE 
\eqref{eq:SPDE:alpha}
takes the form:
\begin{equation*}
\partial_{t} X_{t}(x) = b\bigl(\textrm{\rm Leb}_{\SS^1} \circ ({X}_{t}(\cdot))^{-1}\bigr)  - Y_{t} (x)  + \Delta X_{t}(x)  + \dot{W}_{t}(x), \quad (t,x) \in [0,T] 
\times \SS^1,
\end{equation*}
with the same $X_{0}(\cdot)$ as before as initial condition. 
Then, 
by Theorem 
\ref{thm:discretization}, we get for all $t \in (0,T]$, 
\begin{equation*}
\begin{split}
&
\sup_{x \in \SS^1} \EE_{0}
\bigl[
 \vert X_{t}(x) - \bar{X}_{t}^\star(x) \vert^2 \bigr] 
\\
&\leq
C \varepsilon_{N} \Bigl( 1 + \frac1{t^{3/4}}
 \Bigr)
 + 
C  \EE_{0} \int_{0}^t \bigl\vert 
b\bigl(\textrm{\rm Leb}_{\SS^1} \circ 
(X_{s}(\cdot))^{-1} \bigr)
- 
b(\bar \mu^{\star,N}_{s})
\bigr\vert^2 ds
\\
&\leq 
C \varepsilon_{N} \Bigl( 1 + \frac1{t^{3/4}}
 \Bigr)
+ 
C \EE_{0} \int_{0}^t \int_{\SS^1}
\vert  \bar{X}_{s}^\star(x) -X_{s}(x) 
\vert^2 dx ds
+ 
C\EE_{0} \int_{0}^t W_{2}\bigl(\bar{\mu}_{s}^{\star,N},
\textrm{\rm Leb}_{\SS^1} \circ 
(\bar X_{s}^\star(\cdot))^{-1}
\bigr)^2  ds,
\end{split} 
\end{equation*}
where $C$ now depends upon $\EE_{0} \int_{0}^T
\| Y_{s}(\cdot) \|^2_{L^2(\SS^1;\RR^d)}
ds$ and $\| X_{0}(\cdot)\|^2_{L^2(\SS^1;\RR^d)}$.
Since $\bar \mu^{\star,N}_{s}$ coincides with 
$\textrm{\rm Leb}_{\SS^1} \circ (\bar X^\star_{s}(\cdot))^{-1}$, the last term in the above inequality is 0. 
Therefore,
by the general version of Gronwall's lemma, we get, 
for any $t \in (0,T]$, 
\begin{equation}
\label{eq:sec:5:appli:thm24}
\begin{split}
\sup_{x \in \SS^1} 
\EE_{0}
\bigl[
\vert X_{t}(x) - \bar{X}_{t}^\star(x) \vert^2 \bigr] &\leq  C \varepsilon_{N} 
\Bigl( 1 + \frac{1}{t^{3/4}}
\Bigr).
\end{split}
\end{equation}
%
%
%
\vspace{5pt}

In order to show that we have constructed an approximate Nash equilibria, we apply a variant of the sufficiency proof in the Pontryagin principle. 


\subsubsection*{Particle system associated with ${\boldsymbol \beta}$}
Recall the definition of ${\boldsymbol \beta}$ from the introduction of Section 
\ref{se:construction:approximate:proof}:
Fix a pair $(k_{0},j_{0}) \in \{0,\cdots,N-1\} \times \{1,\cdots,A_{N}\}$
and let  
$\bbeta^{\star k,j} = \balpha^{\star k,j}$, 
for $k \in \{0,\cdots,N-1\}$ and 
$j \in \{1,\cdots,A_{N}\}$, 
with $(k,j) \not = (k_{0},j_{0})$; when $k=k_{0}$ and $j=j_{0}$, 
let 
$\bbeta^{\star k_{0},j_{0}} = {\boldsymbol \gamma}$,
for some
$\RR^d$-valued 
process ${\boldsymbol \gamma}=(\gamma_{t})_{0 \le t \le T}$ that is 
progressively-measurable 
with respect to 
 the filtration generated by the cylindrical white noise ${\boldsymbol W}(\cdot)=(W_{t}(\cdot))_{0 \le t \le T}$.
Then,
 we call
$((\chi_{t}^{k,j})_{0 \leq t \leq T})_{k=0,\cdots,N-1;j=1,\cdots,A_{N}}$
the system of particles:
 \begin{equation*}
 d \chi_{t}^{k,j} =
 \Bigl\{ b \bigl( \bar{\nu}_{t}^N 
 \bigr) + \beta_{t}^{k,j}
 + N^2
 \bigl( \bar \chi^{k+1}_{t} + \bar \chi^{k-1}_{t} - 2 \bar \chi^{k}_{t}
 \bigr) 
 \Bigr\} dt + \sqrt N dB_{t}^{k},
 \quad t \in [0,T],
 \end{equation*}
for $k \in \{0,\cdot,N-1\}$ and 
$j \in \{1,\cdots,A_{N}\}$,
with the initial condition
$\chi_{0}^{k,j} = \bar{X}^k_{0}$, and 
with 
\begin{equation*}
\bar{\nu}_{t}^N = 
\frac{1}{N A_{N}} \sum_{k=0}^{N-1}
\sum_{j=1}^{A_{N}}
\delta_{\chi_{t}^{k,j}},
\end{equation*}
and
\begin{equation*}
\bar \chi_{t}^{k} = \frac{1}{A_{N}}
\sum_{j=1}^{A_{N}} \chi_{t}^{k,j},
\quad t \in [0,T].
\end{equation*}
As usual, we let $\bar \chi_{t}(x) = \sum_{k=0}^{N-1}
\bar \chi_{t}^{k}
{\mathbf 1}_{[k/N,(k+1)/N)}(x)$.

\subsubsection*{Pontryagin principle}
For $(k_{0},j_{0})$ as above, we compute
\begin{equation*}
\begin{split}
&d 
\Bigl[
(\chi_{t}^{k_{0},j_{0}}  - \bar{X}_{t}^{\star,k_{0}}) \cdot \bar Y_{t}^{k_{0}}
\Bigr]
\\
&= \biggl[ 
\Bigl(
b\bigl(\bar \nu_{t}^N\bigr) - b\bigl(\bar \mu_{t}^{\star,N} \bigr) 
+
\beta_{t}^{k_{0},j_{0}} + \bar{Y}_{t}^{k_{0}}  
\\
&\hspace{30pt}
+
N^2 \bigl( \bar \chi_{t}^{k_{0}+1} + \bar \chi_{t}^{k_{0}-1} - 2 \bar \chi_{t}^{k_{0}}
- \bar X_{t}^{\star,k_{0}+1}
- \bar X_{t}^{\star,k_{0}-1}
+ 2 \bar X_{t}^{\star,k_{0}} 
\bigr)
\Bigr) 
\cdot
\bar Y_{t}^{k_{0}}  
\\
&\hspace{30pt}
- 
N \int_{k_{0}/N}^{(k_{0}+1)/N}
\Bigl[
\partial_{x} f \bigl(X_{t}(x),\textrm{\rm Leb}_{\SS^1} \circ (X_{t}(\cdot))^{-1} \bigr) 
\cdot
\bigl( \chi_{t}^{k_{0},j_{0}} - \bar X_{t}^{\star,k_{0}} \bigr) 
\Bigr]
dx
\biggr] dt
\\
&\hspace{15pt} 
  + dm_{t}, 
\end{split}
\end{equation*}
where $(m_{t})_{0 \le t \le T}$ is a square integrable martingale. 
Therefore,
\begin{equation}
\label{eq:convergence:pontryagin}
\begin{split}
&d 
\biggl[ \bigl[
(\chi_{t}^{k_{0},j_{0}}  - \bar{X}_{t}^{\star,k_{0}}) \cdot \bar Y_{t}^{k_{0}}
\bigr]
+ \int_{0}^t 
\Bigl(
f\bigl(\chi_{s}^{k_{0},j_{0}},\bar{\nu}^{N}_{s}
\bigr) - f\bigl(\bar X_{s}^{\star,k_{0}},\textrm{\rm Leb}_{\SS^1} \circ (\bar X_{s}^\star(\cdot))^{-1}
\bigr) \Bigr) ds
\\
&\hspace{30pt}
+ \frac12 
\Bigl( \int_{0}^t
 \vert \beta_{s}^{k_{0},j_{0}} \vert^2
 -
  \vert \bar Y_{s}^{k_{0}} \vert^2
 \Bigr) ds \biggr]
\\
&= \biggl[ 
\frac12 \bigl\vert 
\beta_{t}^{k_{0},j_{0}} + \bar{Y}_{t}^{k_{0}}  
\bigr\vert^2 + \delta^N_{t}
\\
&\hspace{15pt} 
+\Bigl[ f\bigl(\chi_{t}^{k_{0},j_{0}},
\textrm{\rm Leb}_{\SS^1} \circ 
(\bar X_{t}^\star(\cdot))^{-1}
 \bigr)
- 
 f\bigl( \bar X_{t}^{\star,k_{0}},\textrm{\rm Leb}_{\SS^1} \circ (\bar X_{t}^\star(\cdot))^{-1} \bigr)
\\
&\hspace{30pt} - 
\partial_{x} f \bigl(\bar X_{t}^{\star,k_{0}},\textrm{\rm Leb}_{\SS^1} \circ 
(\bar X_{t}^\star(\cdot))^{-1} \bigr)
\cdot 
\bigl( \chi_{t}^{k_{0},j_{0}} - \bar X_{t}^{\star,k_{0}} \bigr) 
\Bigr] \biggr] dt 
\\
&\hspace{15pt} + dm_{t},
\end{split}
\end{equation}
where we have let
\begin{equation*}
\begin{split}
\delta^N_{t}
&= f\bigl(\chi_{t}^{k_{0},j_{0}},\bar{\nu}^N_{t}
\bigr)
- 
f\bigl(\chi_{t}^{k_{0},j_{0}},
\textrm{\rm Leb}_{\SS^1} \circ 
(\bar X_{t}^\star(\cdot))^{-1}
\bigr)
\\
&\hspace{15pt} + \Bigl(
b\bigl(\bar \nu_{t}^N\bigr) - b\bigl(\bar \mu_{t}^{\star,N} \bigr) 
+
N^2 \bigl( \bar \chi_{t}^{k_{0}+1} + \bar \chi_{t}^{k_{0}-1} - 2 \bar \chi_{t}^{k_{0}}
- \bar X_{t}^{\star,k_{0}+1}
- \bar X_{t}^{\star,k_{0}-1}
+ 2 \bar X_{t}^{\star,k_{0}} 
\bigr)
\Bigr) 
\cdot
\bar Y_{t}^{k_{0}}  
\\
&\hspace{15pt}
- N \int_{k_{0}/N}^{(k_{0}+1)/N}
\Bigl[\Bigl(
\partial_{x} f \bigl(X_{t}(x),\textrm{\rm Leb}_{\SS^1} \circ ( X_{t}(\cdot))^{-1} \bigr) 
\\
&\hspace{150pt}
-
\partial_{x} f \bigl(\bar X_{t}^{\star,k_{0}},\textrm{\rm Leb}_{\SS^1} \circ 
(\bar X_{t}^\star(\cdot))^{-1} \bigr) 
\Bigr)
\cdot
\bigl(\chi_{t}^{k_{0},j_{0}} - \bar X_{t}^{\star,k_{0}} \bigr) 
\Bigr]
dx.
\end{split}
\end{equation*}
Hence, taking the expectation in 
\eqref{eq:convergence:pontryagin}, using the convexity of $f$ and inserting the terminal costs, we get:
\begin{equation*}
\begin{split}
&\EE_{0} \biggl[ g\bigl(\chi_{T}^{k_{0},j_{0}},\bar \nu_{T}^N \bigr)  +\int_{0}^T 
\Bigl( f\bigl(\chi_{t}^{k_{0},j_{0}},\bar \nu_{t}^N\bigr) 
+ \frac12 \vert \beta_{t}^{k_{0},j_{0}} \vert^2 
\Bigr) dt \biggr] 
\\
&\geq
\EE_{0} \biggl[ g\bigl( \bar X_{T}^{\star,k_{0}}, \textrm{\rm Leb}_{\SS^1} \circ (\bar X_{T}^\star(\cdot))^{-1} \bigr)  +\int_{0}^T 
\Bigl( f\bigl(\bar X_{t}^{\star,k_{0}},
 \textrm{\rm Leb}_{\SS^1} \circ (\bar X^\star_{t}(\cdot))^{-1}
\bigr) 
+ \frac12 \vert \bar Y_{t}^{\star,k_{0}} \vert^2 
\Bigr) dt \biggr] 
\\
&\hspace{15pt} + 
\EE_{0} \biggl[ 
 g\bigl(\chi_{T}^{k_{0},j_{0}},
  \textrm{\rm Leb}_{\SS^1} \circ (\bar X^\star_{T}(\cdot))^{-1} \bigr)
 -
  g\bigl( \bar X_{T}^{\star,k_{0}}, \textrm{\rm Leb}_{\SS^1} \circ (\bar X_{T}^\star(\cdot))^{-1} \bigr)
\\
&\hspace{155pt}   
  - \partial_{x} g \bigl( \bar X_{T}^{\star,k_{0}}, \textrm{\rm Leb}_{\SS^1} \circ (\bar X_{T}^\star(\cdot))^{-1} \bigr)
  \cdot
  \bigl( 
  \chi_{T}^{k_{0},j_{0}}
  - 
  \bar X_{T}^{\star,k_{0}}
  \bigr) 
    \biggr] 
\\
&\hspace{15pt}     
    + 
    \frac12 \EE_{0} \int_{0}^T
     \bigl\vert 
\beta_{t}^{k_{0},j_{0}} + \bar{Y}_{t}^{k_{0}}  
\bigr\vert^2
dt
    +    \EE_{0} \int_{0}^T \delta_{t}^N 
dt
    + \EE_{0} \delta_{N}',
\end{split}
\end{equation*}
where we have let
\begin{equation*}
\begin{split}
\delta_{N}' &= 
g\bigl(\chi_{T}^{k_{0},j_{0}},
\bar \nu^{N}_{T} \bigr)
- g\bigl(\chi_{T}^{k_{0},j_{0}},
  \textrm{\rm Leb}_{\SS^1} \circ 
  (\bar X_{T}^\star
  (\cdot))^{-1} \bigr) 
 \\ 
&\hspace{15pt}  
  - N \int_{k_{0}/N}^{(k_{0}+1)/N}
\Bigl[\Bigl(
\partial_{x} g \bigl(X_{T}(x),\textrm{\rm Leb}_{\SS^1} \circ (X_{T}(\cdot))^{-1} \bigr) 
\\
&\hspace{150pt}
-
\partial_{x} g \bigl(\bar X_{T}^{\star,k_{0}},\textrm{\rm Leb}_{\SS^1} \circ 
(\bar X_{T}^\star(\cdot))^{-1} \bigr) 
\Bigr)
\cdot
\bigl( \chi_{T}^{k_{0},j_{0}} - \bar X_{T}^{\star,k_{0}} \bigr) 
\Bigr]
dx.
\end{split}
\end{equation*}
By convexity of $g$
and from the identity $(\bar \mu_{t}^{\star,N} = \textrm{\rm Leb}_{{\mathbb S}^1} \circ 
(\bar X_{t}^\star(\cdot))^{-1})_{0 \leq t \leq T}$, we end-up with:
\begin{equation}
\label{eq:conclusion:pontryagin}
\begin{split}
J^{k_{0},j_{0}}\bigl((\bbeta^{k,j})_{k=0,\cdots,N-1;j=1,\cdots,A_{N}}\bigr)
&\geq 
J^{k_{0},j_{0}}\bigl((\balpha^{\star k,j})_{k=0,\cdots,N-1;j=1,\cdots,A_{N}}\bigr)
\\
&\hspace{15pt}
+
    \frac12 \EE_{0} \int_{0}^T
     \bigl\vert 
\beta_{t}^{k_{0},j_{0}} + \bar{Y}_{t}^{k_{0}}  
\bigr\vert^2
dt
 + \EE_{0} \int_{0}^T \delta_{t}^N dt  + \EE_{0}\delta_{N}'.
\end{split}
\end{equation}

\subsubsection*{Proving the convergence of the remainder}
We now investigate the two sequences
$( \delta_{N}')_{N \ge 1}$ and 
$( \int_{0}^T \delta^N_{t} dt )_{N \in \NN^*}$. 
Using once again 
the
identity $(\bar \mu_{t}^{\star,N} = \textrm{\rm Leb}_{{\mathbb S}^1} \circ 
(\bar X_{t}^\star(\cdot))^{-1})_{0 \leq t \leq T}$
together with the regularity properties of the coefficients, 
we have
\begin{equation*}
\begin{split}
&\EE_{0} \bigl[ \vert \delta_{N}'
\vert \bigr] + \EE_{0} \int_{0}^T \vert \delta^N_{t} \vert
dt 
\\
&\leq C 
\biggl(
\EE_{0} \bigl[ W_{2}\bigl(\bar \nu_{T}^N,
\textrm{\rm Leb}_{\SS^1} \circ 
(\bar X_{T}^\star(\cdot))^{-1}\bigr) \bigr]
+ \int_{0}^T
\EE_{0} \bigl[ W_{2}\bigl(\bar \nu_{t}^N,
\textrm{\rm Leb}_{\SS^1} \circ (\bar X_{t}^\star(\cdot))^{-1}\bigr) \bigr]
dt 
\biggr)
\\
&\hspace{15pt} + C \sup_{x \in \SS^1}
\EE_{0} \bigl[ \vert X_{T}(x) - \bar X_{T}^\star(x) 
\vert^2 \bigr]^{1/2}
\Bigl( 
1+
\EE_{0} \bigl[ \vert \chi_{T}^{k_{0},j_{0}}
- 
\bar X_{T}^{\star,k_{0}}
\vert^2 \bigr]^{1/2}  \Bigr) \phantom{\biggr)}
\\
&\hspace{15pt}
+ C 
\biggl(
\int_{0}^T 
 \sup_{x \in \SS^1}
\EE_{0} \bigl[ \vert X_{t}(x) - \bar X_{t}^\star(x) 
\vert^2 \bigr]
dt \biggr)^{1/2}
\biggl[
1+
\biggl( \int_{0}^T
\EE_{0} \bigl[ \vert \chi_{t}^{k_{0},j_{0}}
- 
\bar X_{t}^{\star,k_{0}}
\vert^2 \bigr] dt \biggr)^{1/2}
\biggr]
\\
&\hspace{15pt}
+  C \, 
\EE_{0}
\int_{0}^T 
 \sup_{x \in \SS^1}
 \bigl\vert 
 \Delta^{(N)}(  \bar \chi_{t} - \bar X_{t}^\star )(x)
\bigr\vert 
dt,
\end{split}
\end{equation*}
where, in the last line, we used the fact that the process
$(\bar Y_{t}^{k_{0}})_{0 \leq t \leq T}$
is bounded independently of $k_{0}$, see for instance 
Lemma \ref{lem:bound:pp}. 

Observe from 
\eqref{eq:sec:5:appli:thm24}
 that 
we can find a sequence $(\varepsilon_{N})_{N \in \NN^*}$, converging to $0$ as $N$ tends to $\infty$, such that
\begin{equation*}
\sup_{x \in \SS^1}
\EE_{0} \bigl[ \vert X_{T}(x) - \bar X_{T}^\star(x) 
\vert^2 \bigr]^{1/2}
+
\biggl(
\int_{0}^T 
 \sup_{x \in \SS^1}
\EE_{0} \bigl[ \vert X_{t}(x) - \bar X_{t}^\star(x) 
\vert^2 \bigr]
dt \biggr)^{1/2}
\leq \varepsilon_{N}. 
\end{equation*}
Now, for any $t \in [0,T]$,
\begin{equation*}
W_{2}\bigl(\bar \nu_{t}^N,
\textrm{\rm Leb}_{\SS^1} \circ 
(\bar X_{t}^\star(\cdot))^{-1}\bigr) 
\leq 
\biggl(
\frac1N \frac1{A_{N}}
\sum_{k=0}^{N-1}
\sum_{j=1}^N \vert \chi_{t}^{k,j} - \bar{X}_{t}^{\star,k} \vert^2 \biggr)^{1/2}. 
\end{equation*}
So, we end up with:
\begin{equation}
\label{eq:bound:deltaN:deltaN'}
\begin{split}
&\EE \bigl[\vert \delta_{N}'
\vert
\bigr]  + \EE \int_{0}^T \vert \delta^N_{t} \vert
dt 
\\
&\hspace{15pt}\leq \varepsilon_{N}
\Bigl( 1+ 
\sup_{0 \leq t \leq T}
\EE \bigl[ \vert \chi_{t}^{k_{0},j_{0}}  
- 
\bar{X}_{t}^{\star,k_{0}}
\vert^2
\bigr]^{1/2}
\Bigr)
\\
&\hspace{30pt}
 + 
C
\sup_{0 \leq t \leq T}
\EE_{0} \biggl[
\frac1N \frac1{A_{N}}
\sum_{k=0}^{N-1}
\sum_{j=1}^N 
\vert \chi_{t}^{k,j} - \bar{X}_{t}^{\star,k} \vert^2
\biggr]^{1/2}
+  C \, 
\EE_{0}
\int_{0}^T 
 \sup_{x \in \SS^1}
 \bigl\vert 
 \Delta^{(N)}(  \bar \chi_{t} - \bar X_{t}^\star )(x)
\bigr\vert 
dt.
\end{split}
\end{equation}
Now, for any $t \in [0,T]$, 
\begin{equation*}
\begin{split}
\frac1N \frac1{A_{N}}
\sum_{k=0}^{N-1}
\sum_{j=1}^N \vert \chi_{t}^{k,j} - \bar{X}_{t}^{\star,k} \vert^2
&\leq C \int_{0}^t 
\frac1N \frac1{A_{N}}
\sum_{k=0}^{N-1}
\sum_{j=1}^N \vert \chi_{s}^{k,j} - \bar{X}_{s}^{\star,k} \vert^2 ds 
\\
&\hspace{15pt} + 
\frac{C}{N A_{N}}
\int_{0}^T \vert \gamma_{s} + \bar Y_{s}^{k_{0}} \vert^2 ds 
+ 
C \biggl( \int_{0}^T 
\sup_{x \in \SS^1}
\bigl\vert \Delta^{(N)} \bigl( \bar X_{s}^\star - \bar \chi_{s} \bigr) (x) \bigr\vert ds \biggr)^2,
\end{split}
\end{equation*}
so that, by Gronwall's lemma,
\begin{equation}
\label{eq:stability:barmu:barnu}
\begin{split}
&\sup_{0 \le t \le T}
\frac1N \frac1{A_{N}}
\sum_{k=0}^{N-1}
\sum_{j=1}^N \vert \chi_{t}^{k,j} - \bar{X}_{t}^{\star,k} \vert^2
\\
&\hspace{15pt} \leq  
\frac{C}{N A_{N}}
\int_{0}^T \vert \gamma_{s}+ \bar Y_{s}^{k_{0}} \vert^2 ds 
+ C 
\biggl( \int_{0}^T 
\sup_{x \in \SS^1}
\bigl\vert \Delta^{(N)} \bigl( \bar X_{s}^\star - \bar \chi_{s} \bigr) (x) \bigr\vert ds \biggr)^2.
\end{split}
\end{equation}
We then claim from Proposition 
\ref{stability:interaction} below that there exists a constant $c$, only depending on $T$, such that
\begin{equation*}
\EE\biggl[
\biggl(
\int_{0}^T
\sup_{x \in \SS^1}
\bigl\vert \Delta^{(N)} (\bar X_{t}^\star - \bar \chi_{t})(x) 
\bigr\vert 
dt
\biggr)^2 
\biggr]
\leq \frac{c}{A_{N}^2}\EE_{0}\int_{0}^T \vert \gamma_{t} + \bar Y_{t}^{k_{0}}\vert^2 dt, 
\end{equation*} 
from which we deduce that 
\begin{equation*}
\begin{split}
\sup_{0 \le t \le T}
\EE_{0} \biggl[
\biggl(
\frac1N \frac1{A_{N}}
\sum_{k=0}^{N-1}
\sum_{j=1}^N \vert \chi_{t}^{k,j} - \bar{X}_{t}^{\star,k} \vert
\biggr)^2
\biggr]
& \leq \frac{C}{\min(N,A_{N})^2} 
\EE_{0} \int_{0}^T \vert \gamma_{t}
+ \bar Y_{t}^{k_{0}}
 \vert^2 dt,
\end{split}
\end{equation*}
the constant $C$ being allowed to vary from line to line. 
By a similar argument, but without averaging, we obtain
\begin{equation*}
\begin{split}
\sup_{0 \le t \le T}
\EE_{0} \Bigl[ \vert \chi_{t}^{k_{0},j_{0}} - \bar{X}_{t}^{\star,k_{0}} \vert^2
\Bigr]
& \leq C 
\EE_{0} \int_{0}^T \vert \gamma_{t}
+ \bar Y_{t}^{k_{0}}
 \vert^2 dt. 
\end{split}
\end{equation*}
Returning to 
\eqref{eq:bound:deltaN:deltaN'}, this yields to
\begin{equation*}
\EE_{0} \bigl[
\vert \delta_{N}'
\vert \bigr] + \EE_{0} \int_{0}^T \vert \delta^N_{t} \vert
dt 
\leq \varepsilon_{N} + C 
\varepsilon_{N}
\biggl( 
\EE_{0} \int_{0}^T \vert \gamma_{t}
+ \bar Y_{t}^{k_{0}} \vert^2 dt
 \biggr)^{1/2},
\end{equation*}
and then, inserting into 
\eqref{eq:conclusion:pontryagin}, we get:
\begin{align}
\notag
J^{k_{0},j_{0}}\bigl((\bbeta^{k,j})_{k=0,\cdots,N-1;j=1,\cdots,A_{N}}\bigr)
&\geq 
J^{k_{0},j_{0}}\bigl((\balpha^{\star k,j})_{k=0,\cdots,N-1;j=1,\cdots,A_{N}}\bigr)
+
    \frac12 \EE_{0} \int_{0}^T
     \bigl\vert \gamma_{t} + \bar{Y}_{t}^{k_{0}}  
\bigr\vert^2
dt
\\
&\hspace{15pt}
-  
\varepsilon_{N}
\biggl[ 1+  
\biggl( 
\EE_{0} \int_{0}^T \vert \gamma_{s}
+ \bar Y_{s}^{k_{0}}
 \vert^2 ds
 \biggr)^{1/2} \biggr],
 \label{eq:conclusion:pontryagin:2}
\end{align}
the sequence $(\varepsilon_{N})_{N \in \NN^*}$ being now allowed to depend upon 
$(A_{N})_{N \in \NN^*}$. 

Obviously, we can a find constant $a>0$, independent of $N$, such that the sum of the last two terms in the right-hand side is positive whenever 
$\EE_{0} \int_{0}^T \vert \gamma_{s} \vert^2 ds$
is greater than $a$. In such a case, we have 
\begin{equation*}
J^{k_{0},j_{0}}\bigl((\bbeta^{k,j})_{k=0,\cdots,N-1;j=1,\cdots,A_{N}}\bigr)
\geq 
J^{k_{0},j_{0}}\bigl((\balpha^{\star k,j})_{k=0,\cdots,N-1;j=1,\cdots,A_{N}}\bigr),
\end{equation*}
which is the required inequality. 

Now, if 
$\EE_{0} \int_{0}^T \vert \gamma_{s} \vert^2 ds \leq a$, \eqref{eq:conclusion:pontryagin:2} yields
\begin{equation*}
\begin{split}
J^{k_{0},j_{0}}\bigl((\bbeta^{k,j})_{k=0,\cdots,N-1;j=1,\cdots,A_{N}}\bigr)
&\geq 
J^{k_{0},j_{0}}\bigl((\balpha^{\star k,j})_{k=0,\cdots,N-1;j=1,\cdots,A_{N}}\bigr)
-  
\varepsilon_{N} \bigl( 1+ a^{1/2} \bigr),
\end{split}
\end{equation*}
and the result follows easily.

\subsubsection*{Stability of the interaction}
In order to complete the proof, it remains to evaluate the distance between
$
\Delta^{(N)} \bar{\boldsymbol X}^\star(\cdot)$
and
$\Delta^{(N)} \bar{\boldsymbol \chi}(\cdot)$, 
which is the purpose of the next statement.
\begin{proposition}
\label{stability:interaction}
There exists a constant $C$, only depending on $T$, such that, 
with the same notations as before,
\begin{equation*}
\EE_{0} \biggl[
\biggl(
\int_{0}^T
 \sup_{x \in \SS^1}
\bigl\vert \Delta^{(N)} (\bar X_{t}^\star - \bar \chi_{t}) 
(x)\bigr\vert
dt
 \biggr)^2
\biggr] \leq \frac{C}{A_{N}^2}  \EE_{0} \int_{0}^T \vert \gamma_{t} 
+ \bar Y_{t}^{k_{0}}
\vert^2 dt. 
\end{equation*} 
\end{proposition}

\begin{proof}
By 
\eqref{eq:SPDE:mild:formulation}, we notice that, for any $t \in [0,T]$, 
\begin{equation}
\label{eq:proof:stability:interaction}
\begin{split}
\bar{X}_{t}(\cdot)
- 
\bar{\chi}_{t}(\cdot)
&=
\int_{0}^t 
e^{(t-s) \Delta^{(N)}} \bigl( b(\bar \mu_{s}^{\star,N} ) 
- b( \bar \nu_{s}^{N}) 
\bigr) ds + 
\int_{0}^t
e^{(t-s) \Delta^{(N)}} \bigl( 
\bar{\alpha}_{s}^\star(\cdot)
- 
\bar{\beta}_{s}(\cdot)
\bigr) 
ds
\\
&= (i)_{t} + (ii)_{t}, 
\end{split}
\end{equation}
where we have let
\begin{equation*}
\bar \alpha_{t}^\star(\cdot) = \sum_{k=0}^{N-1}
\alpha_{t}^\star {\mathbf 1}_{[k/N,(k+1)/N)}(\cdot), 
\quad 
\bar \beta_{t}(\cdot) = \sum_{k=0}^{N-1}
\biggl( \frac1{A_{N}}
\sum_{j=1}^{A_{N}}
\beta_{t}^{k,j}
\biggr) {\mathbf 1}_{[k/N,(k+1)/N)}(\cdot).
\end{equation*}
As for $(i)$, using the fact that 
both $b(\bar \mu^{\star,N}_{s})$ and $b(\bar\nu_{s}^{N})$ are constant functions of $x \in \SS^1$ for each $s \in [0,T]$, 
it is absolutely obvious that 
\begin{equation*}
\begin{split}
&(i)_{t} = \int_{0}^t 
\bigl( b(\bar \mu_{s}^{\star,N} ) 
- b( \bar \nu_{s}^{N}) 
\bigr) ds, 
\end{split}
\end{equation*}
and then
$\Delta^{(N)} (i)_{t} = 0$. 
Returning to 
\eqref{eq:proof:stability:interaction},
it suffices to focus on $(ii)_{t}$.

Letting $(\bar \varrho_{t}(\cdot) = (\bar \alpha_{t}^\star - 
\bar \beta_{t})(\cdot))_{0 \le t \le T}$
and following the third step in the proof of Theorem 
\ref{thm:discretization}, we have
\begin{equation*}
\begin{split}
\int_{0}^t e^{(t-s) \Delta^{(N)}} \bigl( 
\bar{\alpha}_{s}^\star
- 
\bar{\beta}_{s}
\bigr)(\cdot) ds
&= \sum_{n \in \NN}
\biggl( \int_{0}^t
\bar \varrho^{n,\pm}_{s} e^{-2N^2[1 - \cos(2 \pi n/N)](t-s)}
ds \biggr) \bar{e}^{n,\pm}(\cdot), 
\end{split}
\end{equation*}
where $(\bar \varrho^{n,\pm}_{s})_{n \in \NN}$ denote the 
Fourier coefficients of the function $\bar \varrho_{s} \in 
L^2(\SS^1;\R^d)$. Here we used the identity
\begin{equation*}
\bar \varrho_{s}(\cdot) = \sum_{n \in \NN}
\sum_{k=0}^{N-1}
\biggl( N \int_{k/N}^{(k+1)/N}
\bar \varrho_{s}(x) dx
\biggr)
{\mathbf 1}_{[k/N,(k+1)/N)}(\cdot)
= 
\sum_{n \in \NN} \bar \varrho^{n,\pm}_{s}
\bar{e}^{n,\pm}(\cdot),
\end{equation*}
which follows from the fact that $\bar \varrho_{s}(\cdot)$ is constant on each $[k/N,(k+1)/N)$. Then,
\begin{equation*}
\Delta^{(N)} (ii)_{t} = \sum_{n \in \NN}
\biggl(
\int_{0}^t
 \bar \varrho^{n,\pm}_{s} 
\bigl( -
 2N^2[1 - \cos(2 \pi n/N)] 
 \bigr)
 e^{-2N^2[1 - \cos(2 \pi n/N)](t-s)}
ds \biggr) \bar{e}^{n,\pm}(\cdot). 
\end{equation*}
We deduce that 
\begin{equation*}
\sup_{x \in \SS^1}
\bigl\vert \Delta^{(N)} (ii)_{t}(x)
\bigr\vert 
\leq 2 \sum_{n \in \NN}
\int_{0}^t
\vert \bar \varrho^{n,\pm}_{s}\vert
\frac{\vert \sin(\pi n/N)\vert}{\pi n/N}
 \bigl( 2N^2[1 - \cos(2 \pi n/N)]
 \bigr) 
 e^{-2N^2[1 - \cos(2 \pi n/N)](t-s)}
ds,
\end{equation*}
which we rewrite
\begin{equation*}
\sup_{x \in \SS^1}
\bigl\vert \Delta^{(N)} (ii)_{t}(x)
\bigr\vert 
\leq \sum_{n \in \NN}
\int_{0}^t
\bar \varrho^{n,\pm}_{s} \cdot h^{n,\pm}_{s} ds,
\end{equation*}
where we have let
\begin{equation*}
\begin{split}
h^{n,\pm}_{s} = \textrm{\rm sign}\bigl(\bar \varrho^{n,\pm}_{s}
  \bigr) 
\bigl( 2N^2[1 - \cos(2 \pi n/N)] \bigr) \frac{\vert \sin(\pi n/N) \vert}{\pi n/N}
 e^{-2N^2[1 - \cos(2 \pi n/N)](t-s)},
\end{split}
\end{equation*}
where $\textrm{\rm sign}(x)$ is understood as 
$(\textrm{\rm sign}(x_{1}),\cdots,\textrm{\rm sign}(x_{d}))$
for $x \in \RR^d$. 
Obviously, 
\begin{equation*}
\sup_{0 \leq s \leq t}
\sum_{n \in \NN} \vert h^{n,\pm}_{s} \vert^2 < \infty,
\end{equation*}
so that, by Parseval's identity, 
\begin{equation*}
\sup_{x \in \SS^1}
\bigl\vert \Delta^{(N)} (ii)_{t}(x)
\bigr\vert 
\leq \int_{0}^t  \int_{\SS^1}
\bar \varrho_{s}(x) \cdot h_{s}(x) dx,
\end{equation*}
with 
\begin{equation*}
h_{s}(\cdot) = \sum_{n \in \NN} h_{s}^{n,\pm} e^{n,\pm}(\cdot). 
\end{equation*}
In fact $\bar \varrho_{s}(\cdot) = [(\gamma_{s} + \bar Y_{s}^{k_{0}})/A_{N}] {\mathbf 1}_{[k_{0}/N,(k_{0}+1)/N)}(\cdot)$. Hence, we have
\begin{equation*}
\sup_{x \in \SS^1}
\bigl\vert \Delta^{(N)} (ii)_{t}(x)
\bigr\vert 
\leq
\frac1{A_{N}}
\biggl\vert
\int_{0}^t  
\bigl( \gamma_{s} + \bar Y_{s}^{k_{0}}
\bigr)
\cdot
\biggl( \int_{k_{0}/N}^{(k_{0}+1)/N}
h_{s}(x) dx
\biggr) ds
\biggr\vert .
\end{equation*}
Clearly, by 
\eqref{eq:preliminary:identity:1},
\begin{equation*}
\begin{split}
\biggl\vert \int_{k_{0}/N}^{(k_{0}+1)/N} h_{s}(x) dx 
\biggr\vert
&= 
\biggl\vert \sum_{n \in \NN} h^{n,\pm}_{s} \int_{k_{0}/N}^{(k_{0}+1)/N} e^{n,\pm}(x) dx
\biggr\vert
\leq 2 \sum_{n \in \NN} \vert h^{n,\pm}_{s} \vert \frac{\vert \sin(\pi n/N) \vert}{\pi n}. 
\end{split}
\end{equation*}
Then,
\begin{equation*}
\begin{split}
&\int_{0}^T
\sup_{x \in \SS^1}
\bigl\vert \Delta^{(N)} (ii)_{t}(x) 
\bigr\vert dt
\\
&\leq \frac4{A_{N}} \sum_{n \in \NN}
\int_{0}^T \int_{0}^t \vert \gamma_{s} +  \bar Y_{s}^{k_{0}} \vert 
\frac{\sin^2(\pi n/N)}{(\pi n)^2}
\bigl(
N^3[1-\cos(2\pi n/N)]
\bigr) e^{-2 N^2[1- \cos(2 \pi n/N)](t-s)}
ds \, dt
\\
&= \frac4{A_{N}} \sum_{n \in \NN}
\int_{0}^T  \vert \gamma_{s} + \bar Y_{s}^{k_{0}} \vert 
\frac{\sin^2(\pi n/N)}{(\pi n)^2}
\bigl(
N^3[1-\cos(2\pi n/N)]
\bigr)
\biggl( \int_{s}^T
 e^{-2 N^2[1- \cos(2 \pi n/N)](t-s)}
dt\biggr)  ds. 
\end{split}
\end{equation*}
We thus have
\begin{equation*}
\begin{split}
\int_{0}^T
\sup_{x \in \SS^1}
\bigl\vert \Delta^{(N)} (ii)_{t}(x) 
\bigr\vert dt
&\leq 
\frac2{A_{N}}
\biggl( \int_{0}^T
 \vert \gamma_{s} +   \bar Y_{s}^{k_{0}} \vert
 ds \biggr)
 \biggl( 
N \sum_{n \in \NN}
\frac{\sin^2(\pi n/N)}{(\pi n)^2}
\biggr)
\\
&\leq \frac2{A_{N}}
 \biggl( 
\int_{0}^T \vert \gamma_{s} +  \bar Y_{s}^{k_{0}} \vert ds \biggr) 
\biggl( 
\frac1N
\sum_{n=0}^N
\frac{\sin^2(\pi n/N)}{(\pi n/N)^2}
+
 N \sum_{n > N} 
\frac{1}{n^2}
\biggr). 
\end{split}
\end{equation*}
So, there exists a constant $C$, only depending on $T$, such that 
\begin{equation*}
\begin{split}
&{\mathbb E}_{0}
\biggl[
\biggl\vert
\int_{0}^T
\sup_{x \in \SS^1}
\bigl\vert \Delta^{(N)} (ii)_{t}(x) 
\bigr\vert dt
\biggr\vert^2
\biggr] 
\leq 
 \frac{C}{A_{N}^2}
 \EE_{0}  
\int_{0}^T \vert \gamma_{t} + \bar Y_{t}^{k_{0}} \vert^2 dt, 
\end{split}
\end{equation*}
which completes the proof.
\end{proof}

\bibliographystyle{plain}
\bibliography{restore}

\end{document}